\documentclass[11pt,oneside,reqno]{amsart}
\usepackage{fullpage}
\usepackage{amsmath,ifthen, amsfonts, amssymb,
srcltx, amsopn, color, enumerate} 
\usepackage[linktocpage=true]{hyperref}
\usepackage[cmtip,arrow]{xy}
\usepackage{pb-diagram, pb-xy}
\usepackage{overpic}
\usepackage{microtype}
\usepackage{mathabx}
\usepackage{tikz}
\usetikzlibrary{decorations.markings,cd}
\usetikzlibrary{decorations.pathreplacing}
\usetikzlibrary{arrows}
\usetikzlibrary{shapes.misc}
\usepackage{pifont}
\usepackage[normalem]{ulem}
\usepackage{graphicx}
\graphicspath{ }

\newenvironment{subproof}[1][\proofname]{%
	\begin{proof}[#1]%
	}{%
	\end{proof}%
}

\newtheorem{thm}{Theorem}[section]
\newtheorem{lem}[thm]{Lemma}

\newtheorem{cor}[thm]{Corollary}

\newtheorem{prop}[thm]{Proposition}
\newtheorem{ques}[thm]{Question} 

\newtheorem{thmi}{Theorem}
\newtheorem{cori}[thmi]{Corollary}

\newtheorem{examplei}[thmi]{Example}

\theoremstyle{definition}
\newtheorem{defni}[thmi]{Definition}

\newtheorem{questioni}[thmi]{Question}

\newtheorem{openQ}{Question}
\newtheorem{openP}[openQ]{Problem}

\theoremstyle{definition}

\newtheorem{defn}[thm]{Definition}
\newtheorem{rem}[thm]{Remark}

\theoremstyle{definition}
\newtheorem{exmp}[thm]{Example}

\newtheorem{claim}{Claim}

\newtheorem{claim*}{Claim}

\DeclareMathOperator{\image}{im}

\DeclareMathOperator{\Out}{Out}

\DeclareMathOperator{\diam}{diam}

\DeclareMathOperator{\CAT}{CAT}

\makeatletter

\newcommand{\Rmnum}[1]{\mathbf{{\expandafter\@slowromancap\romannumeral #1@}}}

\makeatother

\newcommand{\tsh}[1]{\left\{\kern-.7ex\left\{#1\right\}\kern-.7ex\right\}}
\newcommand{\Tsh}[2]{\tsh{#2}_{#1}}
\newcommand{\ignore}[2]{\Tsh{#2}{#1}}

\usepackage{calc}
\usepackage{array}

\newcommand*\mc[1]{\mathcal{#1}}

\newcommand{\cay}{\operatorname{Cay}}

\newcommand{\R}{\mathbb{R}}
\newcommand{\Z}{\mathbb{Z}}
\newcommand*\cone[1]{\widehat{#1}}
\newcommand\decomposition{\sigma_0 \ast \alpha_1 \ast \sigma_1 \ast \cdots \ast \alpha_n \ast \sigma_n}
\newcommand{\deep}{\operatorname{deep}}
\newcommand{\rel}{\operatorname{Relevant}}
\newcommand{\ball}{\operatorname{Ball}}

\usepackage{multicol}
\usepackage{ulem}
\usepackage{float}
\usepackage{subcaption}

\newcommand*\notinfinal[1]{}

\makeindex

\title[Morse local-to-global]{The local-to-global property for Morse quasi-geodesics}
\author{Jacob Russell}
\address{CUNY Graduate Center, 365 Fifth Avenue, New York, NY 10016, USA}
\email{jrussellmadonia@gradcenter.cuny.edu}

\author{Davide Spriano}
\address{ETH Z\"urich, Raemistrasse 101, 8092 Z\"urich, Switzerland}
\email{spriano@maths.ox.ac.uk}

\author{Hung Cong Tran}
\address{University of Oklahoma, Norman, OK 73019-3103, USA}
\email{Hung.C.Tran-1@ou.edu}

\date{\today}
\begin{document}
\maketitle

\begin{abstract}
We show the mapping class group, $\CAT(0)$ groups, the fundamental groups of closed 3--manifolds, and certain relatively hyperbolic groups have a local-to-global property for Morse quasi-geodesics. This allows us to generalize combination theorems of Gitik for quasiconvex subgroups of hyperbolic groups to the stable subgroups of these groups. In the case of the mapping class group, this gives combination theorems for convex cocompact subgroups. We show a number of additional consequences of this local-to-global property, including a Cartan--Hadamard type theorem for detecting hyperbolicity locally and discreteness of translation length of conjugacy classes of Morse elements with a fixed gauge. To prove the relatively hyperbolic case, we develop a theory of deep points for local quasi-geodesics in relatively hyperbolic spaces, extending work of Hruska.
\end{abstract}

\tableofcontents

\section{Introduction}
We provide a new approach for studying groups and spaces with features of non-positive curvature by  defining a local-to-global property for Morse quasi-geodesics.
We show  a variety of results from hyperbolic spaces effortlessly  extend to any space satisfying this Morse local-to-global property. We prove examples of such spaces include  the mapping class group and Teichm\"uller space of finite type, orientable surfaces, the fundamental groups of closed $3$--manifolds, and all $\CAT(0)$ spaces.

\subsection{The local-to-global property for Morse quasi-geodesics}

    Morse quasi-geodesics generalize a key property of quasi-geodesics in hyperbolic spaces. Given a function $M \colon [1,\infty)\times [0,\infty) \to [0,\infty)$, a quasi-geodesic $\gamma$ is \emph{$M$--Morse} if every $(\lambda,\epsilon)$--quasi-geodesic with endpoints on $\gamma$ is contained in the $M(\lambda,\epsilon)$--neighborhood of $\gamma$. We call the function $M$ the \emph{Morse gauge} of $\gamma$.\footnote{To simplify proofs, we use a more technical, but equivalent, definition of Morse quasi-geodesics; see Section \ref{subsec:Morse_geodesics}.}
    
    The study of Morse quasi-geodesics arose from trying to understand the ``hyperbolic directions" in a non-hyperbolic space \cite{Charney_Sultan_CAT(0)} and has since received immense interest in the literature (see \cite{ABD,HamenstaedtHensel_Stability_Outer_Space,ACGH,DMS_divergence,OOS_Lacunary_hyperbolic_groups} for a sampling). Numerous results from hyperbolic spaces have fruitful generalizations to spaces containing infinite Morse quasi-geodesics, particularly with respect to the study of stable subgroups  \cite{AMST,Cordes_Hume_boundary,ADT} and the quasi-isometric classification of spaces \cite{MC1, CCM_quasi-mobius}. Our new contribution to the study of Morse quasi-geodesics is the introduction of a local-to-global property for Morse quasi-geodesics.
    
    Intuitively, a local Morse quasi-geodesic can be thought of as a path where every subpath of a specified length is uniformly a Morse quasi-geodesic. Since quasi-geodesics need not be continuous, we make this idea rigorous by measuring the length in the domain of the map.

\begin{defni}[Local Morse quasi-geodesic]
    Let $L\geq 0$, $\lambda\geq 1$, $\epsilon \geq 0$, and $M$ be a Morse gauge. A map $\gamma \colon [a,b] \to X$ is an \emph{$(L;M;\lambda,\epsilon)$--local Morse  quasi-geodesic} if for any $[s,t] \subseteq [a,b]$ with $|s-t| \leq L$, the restriction of $\gamma$ to $[s,t]$ is an $M$--Morse, $(\lambda,\epsilon)$--quasi-geodesic.
\end{defni}

 A local-to-global property for Morse quasi-geodesics should give conditions for a local Morse quasi-geodesic to be a global Morse quasi-geodesic. However, every finite quasi-geodesic will be Morse for some Morse gauge, implying every $L$--local quasi-geodesic will be a local Morse quasi-geodesic for a Morse gauge depending on $L$. Thus, a meaningful local-to-global property requires the number $L$ to depend on the Morse gauge.
\begin{defni}[Morse local-to-global property]\label{intro_defn:local_to_global}
    A metric space $X$ has the \emph{Morse local-to-global} property if for every Morse gauge $M$ and constants $\lambda \geq 1$, $\epsilon \geq 0$, there exists  $L \geq 0$ so that every $(L;M;\lambda,\epsilon)$--local Morse quasi-geodesic is a global $M'$--Morse, $(\lambda',\epsilon')$--quasi-geodesic for some $M'$, $\lambda'$, and $\epsilon'$ depending only on $M$, $\lambda$, and $\epsilon$. A finitely generated group has the \emph{Morse local-to-global} property if any Cayley graph with respect to a finite generating set has the Morse local-to-global property.
\end{defni}

Since Morse quasi-geodesics are invariant under quasi-isometry, the Morse local-to-global property is a quasi-isometry invariant. In particular, the Cayley graph of a finitely generated group will have the Morse local-to-global property regardless of the choice of finite generating set.

The sizable literature  generalizing the behavior of quasi-geodesics in hyperbolic spaces to Morse quasi-geodesics in all metric spaces may lead one to suspect that all metric spaces would have the Morse local-to-global property. However, the results of our study show that this conjecture is false even among finitely generated groups. 

\begin{examplei}
There exist finitely generated groups whose Cayley graph contain bi-infinite Morse geodesics, but do not have the Morse local-to-global property. See Example \ref{ex:non-examples} for specific groups. 
\end{examplei}

The results of this paper therefore come in two flavors. We show  a wide variety of groups and spaces do have the local-to-global property for Morse quasi-geodesics, and we demonstrate a range of results of consequences for these space.

\subsection{Examples of Morse local-to-global Spaces}

The most significant work of this paper lies in proving the collection of groups and spaces with the Morse local-to-global property is very broad.

\begin{thmi}[Morse local-to-global groups and spaces]\label{intro_thm:examples_of_local_to_global}
	The following groups and spaces have the Morse local-to-global property.
	\begin{itemize}
	    \item Any $\CAT(0)$ space.
		\item The mapping class group of an orientable, finite type surface.
		\item The Teichm\"uller space of an orientable, finite type surface with either the Weil--Petersson or Teichm\"uller metric.
		\item Any graph product of hyperbolic groups.
		\item All finitely generated virtually solvable groups.
		\item Any finitely generated group with an infinite order central element.
		\item The fundamental group of any closed $3$--manifold.	
	\end{itemize}
\end{thmi}

The proof that CAT$(0)$ spaces have the Morse local-to-global property relies on the convexity of the distance function and a characterization of Morse geodesics in CAT$(0)$ spaces due to Charney and Sultan \cite{Charney_Sultan_CAT(0)}. For the mapping class group, Teichm\"uller space, and graph products of hyperbolic groups, we prove a more general result for the class of hierarchically hyperbolic spaces satisfying a minor technical condition that encompasses these examples. The proof for hierarchically hyperbolic spaces rests upon a characterization of Morse quasi-geodesics in these space due to Abbott, Behrstock, and Durham \cite{ABD}.  Virtually solvable groups and groups with infinite order central elements are examples of spaces where the Morse geodesics have uniformly bounded length. We call these spaces \emph{Morse limited} and prove they all trivially satisfy Definition \ref{intro_defn:local_to_global} by taking $L$ sufficiently larger than the bound on the lengths of the Morse geodesics. These ``trivial" examples are essential to proving that the fundamental group of any closed $3$--manifold has the Morse local-to-global property. The fundamental group of a $3$--manifold decomposes into a free product of groups that are either virtually solvable or hierarchically hyperbolic spaces. The proof for all closed $3$--manifold thus follow from the above plus the next theorem about relatively hyperbolic groups.

\begin{thmi}[Hyperbolic relative to Morse local-to-global]\label{intro_thm:relative_hyperbolicity}
	Let $G$ be a group hyperbolic relative to the subgroups $H_1,\dots,H_n$. If each $H_i$ has the Morse local-to-global property, then $G$ has the Morse local-to-global property.
\end{thmi}
The proof of Theorem \ref{intro_thm:relative_hyperbolicity} is, somewhat surprisingly, the most involved of this paper. Our proof develops a theory of ``deep points" for local quasi-geodesics in a relatively hyperbolic space analogous to the deep points for geodesics introduced by Hruska \cite{Hruska10}. The deep points partition a local quasi-geodesic into pieces that alternatingly avoid and pass through the peripheral subsets of the relatively hyperbolic space and maybe of independent interest in the study of relative hyperbolicity.

\begin{thmi}[Local quasi-geodesics in relatively hyperbolic spaces]\label{intro_thm:deep_point_decomposition}
    Let $X$ be hyperbolic relative to a collection $\mc{P}$ of peripheral subsets. For each $\lambda \geq 1$ and $\epsilon \geq 0$, there exists $L >0$ so the following holds. For every $(L;\lambda,\epsilon)$--local quasi-geodesic  $\gamma$ in $X$,  there exist peripheral subsets $\{P_1,\dots,P_n\}$ and a decomposition $\gamma = \decomposition$ so that each $\alpha_i$ is contained in a uniform neighborhood of $P_i$ and each $\sigma_i$  uniformly quasi-isometrically embeds in the coned-off space $\cone{X}$.  This implies each $\sigma_i$ is uniformly a Morse quasi-geodesic in $X$.
\end{thmi}

\subsection{Consequences of the  Morse local-to-global property}
 Once the difficult task of showing a space satisfies the Morse local-to-global property is completed, many arguments from hyperbolic spaces can be seamlessly generalized.
 
 Our first examples of this are extensions of Gitik's combination theorems for quasiconvex subgroups of hyperbolic groups to stable subgroups of Morse local-to-global groups. A streamlined version of the combination theorems is the following (see Section \ref{sec:combination_theorems} for complete statements).

\begin{thmi}[Combinations of stable subgroups]\label{intro_thm:combination_theorems}
	 Let $G$ be a finitely generated group with the Morse local-to-global property and $P$, $Q$ be  stable subgroups of $G$.
	 For each  finite generating set of $G$, there exists $C>0$ (depending only on the stability parameters of $P$ and $Q$) so that if $P\cap Q$ contains all elements of  $P \cup Q$ whose length in $G$ is at most $C$, then the subgroup $\langle P,Q \rangle$ is stable in $G$ and isomorphic to $P \ast_{P\cap Q} Q$. Further, if $P$ is malnormal in $G$, then the same conclusion holds if we only require $P\cap Q$ to contain all elements of $P$  whose length in $G$ is at most $C$.
\end{thmi}

Stable subgroups are a strong generalization of quasiconvex subgroups that requires every pair of elements to be joined by a uniform Morse geodesic that stays uniformly close to the subgroup.  Stable subgroups were introduced by  Durham and Taylor to study the convex cocompact subgroups of the mapping class group \cite{Durham_Taylor_Stability}, but have sense been studied in a variety of non-hyperbolic groups \cite{HamenstaedtHensel_Stability_Outer_Space,ABD,Tran2017}. Theorem \ref{intro_thm:combination_theorems} is particularly interesting in the mapping class group, where it produces combination theorems for the convex cocompact subgroups by the work of Durham and Taylor (see Section \ref{subsec:intro_MCG} for details).  Similar combination theorems have been proved by Dey, Kapovich, and Leeb for Anosov subgroups of semi-simple Lie groups \cite{Dey_Kapovich_Leeb_combination} and by  Mart{\'i}nez-Pedroza and Sisto  for relatively quasiconvex subgroups of relatively hyperbolic groups \cite{Martinez_Sisto_combination}.

In Section \ref{sec:combination_theorems}, we record several corollaries to Theorem \ref{intro_thm:combination_theorems}, including the fact that all infinite normal subgroups of a Morse local-to-global group contain a Morse element.

\begin{cori}[Normal subgroups of Morse local-to-global groups]\label{intro_cor:normal_subgroups}
	Let $G$ be a finitely generated group with the Morse local-to-global property. If $G$ contains a Morse element and $N$ is an infinite normal subgroup of $G$, then $N$ contains a Morse element of $G$.
\end{cori}

An element $g \in G$ is  \emph{$M$--Morse} if $\langle g \rangle$ is undistorted and every pair of elements of $\langle g \rangle$ is joined by a $M$--Morse geodesic in the Cayley graph of $G$. Prominent examples of Morse elements include pseudo-Anosov elements of the mapping class group and rank-1 elements of CAT$(0)$ groups.  
In these cases,  Corollary \ref{intro_cor:normal_subgroups} can be obtained from Ivanov's omnibus subgroup theorem \cite{ivanov_subgroups} or a result of Arzhantseva, Cashen, and Tao for groups acting on spaces with strongly contracting elements \cite[Proposition 3.1]{ACT_Growth_tight}. Corollary \ref{intro_cor:normal_subgroups} provides a proof entirely in terms of the Cayley of the group in these cases, and extends these results to all hierarchically hyperbolic groups and closed $3$--manifold groups containing Morse elements.

Theorem \ref{intro_thm:combination_theorems} also allows us to produce a trichotomy that guarantees Morse local-to-global groups that are not Morse limited or virtually cyclic cannot be amenable. 

\begin{cori}[Trichotomy for Morse local-to-global groups]\label{intro_cor:trichotomy}
    A finitely generated group $G$ with the Morse local-to-global property satisfies exactly one of the following.
	\begin{enumerate}
		\item $G$ is Morse limited.
		\item $G$ is virtually infinite cyclic.
		\item $G$ contains an infinite index stable free subgroup of rank $2$.
	\end{enumerate}
\end{cori}

Our last corollary is an generalization of a theorem of Arzhantseva to stable subgroups of torsion-free, Morse local-to-global groups \cite[Theorem 1]{Arzhantseva_quasiconvex_subgroups}.
\begin{cori}
Let $G$ be a torsion free, Morse local-to-global group. If $Q$ is a non-trivial, infinite index stable subgroup of $G$, then there is an infinite order element $h$ such that $\langle Q,h\rangle \cong Q \ast \langle h \rangle$ and $\langle Q,h\rangle$ is stable in $G$.
\end{cori}

Beyond our combination theorem, we generalize an assertion of Gromov \cite{Gromov_Hyp_Groups} that was proved by Delzant \cite{Delzant} about the set of translation lengths for elements of a hyperbolic group.  We say a conjugacy class is $M$--Morse if all elements with the shortest word length in the conjugacy class are $M$--Morse.

\begin{thmi}[Translation length of Morse elements]\label{intro_thm:translation_lengths}
	Let $G$ be a finitely generated group with the Morse local-to-global property. For each Morse gauge $M$, the set of translation lengths of $M$--Morse conjugacy classes of $G$ is a discrete set of rational numbers.
\end{thmi}

 We also establish a Cartan--Hadamard style theorem for Morse local-to-global spaces that allows for hyperbolicity to be checked locally.

\begin{thmi}[Local condition for hyperbolicity]\label{intro_thm:Cartan-Hadamard}
	Let $X$ be a geodesic metric space with the Morse local-to-global property. For each $\delta\geq 0$, there exists $R\geq 0$ so that if every triangle in $X$ with vertices contained in a ball of radius $R$ is $\delta$--slim, then $X$ is a hyperbolic metric space.
\end{thmi}

Theorem \ref{intro_thm:Cartan-Hadamard} is similar to Gromov's local condition for hyperbolicity in simply connected spaces \cite[Theorem 4.1.A]{Gromov_Hyp_Groups}.  While our result does not apply to all simply connected spaces, the proof of Theorem \ref{intro_thm:Cartan-Hadamard}  is considerably shorter than proofs of Gromov's condition appearing in the literature (see \cite{Coulon_Cartan-Hadamard,Bowditch_Note_on_Hyperbolicity,Papasoglu_algorithm}) while still capturing the important cases of $\CAT(0)$ spaces and the universal covers of closed $3$--manifolds.

\subsection{Convex cocompact subgroups}\label{subsec:intro_MCG}

Farb and Mosher originally introduced convex cocompact subgroups of the mapping class group by analogy with convex cocompact subgroups of Kleinian groups \cite{Farb_Mosher_convex_cocompact}. These subgroups have since attracted substantial attention as they play a critical role in the theory of extensions of hyperbolic groups and the geometry of surface bundles \cite{Hamenstadt_extensions_of_surface_groups,Kent_Leininger_convex_cocompactness}.  A theorem of Durham and Taylor established that the stable subgroups of the mapping class group are precisely the convex cocompact subgroups \cite{Durham_Taylor_Stability}. Using this equivalence, Theorem \ref{intro_thm:combination_theorems} produces a combination theorem for convex cocompact subgroups of the mapping class group.

\begin{thmi}\label{intro_thm:MCG_combination_theorem}
	 Let $\mathrm{MCG}(S)$ be the mapping class group of a orientable surface $S$ of finite type. Suppose $P$ and $Q$ are convex cocompact subgroups of $\mathrm{MCG}(S)$. There exists $C>0$ so that if $P\cap Q$ contains all elements of  $P \cup Q$ whose word length in $\mathrm{MCG}(S)$ is at most $C$, then the subgroup $\langle P,Q \rangle$ is convex cocompact and isomorphic to $P \ast_{P\cap Q} Q$.  Further, if $P$ is malnormal, then the same conclusion holds if we only require $P\cap Q$ to contain all elements of $P$  whose length in $\mathrm{MCG}(S)$ is at most~$C$.
\end{thmi}

A long standing open question of Farb and Mosher makes combination theorems for convex cocompact subgroups particularly interesting

\begin{questioni}[{\cite[Question 1.7]{Farb_Mosher_convex_cocompact}}]\label{intro_ques:surface_subgroup}
Does there exist a convex cocompact subgroup of the mapping class group that is not free or virtually free? 
\end{questioni}

Since examples of virtually free convex cocompact subgroups are abundant, a strategy to answer Question \ref{intro_ques:surface_subgroup} could involve using Theorem \ref{intro_thm:MCG_combination_theorem} to create new examples using an amalgamated free product. This remains beyond our current reach in the mapping class group, but we produce new examples of one-ended stable subgroups of $\CAT(0)$ groups using our combination theorem in Example \ref{ex:example_of_combination_theorem}.

\subsection{Outline of the paper}
Section \ref{sec: background} consists of preliminary results on Morse quasi-geodesics and the definition of the Morse local-to-global property. 
Section \ref{sec: consequences of Mltg} is dedicated to consequences of the Morse local-to-global property. 
Sections \ref{sec:examples} and \ref{sec:relativelyhyperbolic} contain proofs that the Morse local-to-global property is exhibited by a large number of spaces. Section \ref{sec:examples} contains the cases of $\CAT(0)$ spaces, hierarchically hyperbolic spaces, and virtually solvable groups, while Section \ref{sec:relativelyhyperbolic} contains the cases of relatively hyperbolic spaces and $3$--manifold groups. 

Below, we report open questions that arose from our study of the Morse local-to-global property. This outlines a rich new direction for research in groups with features of non-positive curvature.

\subsection{Open questions and further directions}

The examples we give of groups without the Morse local-to-global property are all infinitely presented. Our first question, inspired by Ruth Charney, asks how ``nice" such a group could be.

\begin{openQ}
Does there exist an example of a finitely presented group that is not Morse local-to-global? Does there exist an example of a group with quadratic Dehn function that is not Morse local-to-global? Does there exist an example of a bi-automatic group that is not Morse local-to-global?
\end{openQ}

We have shown that several different notions of non-positive curvature ($\CAT(0)$, hierarchically hyperbolic, relatively hyperbolic) imply the Morse local-to-global property. It is natural to wonder if other groups with features of non-positive curvature are also Morse local-to-global.

\begin{openQ}\label{question:Out(F_n) and Free by cycli}
Do free-by-cyclic groups or the outer automorphism group of a free group have the Morse local-to-global property?
\end{openQ}

The Morse local-to-global property is closed under both free products and direct products. This inspires the question of what other combination of groups is the property closed under.

\begin{openQ}
Do graph products of Morse local-to-global groups have the Morse local-to-global property?
\end{openQ}

\begin{openQ}
When does a graph of Morse local-to-global groups have the Morse local-to-global property?
\end{openQ}

Several of the properties of Morse local-to-global groups that we establish are reminiscent of properties of acylindrically hyperbolic groups. Further, every known example of a Morse local-to-global group that is not Morse limited or virtually cyclic is acylindrically hyperbolic. We ask if this is always the case.

\begin{openQ}\label{question:acylindrically_hyperbolic}
If $G$ is a Morse local-to-global group that is not Morse limited or virtually cyclic, is $G$ acylindrically hyperbolic?
\end{openQ}

Note, the converse of Question \ref{question:acylindrically_hyperbolic} is false. Counterexamples can be found by taking the free product of two copies of either of the groups we show to not have the Morse local-to-global property in Example \ref{ex:non-examples}. The resulting group is acylindrically hyperbolic, but not Morse local-to-global.

One consequence of a positive answer to Question \ref{question:acylindrically_hyperbolic} would be that all Morse local-to-global groups that are not Morse limited or virtually cyclic would have property $P_{naive}$ \cite{Abbott_Dahmani_Pnaive}. That is, for any collection of group elements $g_1,\dots,g_n$, there would exist an element $h$ so that $\langle g_i, h\rangle \cong \langle g_i \rangle \ast \langle h \rangle$ for each $1\leq i \leq n$. Property $P_{naive}$ is a strong version of the ping-pong lemma and very close to several of the results in this paper, making it an attractive property to study directly in these groups.

\begin{openQ}
If $G$ is a Morse local-to-global group that is not Morse limited or virtually cyclic, does $G$ have property $P_{naive}$?
\end{openQ}

A classical application of the local-to-global property of quasi-geodesic in hyperbolic groups is Cannon's proof that the geodesics of a hyperbolic group form a regular language \cite{Cannon_regular_language}. In \cite{Eike_Zal_regular_language}, the local nature of Morse geodesics in $\CAT(0)$ spaces is used to show that the $M$--Morse geodesics of a $\CAT(0)$ group also form a regular language. Extending these result to all Morse local-to-global groups would produce new results for both the mapping class group and $3$--manifold groups.

\begin{openQ}\label{question:regular_language}
Do the $M$--Morse geodesics of a Morse local-to-global group form a regular language?\footnote{Regular languages of Morse geodesic in Morse local-to-global groups have since been constructed in  \cite{CRSZ_languages}.}
\end{openQ}

In CAT$(0)$ spaces and hierarchically hyperbolic spaces, Morse quasi-geodesics have equivalent formulations in terms of contracting and divergence properties \cite{Charney_Sultan_CAT(0),ABD,RST_convexity}. The authors of \cite{ACGH} showed these characterizations do not hold in general metric spaces, but provided characterization of Morse quasi-geodesics using much weaker contracting and divergence properties. We ask if Morse quasi-geodesics in Morse local-to-global spaces have stronger contracting and divergence properties as in the CAT$(0)$ and hierarchically hyperbolic cases. A positive resolution of this question may assist in answering Questions \ref{question:acylindrically_hyperbolic} and \ref{question:regular_language}.

\begin{openQ}
Do Morse quasi-geodesic in Morse local-to-global spaces have stronger contracting or divergence properties than general Morse quasi-geodesics?
\end{openQ}

In section \ref{sec:trivial_examples}, we show that if a space has any asymptotic cone with no cut-points, then it is Morse limited. It is unknown if the converse is true.

\begin{openQ}\label{question:morse_limited_to_unconstriced}
If $G$ is Morse limited, does  $G$ have an asymptotic cone with no cut-points?
\end{openQ}

A positive answer to Question \ref{question:morse_limited_to_unconstriced} would upgrade Corollary \ref{intro_cor:trichotomy} to say every Morse local-to-global group either contains a Morse element or has an asymptotic cone with no cut-points. This would give a positive answer to \cite[Question 6.10]{BD_short_conjugators} that asks if every $\CAT(0)$ group with a cut-point in every asymptotic cone must contain a Morse quasi-geodesic.

Our final two questions seek to improve on results in is paper. The first asks to make our stable subgroup combination theorem effective in specific examples.
\begin{openP}
 Effectivize the stable subgroup combination theorems (Theorem \ref{thm:stable_subgroup_combination}) for the mapping class group and/or CAT(0) groups.
\end{openP}

In the case of the mapping class group, Bowditch shows a much stronger version of Theorem \ref{intro_thm:translation_lengths}, namely that the set of translation lengths of all Morse elements, independent of Morse gauge, is a discrete set of rational numbers \cite[Corollary 1.5]{Bowditch_tight_geodesics}. We ask if the same improvement can be made for all Morse local-to-global groups.

\begin{openQ}\label{question:translation_length}
 If $G$ is Morse local-to-global group, is the set of translation lengths of all Morse elements of $G$ a discrete set of rational numbers?
\end{openQ}

\noindent \textbf{Acknowledgments:}
The authors are grateful to Sam Taylor for insightful conversations and answering many questions. They also thank Emily Stark for helping to find references for Example~\ref{ex:example_of_combination_theorem} and Thomas Ng for comments on a draft of this paper. The first two authors thank their advisors, Jason Behrstock and Alessandro Sisto, for their support and guidance during this project. The first and third author thank the organizers of the 2019 Tech Topology conference where some of the work on this project was completed. The first author would also like to thank the FIM Institute for Mathematical Research at ETH Z\"urich for their hospitality during the conference ``Groups, spaces, and geometries" where much of the work on this project was completed. Lastly, the authors thank the anonymous referee for their helpful comments.

\section{Background and preliminaries}\label{sec: background}
\subsection{Groups and Cayley graphs}
The majority of the metric spaces we are interested in will be the Cayley graphs of finitely generated groups.

\begin{defn}
If $G$ is a group generated by a set $S$, then the \emph{Cayley graph of $G$ with respect to $S$} is the graph with all elements of $G$ as vertices and where $g,h \in G$ are joined by an edge if $g^{-1}h \in S \cup S^{-1}$. We denote the Cayley graph of $G$ with respect to $S$ by $\cay(G,S)$.
\end{defn}

Given a fixed generating set $S$ for $G$, we use $|g|$ to denote the minimum number of elements of $S \cup S^{-1}$ needed to write $g$, i.e., the word length of $g$ with respect to $S$. The word length of $g$ with respect to $S$ is equal to the distance in $\cay(G,S)$ from the identity $e$ to $g$. Every path is $\cay(G,S)$ is labeled by some word $w$ in the set $S \cup S^{-1}$. If a path labeled with the word $w$ starts at $e$ and ends at $g$, then the word $w$ represents $g$. Throughout this article, we will always implicitly consider a finitely generated group as a metric space by equipping it with a word metric with respect to some finite generating set.

\subsection{Morse quasi-geodesics}\label{subsec:Morse_geodesics}
In this section, $(X,d)$ will denote a metric space and $I$ will denote a closed, but possibly unbounded, interval of $\R$. Throughout this article, if $\gamma \colon I \to X$ is a map, we will abuse notation by using $\gamma$ to refer the image of $\gamma$ in $X$. The main objects of this paper are quasi-geodesics with a stability property called Morse.

    \begin{defn}[Quasi-geodesics] 
    A map $\gamma \colon I\to X$ is a \emph{$(\lambda,\epsilon)$--quasi-geodesic} if $\lambda \geq 1$, $\epsilon \geq 0$ and for each $ s,t \in I$
    \[\frac{1}{\lambda}|s-t|-\epsilon\leq d\bigl(\gamma(s),\gamma(t)\bigr)\leq \lambda |s-t|+\epsilon.\]
    We say $\gamma$ is a \emph{finite} quasi-geodesic if $I$ is a compact interval and an \emph{infinite} quasi-geodesic if $I$ is not compact. A \emph{subsegment} of a quasi-geodesic $\gamma \colon I \to X$ is a restriction of $\gamma$ to a closed, connected subset of $I$. The \emph{parametrized length} of a quasi-geodesic is the length of the domain of the quasi-geodesic.
    \end{defn}

    \begin{defn}[Morse quasi-geodesic]\label{defn:Morse_quasi-geodesic}
    Let $M \colon [1,\infty)\times [0,\infty) \to [0,\infty)$ be a  function. The quasi-geodesic $\gamma \colon I \to X$ is an \emph{$M$--Morse quasi-geodesic} if for all $s<t$ in $I$, if $\alpha$ is a $(\lambda,\epsilon)$--quasi-geodesic with endpoints $\gamma(s)$ and $\gamma(t)$, then the Hausdorff distance between $\alpha$ and $\gamma \vert_{[s,t]}$ is bounded by $M(\lambda,\epsilon)$. If $\gamma$ is an $M$--Morse $(\lambda,\epsilon)$--quasi-geodesic, we say $\gamma$ is an $(M;\lambda,\epsilon)$--Morse quasi-geodesic.
    \end{defn}

    Those familiar with the literature will note that Definition \ref{defn:Morse_quasi-geodesic} is stronger than the usual definition of a Morse quasi-geodesic. The next lemma shows that, up to a modification of the Morse gauge, Definition \ref{defn:Morse_quasi-geodesic} is equivalent to the usual definition of Morse quasi-geodesic. This allows us to simplify proofs by utilizing the stronger definition of Morse when working with Morse quasi-geodesics, but only demonstrating the weaker condition in Lemma \ref{lem:Morse_easy_proof} when proving a quasi-geodesic is Morse.

    \begin{lem}[Verification of Morse]\label{lem:Morse_easy_proof}
    Let $\gamma \colon I \to X$ be a $(\lambda,\epsilon)$--quasi-geodesic. Suppose there exists a function $N \colon [1,\infty) \times [0,\infty) \to [0,\infty)$ such that for all $s,t \in I$, if $\alpha$ is a $(\lambda', \epsilon')$--quasi-geodesic with endpoints $ \gamma(s)$ and $\gamma(t)$, then $\alpha$ is contained in the $N(\lambda',\epsilon')$--neighborhood of $\gamma$. Then, $\gamma$ is $M$--Morse where $M$ depends on only $N$, $\lambda$, and $\epsilon$.
    \end{lem}

    \begin{proof}
    This follows directly from \cite[Lemma 2.1]{MC1} and \cite[Lemma 5]{MR_Morse_Boundary}.
    \end{proof}

    \begin{rem}
        When applying previous results from the literature, if $\gamma$ is $M$--Morse in the sense of Definition \ref{defn:Morse_quasi-geodesic}, then $\gamma$ will also be $M$--Morse, with the same Morse gauge, in the more traditional weaker definition. In particular, statements proved with the weaker definition of Morse can be applied to the stronger definition with no modification.
    \end{rem}
    
    The trademark of Morse quasi-geodesics is that they mimic the behavior of quasi-geodesics in a hyperbolic space. In particular, triangles and rectangles with Morse sides will be slim and every Morse quasi-geodesic will be close to a Morse geodesic.

    \begin{defn}
    A geodesic triangle (quadrilateral) is $\delta$--\emph{slim}, if each side is contained in the $\delta$--neighborhood of the other two (three) sides.
    \end{defn}

    \begin{lem}[Morse triangles and quadrilateral; {\cite[Lemmas 2.2, 2.3]{MC1}}]
    \label{lem:Morse_polygons}
    Let $X$ be a geodesic space and $M$ be a Morse gauge.
    \begin{itemize}
        \item If $T$ is a geodesic triangle in $X$ where two of the three sides are $M$--Morse, then $T$ is $4M(3,0)$--slim and there exists a Morse gauge $M'=M'(M)$ such that all three sides of $T$ are $M'$--Morse.
        \item If $R$ is a geodesic quadrilateral in $X$ where three of the four sides are $M$--Morse, then $R$ is $8M(3,0)$--slim and there exists a Morse gauge $M'=M'(M)$ such that all four sides of $R$ are $M'$--Morse.
    \end{itemize}
    
    \end{lem}
    
    \begin{lem}[Close to Morse implies Morse]\label{lem:close_to_Morse}
    \label{lem:Morse_finite_distance}
    If $\gamma$ is an $(M;\lambda,\epsilon)$--Morse quasi-geodesic and $\alpha$ is a $(\lambda',\epsilon')$--quasi-geodesic contained in the $C$--neighborhood of $\gamma$, then $\alpha$ is $M'$--Morse where $M'=M'(M,C,\lambda,\lambda',\epsilon,\epsilon')$. In particular, if $\gamma$ is finite and $\alpha$ is the geodesic connecting the endpoints of $\gamma$, then $\alpha$ is $M'$--Morse where $M'$ depends only on $M$, $\lambda$, and $\epsilon$.
    \end{lem}
    
    \begin{proof}
    This follows immediately from Definition \ref{defn:Morse_quasi-geodesic}
    \end{proof}

   The last result we record on Morse geodesics is an application of the Arzela--Ascoli Theorem to say limits of Morse geodesics are Morse with the same gauge. We state the result just for the cases we shall use, finitely generated groups.
   
   \begin{lem}\label{lem:long_Morse_segments_imply_Morse_ray}
   Let $G$ be a group generated by the finite set $S$. If there exists a Morse gauge $M$ so that for each $n \in \mathbb{N}$ there is an $M$--Morse geodesic $\gamma_n$ in $\cay(G,S)$ with $\diam(\gamma_n) \geq 2n$, then there exists a bi-infinite $M$--Morse geodesic $\gamma \colon \R \to \cay(G,S)$.
   \end{lem}
    
    \begin{proof}
   Without loss of generality, we can assume that each $\gamma_n$ has the form $\gamma_n \colon [-n,n] \to \cay(G,S)$ with $\gamma(0) = e$. Since there are only a finite number of paths of a specific length in $\cay(G,S)$, a diagonal argument produces a subsequence $\gamma_{n_k}$ where $\gamma_{n_k}\vert_{[-\ell,\ell]} = \gamma_{n_\ell}\vert_{[-\ell,\ell]}$ whenever $\ell \leq k$. The map $\gamma \colon \mathbb{R} \to \cay(G,S)$ defined by $\gamma(t) = \gamma_{n_k}(t)$ for $|t|<k$ is then an $M$--Morse geodesic, as the $M$--Morse $\gamma_{n_k}$ exhaust $\gamma$.
    \end{proof}

\subsection{Local quasi-geodesics and the local-to-global property}
In this section, $(X,d)$ will be a metric space and $I$ will be a closed, but not necessarily bounded, interval in $\R$.

 \begin{defn}[Local quasi-geodesic]
    Let $L >0$. The map $\gamma \colon I \to X$ is an \emph{$(L;\lambda,\epsilon)$--local quasi-geodesic} if for any $s<t$ in $I$ with $|s-t| \leq L$,
    $\gamma\vert_{[s,t]}$ is a $(\lambda,\epsilon)$--quasi-geodesic. If, in addition, $\gamma\vert_{[s,t]}$ is $M$--Morse, then $\gamma$ is an \emph{$(L;M;\lambda,\epsilon)$--local Morse quasi-geodesic}. We call the number $L$ the \emph{scale} of the local quasi-geodesic.
\end{defn}

The following theorem of Gromov shows that hyperbolic spaces are characterized by the property that all local quasi-geodesics of sufficiently large scale are global quasi-geodesics.

\begin{thm}[{\cite[Proposition 7.2.E]{Gromov_Hyp_Groups}}]\label{thm:local_to_global_hyperbolic}
A geodesic metric space is hyperbolic if and only if for all $\lambda \geq 1$, $\epsilon \geq 0$ there exists $L$, $\lambda'$, $\epsilon'$ so that every $(L;\lambda,\epsilon)$--local quasi-geodesics is a $(\lambda',\epsilon')$--quasi-geodesic.
\end{thm}
    
The main topic of this paper are spaces where local Morse quasi-geodesics with sufficiently large scale are global Morse quasi-geodesics. In Sections \ref{sec:examples} and \ref{sec:relativelyhyperbolic}, we show that this class of spaces includes CAT$(0)$ spaces, hierarchically hyperbolic spaces such as the mapping class groups and Teichm\"uller space, and spaces hyperbolic relative to spaces with the local-to-global property.

\begin{defn}
Let $\mathcal{M}$ be the collection of all Morse gauges. Let $\Phi\colon  \mathcal{M} \times [1,\infty)\times [0,\infty) \to (0,\infty)\times \mc{M} \times [1,\infty)\times [0,\infty)$ be a map. A metric space $X$ has the $\Phi$--\emph{Morse local-to-global property} if every $(L;M; \lambda,\epsilon)$--local Morse quasi-geodesic in $X$ is an $(N;k,c)$--Morse quasi-geodesic where $(L,N,k,c)=\Phi(M,\lambda,\epsilon)$.
\end{defn}

\begin{rem} \ 
\begin{enumerate}
    \item One can equivalently define the Morse local-to-global property without reference to the map $\Phi$ as follows: $X$ has the \emph{Morse local-to-global} property if for every $\lambda \geq 1$, $\epsilon \geq 0$ and Morse gauge $M$, there exist $L$, $k$, $c$, and $N$ so that every $(L;M;\lambda,\epsilon)$--local Morse quasi-geodesic in $X$ is a global $(N;k,c)$--quasi-geodesic.
    \item  The Morse local-to-global property applies to local Morse quasi-geodesics with both infinite and finite domains. However, when verifying a space has the Morse local-to-global property, it is sufficient to verify the property for only the local Morse quasi-geodesics with finite domains. 
\end{enumerate}

\end{rem}

We now record a pair of basic lemmas about local quasi-geodesics the we will use throughout this paper. The first reduces the problem of showing a local quasi-geodesic is a global quasi-geodesic to showing that every subsegment of the local quasi-geodesic is uniformly close to a geodesic.

\begin{lem} \label{lem:local_close_to_global}
Let $\gamma \colon I \to X$ be an $(L;\lambda,\epsilon)$--local quasi-geodesic and $C\geq0$. If $L > \lambda(3 C +\epsilon +2)$ and for all $[s,t] \subseteq I$, $\gamma\vert_{[s,t]}$ is contained in the $C$--neighborhood of a geodesic from $\gamma(s)$ to $\gamma(t)$, then $\gamma$ is a $(\lambda',\epsilon')$--quasi-geodesic where $\lambda'$ and $\epsilon'$ depend only on $\lambda$, $\epsilon$, and $C$.
\end{lem}

\begin{proof}
Let $t_1, t_2 \in I$ with $t_1<t_2$. Since $\gamma$ is an $(L;\lambda,\epsilon)$--local quasi-geodesic, we can assume $|t_1-t_2| > L$. 

For the first inequality, let $t_1=s_0<s_1<\cdots <s_{n}=t_2$ such that $L/2\leq |s_i-s_{i+1}|<L$. Thus, each $\gamma\vert_{[s_i,s_{i+1}]}$ is a $(\lambda,\epsilon)$--quasi-geodesic and we have
\[  d\bigl(\gamma(t_1),\gamma(t_2)\bigr) \leq \sum \limits_{i=0}^{n-1} \left[\lambda|s_i - s_{i+1}| +\epsilon \right]\leq \sum \limits_{i=0}^{n-1} (\lambda+\epsilon)|s_i-s_{i+1}|\leq  (\lambda+\epsilon) |t_1- t_2|. \]

For the second inequality, let $\alpha$ be a geodesic connecting $\gamma(t_1)$ and $\gamma(t_2)$. 
Since $\gamma\bigl([t_1,t_2]\bigr) \subseteq \mc{N}_C(\alpha)$, there exists a map $p \colon [t_1,t_2] \to \image(\alpha)$ such that $p(t_1) = \gamma(t_1)$, $p(t_2) = \gamma(t_2)$, and $d\bigl(p(s),\gamma(s)\bigr) \leq C$ for all $s \in (t_1,t_2)$. Let $\gamma(t_1) = x_0, x_1,\dots x_{n} = \gamma(t_2)$ be a sequence of points of $\alpha$ so that \[1\leq d(x_i,x_{i+1}) <2 \text{ and } d\bigl(\gamma(t_1),\gamma(t_2) \bigr) = \sum\limits_{i=0}^{n-1} d \bigl( x_i, x_{i+1} \bigr).\] Let $I_i$ be the closed subsegment of $\alpha$ between $x_i$ and $x_{i+1}$ and define $J_i = p^{-1}(I_i)$ for each $0\leq i \leq n-1$.
Since $n \leq 2 d\bigl(\gamma(t_1),\gamma(t_2) \bigr) +1$ and
\[|t_1-t_2| \leq \sum \limits_{i=0}^{n-1} \diam(J_i),\]
 
\noindent the desired inequality will follow if we can uniformly bound $\diam(J_i)$ in terms of $\lambda$, $\epsilon$, and $C$ for each $0\leq i \leq  n-1$.

Suppose $\diam(J_i) >0$ for some $0\leq i \leq n-1$ and let $s,r \in J_i$ with $s<r$. By definition of $J_i$ and $p$ we have
\[d\bigl(\gamma(s),\gamma(r) \bigr) \leq d\bigl(p(s),p(r)\bigr) + 2C \leq d(x_i,x_{i+1}) + 2C \leq 2C+2. \]

 By assumption, for all $v \in [s,r]$, $\gamma(v)$ is within $C$ of a geodesic from $\gamma(s)$ to $\gamma(r)$. This implies \begin{equation*}
    d\bigl(\gamma(s),\gamma(v)\bigr) \leq 3C+2 \text{ for all } v \in [s,r]. \tag{$\ast$}\label{eq:local_close_to_global}
\end{equation*} Because $L > \lambda(3 C +\epsilon +2)$ and $\gamma$ is an $(L;\lambda,\epsilon)$--local quasi-geodesic,
\[d\bigl(\gamma(s),\gamma(s + L) \bigr) \geq \frac{L}{\lambda} - \epsilon > 3C+2.\]
Hence, (\ref{eq:local_close_to_global}) implies that $s+L \not\in [s,r]$ and $|s-r| < L$. Thus, $\gamma\vert_{[s,r]}$ is a $(\lambda,\epsilon)$--quasi-geodesic and we have 
\[|s-r| \leq \lambda \bigl( d(\gamma(s),\gamma(r) )+\epsilon \bigr) \leq \lambda(2C+2 + \epsilon).  \]

This shows $\diam(J_i) \leq \lambda(2C+2 + \epsilon) $ for all $0 \leq i \leq n-1$.
Since $[t_1 , t_2] \subseteq \bigcup \limits_{i=0}^{n-1} J_i$ and $n \leq 2 d\bigl(\gamma(t_1),\gamma(t_2) \bigr)+1$, we have \[|t_1-t_2| \leq \lambda(2C+2 + \epsilon) n \leq 2\lambda(2C+2 + \epsilon) \cdot d\bigl(\gamma(t_1),\gamma(t_2) \bigr) + 2\lambda(2C+2 + \epsilon),\]
and the proof is complete.
\end{proof}

An important source of local quasi-geodesics are concatenations of geodesics. The next lemma shows that if a concatenation of Morse geodesics is a local quasi-geodesic, then is will automatically be a local Morse quasi-geodesic.

\begin{lem}
\label{lem:concatenation_of_morse_is_quasi_geodesic}
Let $X$ be a geodesic metric space. For each $i \in \{1,\cdots,n\}$, let $\gamma_i \colon [a_i,a_{i+1}]\to {X}$ be an $M$--Morse geodesic with $\gamma_i(a_{i+1}) = \gamma_{i+1}(a_{i+1})$ for $1 \leq i \leq n-1$. If there is $L>0$ such that $|a_{i+2}-a_i|>L$ for each $i$ and the concatenation $\gamma = \gamma_1 \ast \cdots \ast \gamma_n$ is an $(L;\lambda,\epsilon)$--local quasi-geodesic for some $\lambda \geq 1$ and $\epsilon \geq 0$, then there is a Morse gauge $M'$ depending on $\lambda$, $\epsilon$, and $M$ such that $\gamma$ is an $(L;M';\lambda,\epsilon)$--local quasi-geodesic. 
\end{lem}

\begin{proof}
For each $t_1<t_2$ in $[a,b]$ satisfying $|t_1-t_2|<L$, we need to show the $(\lambda,\epsilon)$--quasi-geodesic $\gamma\vert_{[t_1,t_2]}$ is $M'$--Morse, where $M'$ depends only on $\lambda$, $\epsilon$, and $M$.

Since $|a_{i+2}-a_i|>L$ for each $i$, there are only three possibilities for the position of $t_1$ and $t_2$ in $[a_1,a_{n+1}]$:
\begin{enumerate}
    \item $t_1$ and $t_2$ both lie in $[a_j,a_{j+1}]$ for some $j$;
    \item $t_1\in[a_j,a_{j+1}]$ and $t_2\in[a_{j+1},a_{j+2}]$ for some $j$;
    \item $t_1\in[a_j,a_{j+1}]$ and $t_2\in[a_{j+2},a_{j+3}]$ for some $j$.
\end{enumerate}
 In all three cases, Lemma~\ref{lem:Morse_polygons} implies any geodesic $\alpha$ connecting $\gamma(t_1)$ and $\gamma(t_2)$ is an $N$--Morse geodesic and that lies in the $C$--neighborhood of $\gamma\bigl([t_1,t_2]\bigr)$ where $C$ and $N$ only depend on $M$. Lemma~\ref{lem:Morse_finite_distance} thus provides a Morse gauge $M'$ depending on $\lambda$, $\epsilon$, and $M$ such that $\gamma\vert_{[t_1,t_2]}$ is $M'$--Morse. Therefore, $\gamma$ is an $(L;M';\lambda,\epsilon)$--local quasi-geodesic.
\end{proof}

\subsection{Stable subgroups}
Intimately connected with the study of Morse geodesics are stable subgroups of finitely generated groups. Stable subgroups are a strong generalization of quasiconvex subgroups of hyperbolic groups.

\begin{defn}[Stable subgroup]\label{defn:stable_subgoup}
    Let $G$ be a group with finite generating set $S$, $M$ be a Morse gauge, and $\mu >0$. We say a subgroup $H < G$ is \emph{$(M,\mu)$--stable in $\cay(G,S)$} if every pair of elements in $H$ is connected by an $M$--Morse geodesic in $\cay(G,S)$ that lies in the $\mu$--neighborhood of $H$. A subgroup $H <G$ is a \emph{stable subgroup} if for any choice of finite generating set $S$ for $G$, there exist $M$ and $\mu$ such that $H$ is $(M,\mu)$--stable in $\cay(G,S)$.
    \end{defn}

The above definition of stable subgroup is equivalent to the definition originally given by Durham and Taylor in \cite{Durham_Taylor_Stability}. In Definition \ref{defn:stable_subgoup}, the parameters $(M,\mu)$ ``measure'' the stability of a subgroup $H$ in the Cayley graph $\cay(G,S)$. While the specific pair $(M,\mu)$ depends on the choice of finite generating set for $G$, the stability of the subgroup $H$ does not. That is, if $S$ and $T$ are two finite generating sets for $G$ and $H$ is $(M,\mu)$--stable in $\cay(G,S)$, then $H$ will be $(M',\mu')$--stable in $\cay(G,T)$ for some $M'$ and $\mu'$ depending on $M$, $\mu$, $S$, and $T$.

Durham and Taylor generalized several properties of quasiconvex subgroups of hyperbolic groups to all stable subgroups. The key properties we need are summarized in the next proposition.

\begin{prop}[Properties of Stable subgroups; {\cite{Durham_Taylor_Stability}}]\label{prop:stable_subgroups}
Let $G$ be a finitely generated group and $H<G$ a stable subgroup.
\begin{enumerate}
    \item $H$ is finitely generated and undistorted.
    \item $H$ is a hyperbolic group.
    \item If $K$ is a finitely generated, undistorted subgroup of $H$, then $K$ is stable in $G$.
\end{enumerate}
\end{prop}

The third author and Antol\'in, Mj, Sisto, and Taylor independently studied the intersection properties of stable subgroups and showed that all infinite index stable subgroups have finite width and are finite index in their commensurators.

\begin{thm}[{\cite[Theorem 1.2]{Tran2017}, \cite[Theorem 1.1, Proposition 3.3]{AMST}}]\label{thm:finite_width}
Let $G$ be a finitely generated group and $H$ an infinite index stable subgroup of $G$. 
\begin{itemize}
    \item $H$ is finite index in the commensurator of $H$ in $G$.
    \item $H$ has finite width, i.e., there exist $n$ so that if $\mc{H}$ is a set of at least $n$ distinct cosets of $H$, then there exists $g_1H,g_2H\in \mc{H}$ so that $g_1Hg_1^{-1} \cap g_2 H g_2^{-1}$ is finite.
\end{itemize}
\end{thm}

The most well studied examples of stable subgroups in the literature are stable cyclic subgroups. The generators of these subgroups are called \emph{Morse elements} since the cyclic subgroup they generate will be a Morse quasi-geodesic in the Cayley graph.

\begin{defn}[Morse element]\label{defn:morse_element}
Let $G$ be a finitely generated group with a finite generating set $S$ and $g \in G$. We say $g$ is $M$--\emph{Morse}
with respect to $S$, if $g$ has infinite order and $\langle g \rangle$ is $(M,\mu)$--stable in $\cay(G,S)$ for some $\mu \geq 0$. 
A group element in $G$ is \emph{Morse} if it is $M$--Morse with respect to some finite generating set $S$.
\end{defn}

While the specific Morse gauge of an element depends on the choice of generating set, whether or not an element is Morse is independent of choice of generating set.

\section{Consequences of the Morse local-to-global property}\label{sec: consequences of Mltg}

We now give our main applications of the local-to-global property for Morse quasi-geodesics. In Section \ref{sec:combination_theorems}, we prove two combination theorems for stable subgroups of Morse local-to-global groups as well as several consequences of these theorems. In Section \ref{sec:translation_length}, we show that the algebraic translation length of conjugacy classes with a fixed Morse gauge is discrete. In Section \ref{sec:Cartan-Hadamard}, we prove our local condition for checking hyperbolicity.

\subsection{Stable subgroup combination theorems}\label{sec:combination_theorems}

The primary results of this section are the following combination theorems for stable subgroups of Morse local-to-global groups. These results extend results of Gitik for quasiconvex subgroups of hyperbolic groups \cite[Theorems 1, 2]{Gitik_ping_pong} and prove Theorems \ref{intro_thm:combination_theorems} and \ref{intro_thm:MCG_combination_theorem} from the introduction.

\begin{thm}\label{thm:stable_subgroup_combination}
Let $G$ be a finitely generated group with the Morse local-to-global property and $S$ be a fixed finite generating set for $G$. If $P,Q$ are $(M,\mu)$--stable subgroups of $G$, then there exists $C = C(M,\mu,S)>0$ such that the following holds for all subgroups $P_1\leq P$ and $Q_1\leq Q$.
\begin{enumerate}
    \item \label{item:combination_thrm_1} If $P_1\cap Q_1=P\cap Q=I$ and $I$ contains all elements of $P_1 \cup Q_1$ whose word length in $G$ is less than $C$, then $\langle P_1,Q_1\rangle \cong P_1\ast_{I}Q_1$. If $P_1$ and $Q_1$ are additionally finitely generated and undistorted in $G$, then the subgroup  $\langle P_1,Q_1\rangle$ is stable in $G$.
    \item \label{item:combination_thrm_2} If $P$ is malnormal in $G$, $P_1\cap Q=P\cap Q=I$, and $I$ contains all elements of $P_1$ whose word length in $G$ is less than $C$, then $\langle P_1,Q\rangle \cong P_1\ast_{I}Q$. If $P_1$ is additionally finitely generated and undistorted in $G$, then the subgroup  $\langle P_1,Q\rangle$ is stable in $G$.     
\end{enumerate}
\end{thm}

The proofs of (\ref{item:combination_thrm_1}) and (\ref{item:combination_thrm_2}) are quite similar to each other and to Gitik's original proofs in the case of hyperbolic groups. Before giving the proofs, we sketch the argument for the case where $P_1 = P$ and $Q_1 =Q$.

The amalgamated product $P \ast_I Q$ naturally surjects onto the subgroups $\langle P , Q\rangle$. The goal is therefore to show this map is an isomorphism by proving any $h \in P \ast _I Q $ that is not in $I = P \cap Q$ is non-trivial under the map  $P \ast _I Q \to G$. The element $h\in P \ast _I Q - I$ can be decomposed as an alternating product $p_1q_1\dots p_nq_n$ of elements of $P$ and $Q$ so that each $p_i$ and $q_i$ contains some subword that is not in $P \cap Q$. Since $P$ and $Q$ are stable subgroups, the path $\gamma$ starting at the identity and labeled by $p_1q_1\dots p_nq_n$ is a concatenation of Morse geodesics. By requiring that all short element of $P$ and $Q$ are contained in $P \cap Q$, we can mimic Gitik's proof in the hyperbolic case to show that $\gamma$ is a local Morse quasi-geodesic with sufficiently large scale. We can then apply the Morse local-to-global property to show $\gamma$ is actually a uniform Morse quasi-geodesic with small enough constants that $\gamma$ cannot be a loop in the Cayley graph of $G$. Since $\gamma$ is not a loop, $p_1q_1\dots p_nq_n$ cannot be the identity and $\langle P, Q \rangle \cong P \ast_I Q$. The stability of $P$ and $Q$ implies that the uniform Morse quasi-geodesic $\gamma$ is contained in a regular neighborhood of  $\langle P, Q \rangle$, proving the subgroup is stable. 
 
\begin{proof}[Proof of (1)]
There exist $M$ and $\mu$ such that $P$ and $Q$ are $(M,\mu)$--stable in $\cay(G,S)$. Let $\delta=4M(3,0)$ and $A$ be the number of elements of $G$ with length less than $2\mu+\delta$. Define $\epsilon=4A\mu+\delta$ and let $M'$ be the Morse gauge so that a concatenation of $M$--Morse geodesics satisfying the hypothesis of Lemma \ref{lem:concatenation_of_morse_is_quasi_geodesic} is an $(L;M';3,\epsilon)$--local Morse geodesic whenever the concatenation is an $(L;3,\epsilon)$--local geodesic. Since $\cay(G,S)$ has the Morse local-to-global property, there are $L>0$, $k\geq 1$, $c\geq 0$ and a Morse gauge $N$ such that every $(L;M';3,\epsilon)$--local Morse quasi-geodesic in $\cay(G,S)$ is an $(N;k,c)$--Morse quasi-geodesic. Let $C=\max\{L,kc\}+1$.

Form the abstract group $P_1\ast_{I}Q_1$ from isomorphic copies of $P_1$ and $Q_1$. Let $\iota \colon P_1\ast_{I}Q_1 \to G$ be the natural map whose image is the subgroup $\langle P_1,Q_1 \rangle <G$. We need to prove that the map $\iota$ is injective. 

Consider an element $h \in P_1\ast_{I}Q_1$ such that $h\notin I$. We will show that $\iota(h)$ is not the identity in $G$. We can write $h$ as a product $h=p_1q_1p_2q_2\cdots q_{m-1}p_m$ where $p_i$ and $q_i$ satisfy the following:
\begin{enumerate}
    \item $q_i\in Q_1-I$ for $1\leq i \leq m-1$;
    \item $p_1,p_m \in P_1$ where $p_1$ (resp. $p_m)$ is a shortest representative of the coset $p_1 I$ (resp. $Ip_m$);
    \item For $2 \leq i \leq m-1$, $p_i\in P_1-I$ is a shortest representative of the double coset $I p_i I$.
\end{enumerate}

 For $1\leq i \leq m$, let $\alpha_i$ and $\beta_i$ be $M$--Morse geodesics  in $\cay(G,S)$ connecting the following points:
 \begin{itemize}
     \item $\alpha_1$ connects $e$ and $p_1$;
     \item $\alpha_m$ connects $p_1q_1\cdots q_{m-1}$ and $p_1q_1\cdots q_{m-1}p_m$;
     \item $\alpha_i$ connects $p_1q_1\cdots q_{i-1}$ and $p_1q_1\cdots q_{i-1}p_i$ for all $2\leq i \leq m-1$;
     \item $\beta_i$ connects $p_1q_1\cdots p_i$ and $p_1q_1\cdots p_iq_i$ for $1\leq i \leq m-1$.
 \end{itemize}
 
  Note, $\alpha_1$ is degenerate if $p_1$ is trivial and $\alpha_m$ is degenerate if $p_m$ is trivial, but all other $\alpha_i$ and $\beta_i$ have length as least $C>L$. Let $\gamma=\alpha_1\ast \beta_1\ast \alpha_2\ast \beta_2\ast\cdots\ast \beta_{m-1} \ast \alpha_m$. We claim that $\gamma$ is an $(L;M';3,\epsilon)$--local Morse quasi-geodesic. 
  
  Since each $\alpha_i$ except $\alpha_1$ and $\alpha_m$ has length at least $ C>L$, Lemma \ref{lem:concatenation_of_morse_is_quasi_geodesic} states it is sufficient to verify that $\gamma$ is an $(L;3,\epsilon)$--local quasi-geodesic. Let $\psi$ be a subsegment of $\gamma$ with parametrized length at most $L$.   If $\psi$ is contained entirely in a single $\alpha_i$ or $\beta_i$, then $\psi$ is a $(3,\epsilon)$--quasi-geodesic. Otherwise, $C > L$ implies that $\psi$ decomposes into two pieces $\eta_1$ and $\eta_2$ where, without loss of generality, $\eta_1 \subseteq \alpha_i$ and $\eta_2 \subseteq \beta_i$ for some $1\leq i \leq m-1$. If $\eta_3$ is a  geodesic in $\cay(G,S)$ connecting the endpoints of $\psi$, then $\eta_1\cup \eta_2 \cup \eta_3$ is $\delta$--slim by Lemma~\ref{lem:Morse_polygons}.

In the proof of  \cite[Theorem 1]{Gitik_ping_pong}, Gitik shows that if $G$ is a $\delta$--hyperbolic group, then $\psi$ will be a $(3,\epsilon)$--quasi-geodesic. This argument only uses the fact that the triangle $\eta_1\cup \eta_2 \cup \eta_3$ is uniformly $\delta$--slim, that the segments $\alpha_i$ and $\beta_i$ are contained in the $\mu$--neighborhood of a coset of either $P$ or $Q$ respectively, that there are only  $A$ elements of $G$ of length $2\delta+\mu$, and the minimality of the choice of the $p_i$. Since these facts remain true in this setting, we can apply the same argument to conclude that $\psi$ is a $(3,\epsilon)$--quasi-geodesic. This implies $\gamma$ is an $(L;3,c)$--local quasi-geodesic and hence an $(L;M';3,\epsilon)$--local Morse quasi-geodesic.

Since $\gamma$ is an $(L; M';3,\epsilon)$--local Morse quasi-geodesic, the Morse local-to-global property implies that $\gamma$ is an $(N;k,c)$--Morse quasi-geodesic. Now, the path $\gamma$ connects the identity $e$ and the element $\iota(h)$ in $\cay(G,S)$. Since the parametrized length of $\gamma$ is greater than $C \geq kc+1$ and the distance between the endpoints of $\gamma$ is positive, and $\iota(h)$ is not the identity. This implies that $\iota$ is injective and  $\langle P_1,Q_1\rangle \cong P_1\ast_{I}Q_1$.

  If $P_1$ and $Q_1$ are finitely generated and undistorted in $G$, then they are both stable subgroups of $G$. To verify that $\langle P_1,Q_1\rangle$ is stable, it is sufficient to check that every element $h \in  \langle P_1,Q_1\rangle - I$ can be connected to the identity with a uniform quality Morse quasi-geodesic that is contained in a uniform neighborhood of $\langle P_1,Q_1\rangle$. Since $P_1$ and $Q_1$ are stable, each of the $\alpha_i$ and $\beta_i$ used to construct $\gamma$ will be uniformly close to a coset of $P_1$ or $Q_1$ respectively. Thus, the $(N;k,c)$--Morse quasi-geodesic $\gamma$ will then be contained in a uniform neighborhood of $\langle P_1,Q_1\rangle$, and $\langle P_1,Q_1\rangle$ will be a stable subgroup of $G$.
  \end{proof}

\begin{proof}[Proof of (2)]
There exist $M$ and $\mu$ such that $P$ and $Q$ are $(M,\mu)$--stable in $\cay(G,S)$. 
Let $\delta=4M(3,0)$ and  $A$ be the number of elements of $G$ with length less than $2\mu+\delta$. Let $\Gamma$ be the quotient of the action of $P$ on $\cay(G,S)$.  If $v_0 \in \Gamma$ is the vertex that  represents the orbit of the identity, then let $r$ be the number of vertices of $\Gamma$ in the ball of radius $\mu+2\delta$ centered at $v_0$. Set $\epsilon=4A\mu+\delta+r^2+1$. Let $M'$ be the Morse gauge so that a concatenation of $M$--Morse geodesics satisfying the hypotheses of Lemma \ref{lem:concatenation_of_morse_is_quasi_geodesic} is an $(L;M';6,\epsilon)$--local Morse geodesic whenever the concatenation is an $(L;6,\epsilon)$--local quasi-geodesic. Since $\cay(G,S)$  has the Morse local-to-global property, there are $L>0$, $k\geq 1$, $c\geq 0$ and a Morse gauge $N$ such that every $(L;M';6,\epsilon)$--local Morse quasi-geodesic is an $(N;k,c)$--Morse quasi-geodesic. Let $C=\max\{L,kc\}+1$.

Form the abstract group $P_1\ast_{I}Q$ from isomorphic copies of $P_1$ and $Q$. As in the proof of (\ref{item:combination_thrm_1}), we need to show that the natural map of $P_1\ast_{I}Q \rightarrow \langle P_1, Q \rangle < G$ is injective. As before, for  $h \in P_1\ast_{I}Q - I$, we can pick a representation $h=p_1q_1p_2q_2\cdots q_{m-1}p_m$ that satisfies the same properties as in (\ref{item:combination_thrm_1}) with the change that $q_i \in Q - \{e\}$ for each $1\leq i \leq m-1$. Continuing to follow the proof of (\ref{item:combination_thrm_1}), we  select  $M$--Morse geodesics $\alpha_i$ and $\beta_i$ in $\cay(G,S)$, and let $\gamma = \alpha_1 \ast \beta_1 \ast \dots \alpha_{m-1}\ast \beta_{m-1} \ast\alpha_m$. Once we show that $\gamma$ is an $(L;M';6,c)$--local Morse quasi-geodesic, the remainder of the proof will finish identically to the proof of (\ref{item:combination_thrm_1}).

 As in (\ref{item:combination_thrm_1}), every $\alpha_i$ except $\alpha_1$ and $\alpha_m$ will have length longer than $C>L$.  Thus, if $\gamma$ is an $(L;6,\epsilon)$--local quasi-geodesic, then it will satisfy the hypotheses of Lemma \ref{lem:concatenation_of_morse_is_quasi_geodesic} and be an $(L;M';6,\epsilon)$--local Morse quasi-geodesic. Let $\psi$ be a subpath of $\gamma$ with parametrized length at most $L$.   If $\psi$ is contained entirely in a single $\alpha_i$ or $\beta_i$, then $\psi$ is a $(6,\epsilon)$--quasi-geodesic. Otherwise, $C > L$ implies that $\psi$ decomposes into three pieces $\eta_1,\eta_2,\eta_3$. Without loss of generality,  $\eta_1 \subseteq \alpha_i$, $\eta_2 \subseteq \beta_i$ and $\eta_3\subseteq \alpha_{i+1}$ for some $1\leq i \leq m-1$ where  at most one of  $\eta_1$ or $\eta_3$ is empty. If $\eta_4$ is a  geodesic in $\cay(G,S)$ connecting the endpoints of $\psi$, then the rectangle $\eta_1\cup \eta_2 \cup \eta_3 \cup \eta_4$ is $2\delta$--slim by Lemma~\ref{lem:Morse_polygons}.
 
 In the proof of  \cite[Theorem 2]{Gitik_ping_pong}, Gitik shows that if $G$ is a $\delta$--hyperbolic group, then $\psi$ will be a $(6,\epsilon)$--quasi-geodesic. This argument only use the fact that  $P$ is malnormal, the rectangle $\eta_1\cup \eta_2 \cup \eta_3 \cup \eta_4$ is uniformly $2\delta$--slim, that the segments $\alpha_i$ and $\beta_i$ are contained in the $\mu$--neighborhood of a coset of either $P$ or $Q$ respectively, that $G$ contains only $A$ elements of length $2\delta + \mu$, and the minimality of the choice of the $p_i$. These facts all remain true in this setting, so we can  apply the same argument to conclude that $\psi$ is a $(6,\epsilon)$--quasi-geodesic.
 
 The remainder of the proof now follows identically to the proof of (\ref{item:combination_thrm_1}).
\end{proof}

Verifying that the intersection of two stable subgroups contains all the short elements of each subgroup can be quite challenging in practice, especially since the function $\Phi$ governing the Morse local-to-global property (and therefore the constant $C$) is often not explicit. However, one can circumvent this difficulty by utilizing the separability of subgroups.

\begin{defn}
A subgroup $H$ of a group $G$ is called \emph{separable} if for every $g\in G-H$, there is a subgroup $K$ of finite index in $G$ such that $H\leq K$ but $g\notin K$. 
\end{defn}

\begin{cor}
\label{cor:separable_intersection}
Let $G$ be a finitely generated group with the Morse local-to-global property. Let $P$ and $Q$ be infinite, stable subgroups of $G$ and  $I =P\cap Q$.
\begin{enumerate}
    \item \label{item:separability_cor_1} If $I$ is separable and infinite index in both $P$ and $Q$, then there exist infinite families of finite index subgroups $P_i <P$ and $Q_i <Q$ with $P_i \cap Q_i = I$, $P_{i} < P_{i-1}$, and $Q_i < Q_{i-1}$  so  that each subgroup $\langle P_i, Q_i\rangle$ is stable in $G$ and  $\langle P_i, Q_i\rangle \cong P_i\ast_{I} Q_i$ .  
    \item \label{item:separability_cor_2} If $P$ is malnormal in $G$ and $I$ is separable and infinite index in $P$, then there exists an infinite family of finite index subgroups $P_i <P$ with $P_i \cap Q = I$ and $P_i < P_{i-1}$ so  that each subgroup $\langle P_i, Q\rangle$ is stable in $G$ and  $\langle P_i, Q\rangle \cong P_i\ast_{I} Q$ . 
\end{enumerate}
\end{cor}

\begin{proof}
We give the proof for (\ref{item:separability_cor_1}) and the proof for (\ref{item:separability_cor_2}) is similar.
Fix a finite generating $S$ for $G$ and let $C$ be the constant from Theorem \ref{thm:stable_subgroup_combination}.(\ref{item:combination_thrm_1}) for $P$ and $Q$. For each of $P$ and $Q$, there exists a finite set of elements outside of $P\cap Q$ that have word length less than $C$. Since $P \cap Q$ is separable in $P$ and $Q$, there exist finite index subgroups $P_1<P$ and $Q_1<Q$ so that $P_1 \cap Q_1 = P \cap Q$ and all elements of $P_1$ and $Q_1$ with word length less than $C$ are contained in $P_1 \cap Q_1$. Since $P_1$ and $Q_1$ are finite index in $P$ and $Q$, they are finitely generated and undistorted in $G$ and the conclusion  follows from Theorem \ref{thm:stable_subgroup_combination}.(\ref{item:combination_thrm_1}). We produce the infinite family of subgroups by inductively separating $P_{i}$ and $Q_{i}$ from the short elements of $P_{i-1}$ and $Q_{i-1}$ that are not contained in $P \cap Q$.
\end{proof}

Corollary \ref{cor:separable_intersection} allows us to produce a plethora of examples where our combination theorem applies with a non-trivial intersection between the subgroups. We demonstrate an explicit example in right-angled Coxeter groups, which are Morse local-to-global by virtue of being  cocompact $\CAT(0)$ groups (or alternatively hierarchically hyperbolic groups).

\begin{exmp}\label{ex:example_of_combination_theorem}
\begin{figure}
\begin{tikzpicture}[scale=0.5]

\draw (0,0) node[circle,fill,inner sep=1.5pt, color=black](1){} -- (-2,-2) node[circle,fill,inner sep=1.5pt, color=black](1){}-- (-2,-5) node[circle,fill,inner sep=1.5pt, color=black](1){}-- (0,-7) node[circle,fill,inner sep=1.5pt, color=black](1){} -- (2,-5) node[circle,fill,inner sep=1.5pt, color=black](1){} -- (2,-2) node[circle,fill,inner sep=1.5pt, color=black](1){} -- (0,0) node[circle,fill,inner sep=1.5pt, color=black](1){};

\draw (-2,-2) node[circle,fill,inner sep=1.5pt, color=black](1){} -- (0,-4) node[circle,fill,inner sep=1.5pt, color=black](1){}-- (2,-2) node[circle,fill,inner sep=1.5pt, color=black](1){};

\node at (0, 0.5) {$a$}; \node at (-2.5, -2) {$b$};\node at (-2.5, -5) {$c$}; \node at (0, -7.5) {$d$};\node at (2.5, -5) {$e$}; \node at (2.5, -2) {$f$}; \node at (0, -4.5) {$g$};

\draw (0,0) node[circle,fill,inner sep=1.5pt, color=black](1){} -- (-7,-4) node[circle,fill,inner sep=1.5pt, color=black](1){}-- (0,-7) node[circle,fill,inner sep=1.5pt, color=black](1){};

\draw (0,0) node[circle,fill,inner sep=1.5pt, color=black](1){} -- (-7,-4) node[circle,fill,inner sep=1.5pt, color=black](1){}-- (0,-7) node[circle,fill,inner sep=1.5pt, color=black](1){};

\draw (0,0) node[circle,fill,inner sep=1.5pt, color=black](1){} -- (5,0) node[circle,fill,inner sep=1.5pt, color=black](1){}-- (5,-7) node[circle,fill,inner sep=1.5pt, color=black](1){}--(0,-7) node[circle,fill,inner sep=1.5pt, color=black](1){};

\node at (-7.8, -4) {$m$}; \node at (5, 0.5) {$n$}; \node at (5, -7.5) {$p$};

\end{tikzpicture}
\caption{Defining graph $\Gamma$ for the right-angled Coxeter group in Example \ref{ex:example_of_combination_theorem}}
\label{figure:defining_graph}
\end{figure}
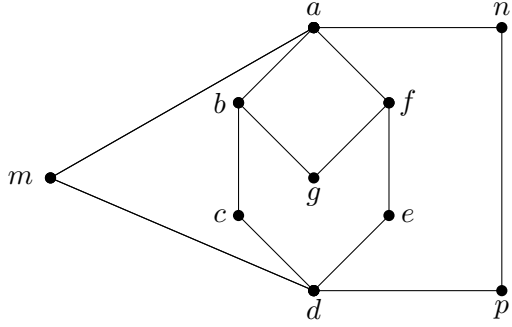

Let $G$ be the right-angled Coxeter group defined by the graph $\Gamma$ in Figure \ref{figure:defining_graph}. Let $\Gamma_1$ be the subgraph of $\Gamma$ with vertex set $\{a,b,c,d,m\}$ and let $\Gamma_2$ be the subgraph of $\Gamma$ with vertex set $\{a,f,e,d,p,n\}$. Let $P$ and $Q$ be the right-angled Coxeter subgroups of $G$ defined by $\Gamma_1$ and $\Gamma_2$ respectively. Both $P$ and $Q$ are stable subgroups by \cite[Corollary 7.12]{Tran2017}, and $P\cap Q$ is the virtually cyclic subgroup $I$ generated by the vertex set $\{a,d\}$. While the subgroups $\langle P,Q \rangle$ is isomorphic to $P\ast_I Q$, $\langle P,Q \rangle$ is not a stable subgroup by \cite[Corollary 7.12]{Tran2017}. However, $I$ is separable in both $P$ and $Q$  (see \cite[Corollary 1.2]{Agol_Virtual_Haken} for instance), so there are two finite index subgroups $P_1 <P$ and $Q_1<Q$ such that $P_1\cap Q_1=I$, $\langle P_1,Q_1 \rangle \cong P_1\ast_I Q_1$, and $\langle P_1,Q_1 \rangle$  is a stable subgroup by Corollary~\ref{cor:separable_intersection}.  In addition to $I$ being non-trivial, $\langle P_1, Q_1 \rangle$ is a one-end group as it is virtually a graph of closed surface groups with cyclic edge groups \cite[Theorem 18]{wilton_one_ended_graphs_of_free_groups}. 
\end{exmp}

When one of the subgroups is cyclic, we can separate the subgroup from short elements by raising the generator to sufficiently high powers. In Section \ref{sec:trivial_examples}, we will employ this idea to show that all Morse local-to-global groups that contain a Morse element  are either virtually cyclic or contain a stable free subgroup of rank 2. Here, we employ this trick to generalize a theorem of Arzhantseva \cite[Theorem 1]{Arzhantseva_quasiconvex_subgroups}.

\begin{cor}\label{cor: extending_stable_subgroups}
Let $G$ be a torsion free, Morse local-to-global group. If $Q$ is a non-trivial infinite index stable subgroup of $G$, then there is an infinite order element $h$ such that $\langle Q,h\rangle \cong Q \ast \langle h \rangle$ and $\langle Q,h\rangle$ is stable in $G$.     
\end{cor}

\begin{proof}
  By Theorem \ref{thm:finite_width}, there exists $g \in G$ such that $Q \cap gQg^{-1} = \{e\}$. Let $k$ be a non-identity element of $Q$. Since $Q$ is hyperbolic, $\langle gkg^{-1}\rangle$ is undistorted in $gQg^{-1}$ implying $\langle gkg^{-1}\rangle$ is  stable in $G$. Let $P$ be the commensurator of $\langle gkg^{-1}\rangle$ in $G$. By  Theorem \ref{thm:finite_width},  $\langle gkg^{-1}\rangle$ has finite index in $P$ implying $P$ is stable. 

We now prove that $P$ is malnormal. If $u \in G$ such that $uPu^{-1}\cap P \neq \{e\}$, then $uPu^{-1}\cap P$ is infinite as $G$ is torsion free. Since $\langle gkg^{-1}\rangle$ is finite index in $P$, $u\langle gkg^{-1}\rangle u^{-1}\cap \langle gkg^{-1}\rangle$ is finite index in $uPu^{-1}\cap P$. Therefore, $u\langle gkg^{-1}\rangle u^{-1}\cap \langle gkg^{-1}\rangle$ is infinite and finite index in both cyclic subgroups $u\langle gkg^{-1}\rangle u^{-1}$ and $\langle gkg^{-1}\rangle$. Thus, $u$ is an element of $P$, the commensurator of $\langle gkg^{-1}\rangle$. This proves $P$ is malnormal.  

By Theorem~\ref{thm:stable_subgroup_combination}.(\ref{item:combination_thrm_2}) the subgroup generated by $Q$ and $h=gk^ng^{-1}$ for $n$ large enough is then stable in $G$  and isomorphic to $Q\ast \langle h \rangle$. 
\end{proof}

Our final application of our combination theorem is to show that every normal subgroup of a Morse local-to-global group contains a Morse element. For the mapping class group of  a surface more complex than a one-holed torus or a $\CAT(0)$ group, the Morse elements are respectively the pseudo-Anosov and rank-1 elements. Thus, Corollary \ref{cor:Morse elements in normal subgroups} shows every infinite normal subgroup of such a mapping class group or $\CAT(0)$ group contains a pseudo-Anosov or rank-1 element respectively.

\begin{cor}\label{cor:Morse elements in normal subgroups}
Let $G$ be a finitely generated group that contains an infinite order Morse element. If $G$ has the Morse local-to-global property, then every infinite normal subgroup of $G$ contains an infinite order Morse element of $G$. 
\end{cor}

\begin{proof}
Let $h\in G$ be an infinite order Morse element and $N$ an infinite normal subgroup of $G$. Assume that no positive power of $h$ is an element of $N$. 
The subgroup $H = \langle h \rangle$ is an infinite index stable subgroup of $G$ because $h$ is a Morse element. Since $H$ is finite index in its commensurator in $G$ (Theorem \ref{thm:finite_width}) and $N$ is infinite, there exists an infinite number of distinct left cosets of $H$ with representatives in $N$. 
Theorem \ref{thm:finite_width} thus provides $g\in N$ such that $H\cap gHg^{-1} = \{e\}$.
By Theorem \ref{thm:stable_subgroup_combination}.(\ref{item:combination_thrm_1}), there is $n>0$ so that the elements $h^{n}$ and $gh^{n}g^{-1}$ generate a stable free subgroup. 
 Since $N$ is a normal subgroup and $g \in N$, $(h^{-n}gh^{n})g^{-1}$ is an infinite order Morse element of $N$.
\end{proof}

\begin{rem}
By \cite[Lemma 3.25]{DMS_divergence}, if the subgroup $N$ is also finitely generated, then the Morse element of $G$ it contains is also a Morse element of $N$ with respect to the word metric on $N$.
\end{rem}

\subsection{Translation lengths for Morse elements}\label{sec:translation_length}

We now generalize a result stated by Gromov and proved by Delzant on the discreteness of the set of algebraic translation lengths of elements of hyperbolic groups. We show that in the case of Morse local-to-global groups, the same result applies to conjugacy classes with a fixed Morse gauge.

\begin{defn}
Let $G$ be a group with finite generating set $S$. The \emph{algebraic translation length} of $g\in G$ is defined to be $$\tau_{G,S}(g)=\lim_{n\to \infty} \frac{|g^n|}{n}.$$
The limit $\displaystyle \lim_{n\to \infty} \frac{|g^n|}{n}$ always exists, because the function $n\mapsto |g^n|$ is sub-additive.
\end{defn}

For a given finite generating set, the algebraic translation length depends only on the conjugacy class of  an element. While being Morse is also a conjugacy class invariant, the specific Morse gauge of an element is not. Thus, we need to define a useful definition of Morse gauge for a conjugacy class of Morse elements. We do so by taking the Morse gauge for the elements in the conjugacy class with the shortest word length.

\begin{defn}[Morse gauge of a conjugacy class]
Let $G$ be a finitely generated group with a finite generating set $S$. For each element $g\in G$, let $[g]$ denote the conjugacy class of $g$ and $\mathrm{short}([g])$ be the collection of group elements in $[g]$ with minimal length with respect to $S$. The conjugacy class $[g]$ is \emph{$M$--Morse} with respect to $S$ if every element of $\mathrm{short}([g])$ is $M$--Morse.
\end{defn}

The key use of the Morse local-to-global property in establishing discreteness of translation length is Lemma \ref{lem:uniform_embedding_morse} below, which says $M$--Morse elements with minimal word length in their conjugacy class have a Morse quasi-axis with constants depending only on $M$.

\begin{defn}[Quasi-axis]
\label{defn:quasi-axis}
Let $G$ be a finitely generated group with a finite generating set $S$. The \emph{quasi-axis} for group element $g \in G$ and a geodesic $\alpha$ from $e$ to $g$ in $\cay(G,S)$ is the path $p_g \colon (-\infty,\infty) \to \cay(G,S)$ such that for all $n \in \Z$, $p_g$ restricted to the interval $[(n-1)|g|,n|g|]$ is~$g^n \alpha$.

\end{defn}

\begin{lem}\label{lem:uniform_embedding_morse}
Let $G$ be a finitely generated group with the Morse local-to-global property and $S$ be a finite generating set for $G$. Let $g\in G$ be an infinite order element such that $|g|$ is minimal in $[g]$. If $g$ is $M$--Morse with respect to $S$, then any quasi-axis $p_g$ is an $(N;\lambda,\epsilon)$--Morse quasi-geodesic in $\cay(G,S)$, where $\lambda$, $\epsilon$, and $N$  depend only on $G$, $S$, and $M$.
\end{lem}

\begin{proof}

Let $\delta = 4M(3,0)$ be the constant so that geodesic triangles with two $M$--Morse sides are $\delta$--slim (Lemma \ref{lem:Morse_polygons}). Let $M_0$ be the Morse gauge, depending only on $M$, so that any concatenation of $M$--Morse geodesics satisfying the hypothesis of Lemma \ref{lem:concatenation_of_morse_is_quasi_geodesic} is an $(L;M_0;1,4\delta +4)$--local Morse quasi-geodesic whenever the concatenation is an $(L;1,4\delta +4)$--local quasi-geodesic. Let $L$, $\lambda_0$, $\epsilon_0$ be the constants and $N_0$ the Morse gauge such that every $(L;M_0;1,4\delta +4)$--local Morse quasi-geodesic in $\cay(G,S)$ is a global $(N_0;\lambda_0,\epsilon_0)$--Morse quasi-geodesic in $\cay(G,S)$. 

Let $\alpha$ be any geodesic from $e$ to $g$ and $p = p_g$ be the quasi-axis for $g$ and $\alpha$ as described in Definition \ref{defn:quasi-axis}. We  first use the Morse local-to-global property to show that if $|g|>L$, then the quasi-axis $p$ is an $(N_0;\lambda_0,\epsilon_0)$--Morse quasi-geodesic. 

\begin{claim}\label{claim:translation_length}
If $|g| >L$, then $p$ is an $(L;M_0;1,4\delta+4)$--local Morse quasi-geodesic.
\end{claim}

\begin{subproof}
We first show $p$ is an $(L;1,4\delta+4)$--local quasi-geodesic.

Suppose $|g|>L$ and let $t_1,t_2\in (-\infty,\infty)$ with $t_1<t_2$ and $|t_2-t_1|<L$. Since the length of 
$\alpha$ is greater than $L$,  $p\bigl([t_1,t_2]\bigr)$ must be contained in either $g^n \alpha$  or  $g^n\alpha \cup g^{n+1}\alpha$ for some $n \in \mathbb{Z}$. If $p\bigl([t_1,t_2]\bigr) \subseteq g^n \alpha$, then $p\vert_{[t_1,t_2]}$ is an $M$--Morse 
geodesic, so suppose $p\bigl([t_1,t_2]\bigr) \not\subseteq g^n\alpha$ for any $n \in \mathbb{Z}$. Without loss of generality, we can assume $p\bigl([t_1,t_2]\bigr) \subseteq g^{-1}\alpha \cup \alpha$.

Let $s_1,s_2 \in [t_1,t_2]$. Since $p\vert_{[t_1,t_2]}$ is a concatenation of geodesics, the inequality $d\bigl(p(s_1),p(s_2)\bigr)\leq |s_2-s_1|$ holds by the triangle inequality. For the other inequality we can assume that $p(s_1) \in g^{-1} \alpha$ and $p(s_2) \in \alpha$. 

Let $\eta$ be a geodesic in $\cay(G,S)$ connecting $p(s_1)$ and $p(s_2)$, $\gamma_1$ be the geodesic segment of $g^{-1} \alpha$ connecting $e$ and $p(s_1)$, and  $\gamma_2$ be the geodesic segment of $\alpha$ connecting $e$ and $p(s_2)$. Since the triangle $\gamma_1 \cup \eta \cup \gamma_2$ is $\delta$--slim and the sum of the lengths of $\gamma_1$ and $\gamma_2$ is less than $|g|$, there are group elements $z \in \eta$ and $v_1, v_2, g_0 \in G$ satisfying the following:
\begin{enumerate}
    \item $v^{-1}_1$ is a vertex of $\gamma_1$ and $v_2$ is a vertex of $\gamma_2$;
    \item $d(z,v^{-1}_1)\leq \delta +1$ and $d(z,v_2) \leq \delta +1$;
    \item $g = v_2g_0v_1$ with $|g|=|v_2|+|g_0|+|v_1|$.
\end{enumerate}

\noindent Now, $v_2^{-1} g v_2 = g_0v_1v_2$, which implies $|v_2^{-1} g v_2| \leq |g_0| +|v_1v_2|\leq |g_0|+2\delta+2.$ Since $g$ is an element of minimal length in the conjugacy class, we have
\[|g| = |v_1| + |g_0| + |v_2| \leq |v_2^{-1} g v_2| \leq |g_0| +2\delta+2.\] Thus, $|v_1| + |v_2| \leq  2\delta+2$.  We now finish establishing that $p\vert_{[t_1,t_2]}$ is a $(1,4\delta +4)$--quasi-geodesic with the following calculation.
\begin{align*}
    d\bigl(p(s_1),p(s_2)\bigr) &=d\bigl(p(s_1),z\bigr)+d\bigl(z,p(s_2)\bigr)\\&\geq d\bigl(p(s_1),v^{-1}_1\bigr) + d(p(s_2),v_2\bigr) - 2 \delta -2\\ &=   \bigl(d\bigl(p(s_1),p(0)\bigr)-|v_1|\bigr) + \bigl(d\bigl(p(s_2),p(0)\bigr)-|v_2|\bigr) -2 \delta -2\\
    &\geq |s_2-s_1| - 4\delta -4.
\end{align*}

Since  $p$ is an $(L;1,4\delta +4)$--local quasi-geodesic, $p$ satisfies the hypothesis of Lemma \ref{lem:concatenation_of_morse_is_quasi_geodesic} as $\alpha$ has length more than $L$. Thus, $p$ is an $(L;M_0;1,4\delta +4)$--local Morse quasi-geodesic.
\end{subproof}

By the Morse local-to-global property of $\cay(G,S)$, Claim \ref{claim:translation_length} implies that $p=p_g$ is an $(N_0;\lambda_0,\epsilon_0)$--Morse quasi-geodesic when $|g| > L$.

To finish the proof in the remaining cases, recall  that $L$  depends only on $M$, $G$, and $S$. Therefore, the number of $M$--Morse elements with word length at most $L$ is bounded by a constant depending only on $M$, $G$, and $S$. This implies that there are constants $\lambda_1$, $\epsilon_1$ and a Morse gauge $N_1$ depending only on $M$, $G$, and $S$ such that $p_g$ is an $(N_1;\lambda_1,\epsilon_1)$--quasi-geodesic for each $M$--Morse element $g$ satisfying $|g|\leq L$. Thus, the lemma follows with  $\lambda=\max\{\lambda_0,\lambda_1\}$, $\epsilon=\max\{\epsilon_0,\epsilon_1\}$, and $N=\max\{N_0,N_1\}$.
\end{proof}

With Lemma \ref{lem:uniform_embedding_morse} in hand, we can apply Delzant's argument to show discreteness of translation length for conjugacy classes with a fixed Morse gauge.

\begin{thm}\label{thm:discrete_translation_lengths}
Let $G$ be a finitely generated group with the Morse local-to-global property. For all Morse gauges $M$ and finite generating sets $S$, the set of algebraic translation lengths of elements of $G$ whose conjugacy class is $M$--Morse with respect to $S$ is a discrete subset of the rational numbers.
\end{thm}

\begin{proof}
Let $[g]$ be an $M$--Morse conjugacy class in $G$ with respect to $S$ and let $u$ be the shortest representative of $[g]$. By Lemma~ \ref{lem:uniform_embedding_morse}, any quasi-axis $p_u$ is an $(N;\lambda,\epsilon)$--Morse quasi-geodesic where $N$, $\lambda$, and $\epsilon$ depend only on $M$, $G$, and $S$. Let $U_+$ and $U_-$ be the endpoint of $p_u$ in the Morse boundary of $\cay(G,S)$ (see \cite{MC1} for the definition of the Morse boundary). By \cite[Lemma 9]{MR_Morse_Boundary}, there exists $R >0$, depending on $M$, $G$, and $S$ such that every bi-infinite geodesic in $\cay(G,S)$ with endpoints $U_+$ and $U_-$ is contained in the $R$--neighborhood of $p_u$.
We can now follow the proof of \cite[Theorem III.H.3.17]{Bridson_Haefliger} verbatim. \end{proof}

\begin{rem}
Morse gauges come with a natural partial order: $M_1 \preceq M_2$ if  $M_1(\lambda,\epsilon) \leq M_2(\lambda,\epsilon)$ for all $(\lambda,\epsilon) \in [1,\infty) \times [0,\infty)$. If $M_1 \preceq M_2$, then every conjugacy class that is $M_1$--Morse is also $M_2$--Morse. Thus, Theorem \ref{thm:discrete_translation_lengths} can be rephrased as ``the set of translation lengths of all conjugacy classes that are at most $M$--Morse is discrete for every Morse gauge $M$".
\end{rem}

\subsection{A Cartan--Hadamard theorem for Morse local-to-global spaces}\label{sec:Cartan-Hadamard}

Our final application of the Morse local-to-global property is a short proof that hyperbolicity in Morse local-to-global spaces can be checked locally. 

\begin{defn}
A geodesic metric space $X$ is \emph{$(R;\delta)$--locally hyperbolic} if for each point $x\in X$ and for each triple of points $a,b,c \in \ball_R(x)$, any geodesic triangle with vertices $a$, $b$, and $c$ is $\delta$--slim.
\end{defn}

\begin{thm}\label{thm:cartan-hadamard}
Let $X$ be a metric space with the $\Phi$--Morse local-to-global property. For each $\delta \geq 0$, there exist $R=R(\delta, \Phi)$ and $\delta'=\delta'(\delta, \Phi)$ such that if $X$ is $(R;\delta)$--locally hyperbolic, then $X$ is globally $\delta'$--hyperbolic.
\end{thm}

To prove Theorem \ref{thm:cartan-hadamard}, we show  local hyperbolicity implies every geodesic is uniformly locally Morse. We can then apply the Morse local-to-global property to conclude that every geodesic is Morse. This implies hyperbolicity by Lemma \ref{lem:Morse_polygons}. We first record an auxiliary lemma.

\begin{lem}
\label{lem:mimic}
For each $\delta>0$ there is a Morse gauge $M$ so that the following holds. Let $X$ be an $(R;\delta)$--locally hyperbolic space and $x \in X$. Let $\alpha$ be a geodesic in $X$ and $\beta$ be a $(\lambda,\epsilon)$--quasi-geodesic with the same endpoints as $\alpha$. If $\alpha$ and $\beta$ are both contained $\ball_{R/4}(x)$, then the Hausdorff distance between $\beta$ and $\alpha$ is at most $M(\lambda,\epsilon)$.
\end{lem}

\begin{proof}
Any geodesic connecting two points in $\alpha\cup \beta$ is contained in $\ball_R(x)$. Since $X$ is $(R;\delta)$--locally hyperbolic, we can follow the proof of \cite[Theorem III.H.1.7]{Bridson_Haefliger} to find a Morse gauge for $\alpha$ depending only on $\delta$.
\end{proof}

\begin{proof}[Proof of Theorem \ref{thm:cartan-hadamard}]
Let $\delta >0$ and $M$ be the Morse gauge, depending only on $\delta$, from  Lemma~\ref{lem:mimic}. By increasing $M$, we can assume that $M(\lambda, \epsilon) \geq \lambda^2(\lambda+2\epsilon)+\epsilon +1$ for all $\lambda\geq 1$ and $\epsilon\geq 0$. Let $N$ be the Morse gauge and $L$, $k$, $c$ be the constants such that every $(L;M;1,0)$--local Morse quasi-geodesic in $X$ is a global $(N;k,c)$--Morse quasi-geodesic, i.e., $(L,N,k,c) = \Phi(M,1,0)$. Let $R=8L^3+4L+1$ and assume that $X$ is $(R;\delta)$--locally hyperbolic. If we show that every geodesic of $X$ is an $(L;M;1,0)$--Morse quasi-geodesic, then every geodesic of $X$ is $N$--Morse. By Lemma \ref{lem:Morse_polygons}, this implies that $X$ is $\delta'$--hyperbolic where $\delta' = 4N(3,0)$.

Let $\gamma$ be an arbitrary geodesic in $X$ and $\alpha \colon[a,b]\to X$ be a $(\lambda,\epsilon)$--quasi-geodesic with endpoints on $\gamma$ such that $d\bigl(\alpha(a),\alpha(b)\bigr)<L$. This forces $|a-b|\leq \lambda(L+\epsilon)$. 

First, assume $L \leq \lambda+\epsilon$. For each $t\in[a,b]$ we have \[d\bigl(\alpha(t),\alpha(a)\bigr)\leq \lambda|t-a|+\epsilon\leq \lambda^2(\lambda+2\epsilon)+\epsilon< M(\lambda,\epsilon).\]
Therefore, the Hausdorff distance between $\alpha$ and the subsegment of $\gamma$ between $\alpha(a)$ and $\alpha(b)$ is bounded by $M(\lambda,\epsilon)$.

Now, assume $L > \lambda+\epsilon$. Therefore, $|a-b|\leq 2L^2$, and we have
\[d\bigl(\alpha(t),\alpha(a)\bigr)\leq \lambda |t-a|+\epsilon\leq 2 L^3+L<R/4 \text { for all } t\in[a,b].\]
Thus, $\alpha$ and the subsegment of $\gamma$ between $\alpha(a)$ and $\alpha(b)$ both lie in the ball $\ball_{R/4}\bigl(\alpha(a)\bigr)$. The Hausdorff distance between $\alpha$ and the geodesic segment of $\gamma$ between $\alpha(a)$ and $\alpha(b)$ is then at most $M(\lambda,\epsilon)$ by Lemma \ref{lem:mimic}. 

This verifies that the geodesic $\gamma$ is an $(L;M;1,0)$--local Morse quasi-geodesic. Since $X$ has the Morse local-to-global, $\gamma$ is $N$--Morse. Hence,  $X$ is $\delta'$--hyperbolic where $\delta' = \delta'(\delta,\Phi)$. 
\end{proof}

\section{Spaces with the Morse local-to-global property}\label{sec:examples}
In this section, we prove unconstricted spaces, $\CAT(0)$ spaces, and hierarchically hyperbolic spaces with a technical condition have the Morse local-to-global property. This covers all of the examples given in Theorem  \ref{intro_thm:examples_of_local_to_global}, except the fundamental groups of closed $3$--manifolds with Nil or Sol components in their prime decomposition. The proof for all closed $3$--manifold groups is contained in the next section on relative hyperbolicity. In section \ref{sec:trivial_examples}, we also give examples of groups that contain infinite Morse quasi-geodesics, but do not have the Morse local-to-global property. As we will not be needing the full breath of the definitions, we opt to give the salient properties of each class of spaces that we need and forgo giving the definitions in full. We direct the reader to the following sources for detailed background on the spaces in question: unconstricted spaces \cite{DrutuSapir}, $\CAT(0)$ spaces \cite{Bridson_Haefliger}, and hierarchically hyperbolic spaces \cite{BHS_HHSII}.

\subsection{Spaces that trivially satisfy the definition and non-examples.}\label{sec:trivial_examples}

The simplest examples of non-hyperbolic spaces with the Morse local-to-global property are spaces, like the Euclidean plane, where the Morse quasi-geodesics have uniformly bounded domains. We call these spaces \emph{Morse limited}.

\begin{defn}
A metric space $X$ is \emph{Morse limited} if for  every  Morse gauge $M$ and $\lambda \geq 1, \epsilon \geq 0$ there exists  $B \geq 0$ so that every $(M;\lambda, \epsilon)$--Morse quasi-geodesic $\gamma \colon I\to X$ has $\diam(I) \leq B$. A finitely generated group $G$ is \emph{Morse limited} if  $\cay(G,S)$ is Morse limited for some finite generating set $S$.
\end{defn}

Since Morse quasi-geodesics are preserved by quasi-isometries, being Morse limited is a quasi-isometry invariant. 

When a space is Morse limited, it trivially satisfies the Morse local-to-global property by choosing the local scale to be larger than the bound on the length of the domain of a Morse quasi-geodesic. 
\begin{lem}\label{lem: trivially ltg}
Let $X$ be a metric space. If $X$ is Morse limited, then $X$ has the Morse local-to-global property.
\end{lem}

\begin{proof}
Let $B \geq 0$ so that the domain of any $(M;\lambda,\epsilon)$--Morse quasi-geodesic has diameter at most $B$. Let $L > B$ and suppose $\gamma \colon I \to X$ is an $(L;M;\lambda,\epsilon)$--local Morse quasi-geodesic. If $\diam(I) > B$, then there exists $[a,b]\subseteq I$ with $B<|a-b|\leq L$. However, this is a contradiction as $\gamma \vert_{[a,b]}$ would then be an $(M;\lambda,\epsilon)$--Morse quasi-geodesic with domain longer than $B$. Thus, $\diam(I) \leq B <L$ and $\gamma$ is an $(M;\lambda,\epsilon)$--Morse quasi-geodesic.
\end{proof}

While Lemma \ref{lem: trivially ltg} is a simple observation, it plays a key role when combined with the relative hyperbolicity results of Section \ref{sec:relativelyhyperbolic}. In particular, the fundamental groups of all closed 3--manifolds will have the Morse local-to-global property as they are hyperbolic relative to subgroups that have the Morse local-to-global property either by being hierarchically hyperbolic space or by being Morse limited. 

A large class of examples of Morse limited spaces come from the class of unconstricted groups introduced by Dru\c{t}u and Sapir.

\begin{defn}[{\cite{DrutuSapir}}]
A finitely generated group $G$ is \emph{unconstricted} if for any finite generating set $S$, there exists an asymptotic cone of $\cay(G,S)$ that does not contain a cut-point.
\end{defn}

Drutu and Sapir demonstrate several classes of groups are unconstricted. The most interesting ones for us will be solvable groups, which satisfy a law by virtue of being uniformly amenable \cite[Corollary 6.17]{DrutuSapir}.

\begin{thm}[{\cite[Theorem 6.5, Corollary 6.14]{DrutuSapir}}]
Let $G$ be finitely generated group that is not virtually cyclic. If $G$ either satisfies a law or has  a central, infinite cyclic subgroup, then  $G$ is unconstricted.
\end{thm}

Examples of groups that are not unconstricted included those that contain a bi-infinite Morse geodesics.

\begin{lem}[{A consequence of \cite[Proposition 1]{DMS_Corrigendum}}]
Let $G$ be a group generated by a finite generating set $S$. If $\cay(G,S)$ contains a bi-infinite Morse geodesic, then for every finite generating set $T$, every asymptotic cone of $\cay(G,T)$ contains a cut-point.
\end{lem}

\begin{prop}\label{prop: unconstricted implies trivial ltg}
 If a finitely generated group $G$ is unconstricted, then $G$ is Morse limited.
\end{prop}
\begin{proof}
We prove the contrapositive
Let $S$ be a finite generating set for $G$ and suppose there is a sequence, $\gamma_n \colon I_n \to \cay(G,S)$, of  $(M; \lambda, \epsilon)$--Morse quasi-geodesics with $\diam(I_n) \to \infty$ as $n \to \infty$. By Lemma \ref{lem:close_to_Morse}, there exists a Morse gauge $M'$ so that $\cay(G,S)$ contains $M'$--Morse geodesics of arbitrarily long length. This implies $G$ is not unconstricted as $\cay(G,S)$ will contain an infinite $M'$--Morse geodesic by applying Lemma \ref{lem:long_Morse_segments_imply_Morse_ray}
\end{proof}

Morse limited groups give a structural dichotomy to Morse local-to-global group: either the group is Morse limited or it contains a Morse element.

\begin{thm}\label{thm:trivial_Morse_element_dichotomy}
Let $G$ be a finitely generated group with Morse local-to-global property. Either $G$ contains a Morse element or $G$ is Morse limited. 
\end{thm}

Since Lemma \ref{lem:long_Morse_segments_imply_Morse_ray} guarantees that a group containing arbitrarily long Morse geodesics must contain an infinite Morse geodesic, Theorem \ref{thm:trivial_Morse_element_dichotomy} follows immediately from the next proposition plus Lemma~\ref{lem:long_Morse_segments_imply_Morse_ray}.

\begin{prop}\label{prop:Morse_ray_imlies_Morse_element}
Let $G$ be a finitely generated group with the Morse local-to-global property. If there is some finite generating set $S$ of $G$ such that the Cayley graph $\cay(G,S)$ contains a Morse geodesic ray, then $G$ contains a Morse element.
\end{prop}

\begin{proof}
Let $\alpha \colon [0,\infty)\to \cay(G,S)$ be an $M$--Morse geodesic ray in $\cay(G,S)$. We can assume that $\alpha(0) = e$. Since $\cay(G,S)$ has the Morse local-to-global property, there is a positive integer $L$, real numbers $\lambda\geq 1$, $\epsilon \geq 0$, and Morse gauge $N$ such that any $(L;M;1,0)$--local Morse quasi-geodesic in $\cay(G,S)$ is an $(N;\lambda,\epsilon)$--Morse quasi-geodesic. 

For each integer $i\geq 0$, let $\alpha_i$ be the subsegment of $\alpha$ from $\alpha(i\cdot 2L)$ to $\alpha\bigl((i+1)\cdot 2L\bigr)$. There are two non-negative integers $m$ and $n$ such that $n-m\geq 2$ and $\alpha_m$ and $\alpha_n$ are both labeled by same word $w_1$ in $S$. Let $w_2$ be the word in $S$ labeling the subsegment of $\alpha$ from the end of $\alpha_m$ to the beginning of $\alpha_n$. By construction, $|w_1|,|w_2| \geq 2L$ and any geodesic labeled by any subword of $w_1w_2$ or $w_2w_1$ is an $M$--Morse geodesic.

Let $\ell=|w_1|+|w_2|$ and define $\beta \colon (-\infty,\infty)\to \cay(G,S)$ to be the path such that $\beta(0)=e$ and for each $i \in \mathbb{Z}$, $\beta\vert_{[i\ell,(i+1)\ell]}$ is labeled by the word $w_1w_2$. Then, $\beta$ is an $(L;M;1,0)$--local Morse quasi-geodesic and hence an $(N;\lambda,\epsilon)$--Morse quasi-geodesic. This implies that $w_1w_2$ represent a Morse element of $G$.  
\end{proof}

By applying our combination theorem for stable subgroups (Theorem \ref{thm:stable_subgroup_combination}), we  expand Theorem \ref{thm:trivial_Morse_element_dichotomy} to show that Morse local-to-global groups that are not Morse limited are either virtually cyclic or contain a stable free subgroups of rank 2.

\begin{cor}\label{cor:trichotomy}
If $G$ is a finitely generated group with the Morse local-to-global property, then $G$ satisfies exactly one of the following:
\begin{enumerate}
    \item $G$ is Morse limited;
    \item $G$ is virtually infinite cyclic;
    \item $G$ contains a stable, free subgroup of rank $2$.
\end{enumerate}
\end{cor}

\begin{proof}
Assume that $G$ is not Morse limited and not virtually a cyclic group.  By Theorem~\ref{thm:trivial_Morse_element_dichotomy}, $G$ must contain a Morse element $h$. The subgroup $H = \langle h\rangle$ is then an infinite index stable subgroup of $G$. Thus, there is an element $g \in G$ such that $gHg^{-1}\cap H =\{e\}$ (Theorem \ref{thm:finite_width}). By Theorem~\ref{thm:stable_subgroup_combination}, there is $n>0$ such that the elements $h^{n}$ and $gh^{n}g^{-1}$ generate a stable free subgroup of rank $2$.
\end{proof}

Theorem \ref{thm:trivial_Morse_element_dichotomy} and Corollary \ref{cor:trichotomy} allow us to give examples of finitely generated groups that contain bi-infinite Morse geodesics (and even Morse elements), but do not have the Morse local-to-global property.

\begin{exmp}[Groups that are not Morse local-to-global]\label{ex:non-examples}
The following finitely generated groups do not have the Morse local-to-global property.
\begin{enumerate}
    \item In \cite{Fink_Torsion}, Finks the author gives an example of a finitely generated, torsion group that contains a bi-infinite Morse geodesic. Theorem \ref{thm:trivial_Morse_element_dichotomy} shows this groups is not Morse local-to-global as it is not Morse limited and does not contain a Morse element.
    \item  In \cite[Theorem 1.12]{OOS_Lacunary_hyperbolic_groups}, Osin, Ol'shanskii, and Sapir produce a finitely generated, non-virtually cyclic group $G$ where every proper, non-trivial subgroup is infinite cyclic and stable. Every non-identity element of $G$ is Morse, but $G$ does not contains a free subgroup of rank 2. Thus, $G$ is not Morse local-to-global by Corollary \ref{cor:trichotomy}.
\end{enumerate}
\end{exmp}

Since the direct product of any group with $\mathbb{Z}$ is unconstricted and hence Morse limited, the above examples show the Morse local-to-global property does not descend to undistorted finitely generated subgroups.

\subsection{$\CAT(0)$ spaces}

The main tool we need to establish the Morse local-to-global property in $\CAT(0)$ spaces is the closest point projection onto geodesics. For the remainder of this section $\pi_\gamma$ will denote the closest point projection onto the geodesic $\gamma$ described in the next lemma.

\begin{lem}[{\cite[Proposition 2.4, Corollary 2.5]{Bridson_Haefliger}}]\label{lem:CAT(0)_closest_projection}
Let $X$ be a $\CAT(0)$ space and $\gamma$ be a geodesic in $X$.
\begin{enumerate}
    \item \label{item:definition}  There exists a continuous function $\pi_\gamma \colon X \to \gamma$ so that for all $x \in X$, $\pi_\gamma(x)$ is the unique point in $\gamma$ that minimizes the distance from $x$ to $\gamma$.
    \item \label{item:convexity} For any geodesic $\eta \colon [a,b] \to X$, \[d(\eta(t), \pi_\gamma(\eta(t)) \leq \max \{ d(\eta(a), \pi_\gamma(\eta(a)), d(\eta(b),\pi_\gamma(\eta(b))\}\] where $\pi_\gamma$ is the map from (\ref{item:definition}).
\end{enumerate}
\end{lem}
 
 We  also need a characterization of Morse geodesics in $\CAT(0)$ spaces in term of the closest point projection established by Charney and Sultan. 

\begin{defn}\label{defn:contracting}
Let $D\geq 0$. A geodesic $\gamma$ in a $\CAT(0)$ space $X$ is \emph{$D$--contracting} if for all $x ,y \in X$ with $d(x,y) < d(x,\gamma)$, the distance $d\bigl(\pi_\gamma(x), \pi_\gamma(y) \bigr)$ is at most $D$.
\end{defn}

\begin{thm}[{\cite[Theorem 2.9]{Charney_Sultan_CAT(0)}}]\label{thm:CAT(0)_Morse_equiv_contracting}
Let  $\gamma$ be a geodesic in a $\CAT(0)$ space. 
\begin{enumerate}
    \item For each Morse gauge $M$, there exists $D=D(M) \geq 0$ so that if $\gamma$ is $M$--Morse, then $\gamma$ is $D$--contracting.
    \item For each $D \geq 0$, there exists a Morse gauge $M = M(D)$ so that if $\gamma$ is $D$--contracting, then $\gamma$ is $M$--Morse.
\end{enumerate}
\end{thm}

Finally, we need a lemma of the third author about concatenating geodesics in any geodesic metric space. 

\begin{lem}[{\cite[Lemma 3.6]{Tran2017}}]\label{lem:uniform_quasi-geodesic}
Let $\gamma=\gamma_1\ast \gamma_2 \ast \gamma_3$ be the concatenation of the geodesics $\gamma_1$, $\gamma_2$, and $\gamma_3$ in a geodesic space. If the length of $\gamma_2$ is more than $21$ times the sum of the lengths of $\gamma_1$ and $\gamma_3$, then there is a $(5,0)$--quasi-geodesic $\alpha$ with the same endpoints of $\gamma$ and $\alpha\cap\gamma_2\neq \emptyset$. 
\end{lem}

We prove the Morse local-to-global property. First, we use Lemma \ref{lem:CAT(0)_closest_projection}.(\ref{item:convexity}) to show that a local Morse quasi-geodesic of sufficient scale is close to the geodesic between its end points (Proposition \ref{prop:CAT(0)_Key}). This implies the local Morse quasi-geodesic is a global quasi-geodesic by Lemma \ref{lem:local_close_to_global}. Next, we use the contracting characterization of Morse geodesics to prove this global quasi-geodesic is Morse (Theorem \ref{thm:CAT(0)_are_local_to_global}).

\begin{prop}\label{prop:CAT(0)_Key}
Let $X$ be a $\CAT(0)$ space. For each $\lambda\geq 1$, $\epsilon\geq 0$, and Morse gauge $M$, there are constants $\ell\geq 0$ and $C\geq 0$ such that the following holds. If $L \geq \ell$ and $\gamma \colon [a,b] \to X$ is an $(L;M;\lambda,\epsilon)$--local Morse quasi-geodesic and $\beta$ is the geodesic connecting $\gamma(a)$ and $\gamma(b)$, then the Hausdorff distance between $\gamma$ and $\beta$ is less than  $C$.
\end{prop}

\begin{proof}
For any two points $x,y \in X$, let $[x,y]$ denote the unique geodesic from $x$ to $y$ in the $\CAT(0)$ metric on $X$. Let $M'$ be the Morse gauge  provided by Lemma \ref{lem:close_to_Morse} so that any geodesic connecting the endpoints of an $(M;\lambda,\epsilon)$--Morse quasi-geodesic is $M'$--Morse. Let $K_1 = M(1,0)$, $K_2 = M'(5,0)$, and $R = 2(K_2+K_1 +1)$. Let $L\geq \ell = 4\lambda(44R+\epsilon+K_1)$.

Let $\gamma \colon [a,b] \to X$ be an $(L;M;\lambda,\epsilon)$--local Morse quasi-geodesic and $\beta$ be the geodesic connecting $\gamma(a)$ and $\gamma(b)$. We first show that $\gamma \subseteq \mc{N}_{C_1}(\beta)$ for some $C_1$ depending only on $M$, $\lambda$, and $\epsilon$.

Let $t\in [a,b]$ such that \[d\bigl(\gamma(t),\beta\bigr)+1>\sup\limits_{s\in [a,b]} d\bigl(\gamma(s),\beta\bigr).\] 

If $|a-t|<\ell/4$ or $|t-b|<\ell/4$, then the distance $d\bigl(\gamma(t),\beta\bigr)$ is bounded by a constant depending only on $\lambda$, $\epsilon$, and $M$.
Assume $|a-t|\geq \ell/4$ and $|t-b|\geq \ell/4$. Let $x=\gamma(t-\ell/4)$, $y=\gamma(t+\ell/4)$, and $\beta_1$ be the geodesic connecting $x$ and $y$. Since $\gamma$ is locally $M$--Morse, the Hausdorff distance between $\beta_1$  and  $\gamma\vert_{[t-\ell/4,t+\ell/4]}$ is at most $K_1 = M(1,0)$. Let $z$ be a point in $\beta_1$ such that $d\bigl(\gamma(t),z\bigr)<K_1$.  Note, both geodesic segments $[x,z]$ and $[z,y]$ are $M'$--Morse. 

\begin{claim}
 $d(z,\beta) < 2R = 4(M(1,0)+M'(5,0) +1)$.
\end{claim}
\begin{subproof}
 Assume for the purposes of contradiction that $d(z,\beta)\geq 2R$. We shall create a contradiction to Lemma \ref{lem:CAT(0)_closest_projection}.(\ref{item:convexity}) by showing there  is $w_1 \in [x,z]$ and $w_2 \in [z,y]$ such that the distances $d(w_1,\beta)$ and $d(w_2,\beta)$ are both strictly less than the distance $d(z,\beta)$. We only give the proof for the existence of $w_1$ as the proof for the existence of $w_2$ is entirely analogous.
 
 If $d(x,\beta)<R$, then we can choose $w_1=x$. Otherwise, let $x_1$ and $z_1$ be the points in $X$ such that \[d(x,x_1)=R; \quad \ d(x_1,\beta)=d(x,\beta)-R \] and  \[ d(z,z_1)=R; \quad d(z_1,\beta)=d(z,\beta)-R.\]
 Since $\gamma$ is a local quasi-geodesic, we have \[d(x,z)\geq d\bigl(\gamma(t-\ell/4),\gamma(t)\bigr)-d\bigl(\gamma(t),z\bigr)\geq \left(\frac{\ell}{4\lambda}-\epsilon\right)-K_1= 44R.\] 
 This implies,
\[d(x_1,z_1)\geq d(x,z)-d(x,x_1)-d(z,z_1)\geq 42R.\]

We can now apply Lemma~\ref{lem:uniform_quasi-geodesic} to the concatenation of the geodesics $[x,x_1]$, $[x_1,z_1]$, and $[z_1,z]$ to produce a $(5,0)$--quasi-geodesic $\beta_2$ connecting $x$ and $z$ with $\beta_2\cap [x_1,z_1] \neq \emptyset$. Since $[x,z]$ is $M'$--Morse,  any point in $\beta_2$ lies in the $M'(5,0)$--neighborhood of $[x,z]$. Therefore, there is $w_1 \in [x,z]$ and $u \in \beta_2\cap [x_1,z_1]$ such that $d(w_1,u)<K_2 = M'(5,0)$.

By Lemma \ref{lem:CAT(0)_closest_projection}.(\ref{item:convexity}), we have $d(u,\beta)\leq d(x_1,\beta)$ or $d(u,\beta)\leq d(z_1,\beta)$. In either case, we have  $d(u,\beta)\leq d(z,\beta)+1+K_1-R$ as 
\[d(x_1,\beta)=d(x,\beta)-R\leq d\bigl(\gamma(t),\beta\bigr)+1-R\leq d(z,\beta)+1+K_1-R\]
and
\[d(z_1,\beta)=d(z,\beta)-R.\]
 This implies that
\[d(w_1,\beta)\leq d(w_1,u)+d(u,\beta)\leq K_2+ d(z,\beta)+1+K_1-R<d(z,\beta)\]
as desired.
\end{subproof}

Since the distance $d(z,\beta)$ is bounded above by $2R$, we have \[d\bigl(\gamma(t),\beta\bigr) \leq 2R + K_1 = 4\bigl(M(1,0)+M'(5,0)\bigr) + M(1,0).\] Therefore,  $\gamma$ lies in the $C_1$--neighborhood of $\beta$ where $C_1 = 2R+K_1+1$ depends only on $\lambda$, $\epsilon$, and~$M$.

We now prove that $\beta$ also lies in the $C_2$--neighborhood of $\gamma$ where $C_2$ depends only on $\lambda$, $\epsilon$, and $M$. Let $a=t_0<t_1<\cdots<t_n=b$ with $|t_{i+1}-t_i|<\ell$. Since each $\gamma\vert_{[t_i,t_{i+1}]}$ is a $(\lambda,\epsilon)$--quasi-geodesic, the distance between $\gamma(t_i)$ and $\gamma(t_{i+1})$ is bounded above by $\lambda\ell+\epsilon$. Let $u_0= \gamma(a)$, $u_n= \gamma(b)$, and for each $1\leq i \leq n-1$, let $u_i$ be a point in $\beta$ such that the distance between $\gamma(t_i)$ and $u_i$ is at most $C_1$. By the triangle inequality, the distance between $u_i$ and $u_{i+1}$ is at most $\lambda\ell+\epsilon+2C_1$. If $u\in \beta$, then $u$ must lie in a subsegment of $\beta$ with endpoints $u_i$ and $u_{i+1}$ for some $0\leq i < n$. Therefore, the distance between $u$ and $\gamma(t_i)$ is at most $\ell \lambda+\epsilon +3C_1 $.  This implies that $\beta$ lies in the $(\lambda\ell+\epsilon+3C_1)$--neighborhood of $\gamma$.
\end{proof}

\begin{thm}\label{thm:CAT(0)_are_local_to_global}
For each $\lambda\geq 1$, $\epsilon\geq 0$, and Morse gauge $M$, there are constants $L\geq 0$, $\lambda'\geq 1$, $\epsilon'\geq 0$, and a Morse gauge $M'$ such that the following holds. If $\gamma$ is an $(L;M;\lambda,\epsilon)$--local Morse quasi-geodesic in a $\CAT(0)$ space $X$, then $\gamma$ is an $(M';\lambda',\epsilon')$--Morse quasi-geodesic in $X$.
\end{thm}

\begin{proof}
Let $\ell$ and $C$ be the constants from Proposition \ref{prop:CAT(0)_Key} for $\lambda$, $\epsilon$, and $M$. Let $L \geq \lambda(3 C + \epsilon +2 +\ell) +1$ and let $\gamma \colon [a,b] \to X$ be an $(L;M;\lambda,\epsilon)$--local Morse quasi-geodesic in a $\CAT(0)$ space $X$.  By Proposition \ref{prop:CAT(0)_Key}, every subsegment of $\gamma$ in contained in the $C$--neighborhood of the geodesic between its endpoints. Since $L > \lambda(3 C + \epsilon +2 )$, Lemma \ref{lem:local_close_to_global} implies $\gamma$ is a $(\lambda',\epsilon')$--quasi-geodesic for $\lambda'$ and $\epsilon'$  depending ultimately only on $\lambda$, $\epsilon$, and $M$.

To show  $\gamma$ is Morse, consider the geodesic $\beta$ connecting $\gamma(a)$ and $\gamma(b)$. By Theorem \ref{thm:CAT(0)_Morse_equiv_contracting} and Lemma \ref{lem:close_to_Morse}, it suffices to show that $\beta$ is $D$--contracting for some $D$ depending only on $M$, $\lambda$, and $\epsilon$. Let $M'$ be the Morse gauge so that any geodesic with endpoints in the $C$--neighborhood of an $(M;\lambda,\epsilon)$--Morse quasi-geodesic is $M'$--Morse. Let $D>0$ so that any $M'$--Morse geodesic in $X$ is $D$--contracting and suppose $L > \lambda (D +2C + \epsilon +1)$. Note, such a $D$ depends ultimately only on $M$, $\lambda$, and $\epsilon$.

Suppose, for the purposes of contradiction, that $\beta$ is not $D$--contracting. Then, there are $x$ and $y$ in $X$ such that $d(x,y)<d\bigl(x,\pi_{\beta}(x)\bigr)$, but $d\bigl(\pi_{\beta}(x), \pi_{\beta}(y)\bigr) > D$.
Let $c = \min\{1,d\bigl(\pi_{\beta}(x), \pi_{\beta}(y)\bigr)-D\}$.
By the continuity of the projection map $\pi_{\beta}$, there is a point $y_1$ in the geodesic connecting $x$ and $y$ such that $d\bigl(\pi_{\beta}(x), \pi_{\beta}(y_1)\bigr)= D  + c/2$. Let $\beta'$ be the geodesic subsegment of $\beta$ connecting $\pi_{\beta}(x)$ and $\pi_{\beta}(y_1)$.
Then  $\pi_{\beta}(x)=\pi_{\beta'}(x)$, $\pi_{\beta}(y_1)=\pi_{\beta'}(y_1)$, and $d(x,y_1)\leq d(x,y)$. Therefore, $d(x,y_1)<d\bigl(x,\pi_{\beta'}(x)\bigr)$ and $d\bigl(\pi_{\beta'}(x), \pi_{\beta'}(y_1)\bigr)=d\bigl(\pi_{\beta}(x), \pi_{\beta}(y_1)\bigr) >D$. This implies that $\beta'$ is not $D$--contracting.
However, since $L > \lambda (D +2C + \epsilon +1)$, $\beta'$ is a geodesic with endpoints in the $C$--neighborhood of an $(M;\lambda,\epsilon)$--Morse quasi-geodesic and hence $D$--contracting. This contradiction implies $\beta$ must be $D$--contracting.
\end{proof}

\subsection{Hierarchically hyperbolic spaces and Morse detectability} \label{subsec:HHS}
For hierarchically hyperbolic spaces, the Morse local-to-global property follows from a result of Abbott, Behrstock, and Durham that established most hierarchically hyperbolic spaces admit a projection onto a hyperbolic space that detects when a quasi-geodesic is Morse. We prove that any space with this \emph{Morse detectability} property has the Morse local-to-global property. In addition to simplifying the proof in the case of hierarchically hyperbolic spaces, this approach also provides an avenue for producing new examples of Morse local-to-global spaces.

\begin{defn}\label{Morse detectable}
A metric space $X$ is \emph{Morse detectable} if there exists a $\delta$--hyperbolic space $Y$ and a coarsely Lipschitz map $\pi \colon X \to Y$ such that for every $(\lambda, \epsilon)$--quasi-geodesic $\gamma \colon [a,b] \to X$, the following holds.
\begin{enumerate}
   \item \label{item:Morse detectable 1} If $\gamma$ is $M$--Morse, then $\pi\circ \gamma$ is a $(k,c)$--quasi-geodesic in $Y$ where $(k,c)$ is determined by $\lambda$,  $\epsilon$, $\delta$, and $M$.
    \item \label{item:Morse detectable 2} If $\pi\circ \gamma$ is a $(k,c)$--quasi-geodesic in $Y$, then $\gamma$ is $M$--Morse, where $M$ is determined by $k$, $c$, $\lambda$, $\epsilon$, and $\delta$.
\end{enumerate}
\end{defn}

\begin{thm}
\label{thm:Morse local-to-globle for Morse detectable}
If $X$ is a Morse detectable metric space, then $X$ has the Morse local-to-global property.
\end{thm}

\begin{proof}
Let $Y$ be the $\delta$--hyperbolic space and $\pi \colon X \to Y$ be the coarsely Lipschitz map satisfying Definition~\ref{Morse detectable}. Fix $\lambda \geq 1$, $\epsilon \geq 0$, and Morse gauge $M$. 

Since $X$ is Morse detectable, there exists $k\geq 1$ and $c\geq 0$ depending on $\lambda$, $\epsilon$, $\delta$, and $M$ so that if $\eta$ is an $(M;\lambda,\epsilon)$--Morse quasi-geodesic, then $\pi\circ \eta$ is a $(k,c)$--quasi-geodesic. By the local-to-global property for quasi-geodesics in hyperbolic spaces (Theorem \ref{thm:local_to_global_hyperbolic}), there are constants $L> 2\epsilon$, $k'\geq 1$ and $c'\geq 0$ depending only on $k$, $c$, and $\delta$ such that any $(L;k,c)$--local quasi-geodesic in $Y$ is a $(k',c')$--quasi-geodesic. 

Let $\gamma \colon [a,b] \to X$ be an $(L;M;\lambda,\epsilon)$--local Morse quasi-geodesic. 
 By (\ref{item:Morse detectable 1}) of Morse detectability, $\pi\circ \gamma$ is an $(L;k,c)$--local quasi-geodesic in $Y$. Since $Y$ is $\delta$--hyperbolic,  $\pi\circ \gamma$ is a $(k',c')$--quasi-geodesic.   Using  Lemma \ref{lem: Folklore} below, this implies $\gamma$ is a $(\lambda',\epsilon')$--quasi-geodesic in $X$ where $\lambda'$ and $\epsilon'$ depend only on $\lambda$, $\epsilon$, $\delta$, $M$, and the coarse Lipschitz constants of $\pi \colon X \to Y$ . Applying (\ref{item:Morse detectable 2}) of Morse detectable makes $\gamma$ an $(M';\lambda',\epsilon')$--Morse quasi-geodesic in $X$ where $M'$ depends only on $M$, $\delta$, $\lambda$, and $\epsilon$.
\end{proof}

\begin{lem}\label{lem: Folklore}
Let $X, Y$ be metric spaces  and $\pi\colon X \to Y$ be a $(K,C)$--coarsely Lipschitz map. Let $\gamma \colon I \to X$ be an $(L;\lambda, \epsilon)$--local quasi-geodesic with $L> 2\epsilon$. For all $k \geq 1$ and $c \geq 0$, there exists $\lambda'$ and $\epsilon'$ so that if $ \pi \circ \gamma \colon I \to Y$ is a $(k,c)$--quasi-geodesic, then $\gamma$ is a $(\lambda',\epsilon')$--quasi-geodesic in $X$.
\end{lem}

\begin{proof}
Let $t_1,t_2 \in I$ with $t_1 < t_2$ and choose $t_1=s_0 < \cdots < s_n = t_2$ such that $L/2 \leq |s_i-s_{i+1}|\leq L$. Since $L > 2\epsilon$ and $\gamma$ is an $(L;\lambda, \epsilon)$--local quasi-geodesic, the triangle inequality gives 
\[d_X\bigl(\gamma(t_1),\gamma(t_2)\bigr) \leq \sum\limits_{i=0}^{n} d_X\bigl(\gamma(s_i),\gamma(s_{i+1})\bigr) \leq \sum\limits_{i=0}^{n} \bigl[\lambda |s_i - s_{i+1}| + \epsilon\bigr] \leq   (\lambda +1)\cdot |t_1 - t_2|.\]
Since $\pi$ is $(K,C)$--coarsely Lipschitz and  $\pi \circ \gamma$ is a $(k,c)$--quasi-geodesic, we have
\[\frac{1}{k}|t_1 - t_2| - c \leq d_Y\bigl(\pi \circ \gamma (t_1), \pi \circ \gamma (t_2)\bigr) \leq Kd_X\bigl(\gamma(t_1), \gamma(t_2)\bigr) +C \leq K(\lambda+1)\cdot|t_1 - t_2| +C. \qedhere\]
\end{proof}

Abbott, Behrstock, and Durham showed that all hierarchically hyperbolic spaces satisfying a minor technical condition are Morse detectable \cite{ABD}. This covers all of the natural examples of hierarchically hyperbolic spaces as outlined in Corollary \ref{cor:Examples_of_HHS}.  Abbott, Behrstock, and Durham provide an explicit description of the space and projection map for Morse detectability. In the case of the mapping class group, the space is the curve graph and the projection map is the subsurface projection of Masur and Minsky.

\begin{thm}[{\cite[Corollary 6.2, Theorem 7.2]{ABD}}]
\label{thm:HHS_are_Morse_detectable}
If $X$ is a hierarchically hyperbolic space with the bounded domain dichotomy, then $X$ is Morse detectable. In particular, all  hierarchically  hyperbolic groups are Morse detectable.
\end{thm}

\begin{cor}\label{cor:Examples_of_HHS}
The following groups and spaces have the Morse local-to-global property.
\begin{itemize}
    \item The mapping class group of an orientable, finite type surface.
    \item Graph products of hyperbolic groups.
    \item The Teichm\"uller space of an orientable, finite type surface with either the Teichm\"uller or Weil--Petersson metric.
    \item The fundamental group of a closed $3$--manifold without Nil or Sol components in its prime decomposition.
\end{itemize}
\end{cor}

Theorem \ref{thm:Morse local-to-globle for Morse detectable} also provides a possible avenue for proving the group $\Out(F_n)$ has the Morse local-to-global property by showing it is Morse detectable. Several partial results in this direction exist in the literature \cite{HamenstaedtHensel_Stability_Outer_Space, DowdallTaylor_Hyperbolic_extensions,ADT}, but the question of Morse detectability remains open and likely requires an innovative understanding of the Morse quasi-geodesics of $\Out(F_n)$.

\begin{ques}
 Is $\Out(F_n)$ or Outer Space Morse detectable for $n \geq 3$?
\end{ques}

\section{Spaces hyperbolic relative to Morse local-to-global spaces}\label{sec:relativelyhyperbolic}
In this final section, we show the Morse local-to-global property is inherited under relative hyperbolicity.

\begin{thm}\label{thm:local_to_global_and_relative_hyperbolicity}
Let $X$ be a geodesic metric space that is hyperbolic relative to a collection of peripheral subsets $\mc{P}$. If each element of $\mc{P}$ has the $\Phi$--Morse local-to-global property, then $X$ has the $\Psi$--Morse local-to-global property.
\end{thm}

If a group $G$ is hyperbolic relative to subgroups $H_1,\dots,H_n$, then the Cayley graph of $G$ with respect to any finite generating set is hyperbolic relative to the collection of left cosets of the $H_i$. Since each coset of the $H_i$ is isometric to $H_i$, if each $H_i$ has the $\Phi_i$--Morse local-to-global property, then there exists $\Phi$ so that every coset of one of the $H_i$ has the $\Phi$--Morse local-to-global property.  Theorem \ref{thm:local_to_global_and_relative_hyperbolicity}  then implies $G$ is also Morse local-to-global, proving Theorem \ref{intro_thm:relative_hyperbolicity} from the introduction.

Theorem \ref{thm:local_to_global_and_relative_hyperbolicity} implies that the Morse local-to-global property is closed under free products of finitely generated groups. Combining this with the work in Section \ref{sec:examples}, we deduce that the fundamental groups of all close $3$--manifolds have the Morse local-to-global property.

\begin{cor}\label{cor:3-manifolds}
If $M$ is a closed $3$--manifold, then $\pi_1(M)$ has the Morse local-to-global property.
\end{cor}

\begin{proof}
By the geometrization of closed $3$--manifolds, if $M$ is a closed $3$--manifold, then $M$ has a prime decomposition $M = M_1 \# M_2 \# \dots \# M_n$ where each $M_i$ is either geometric or has a mixed geometry. Thus, $\pi_1(M) \cong \pi_1(M_1) \ast \dots \ast \pi_1(M_n)$ where each $\pi_1(M_i)$ is either virtually solvable or a hierarchically hyperbolic space with the bounded domain dichotomy \cite[Theorem 10.1]{BHS_HHSII}. In both cases, $\pi_i(M)$ has the Morse local-to-global property by Section \ref{sec:trivial_examples}  or Corollary \ref{cor:Examples_of_HHS}.  Since $\pi_1(M)$  is hyperbolic relative to the collection of left cosets of $\pi_1(M_1),\dots, \pi_1(M_n)$, Theorem \ref{thm:local_to_global_and_relative_hyperbolicity} implies $\pi_1(M)$ has the Morse local-to-global property.
\end{proof}

The proof of Theorem \ref{thm:local_to_global_and_relative_hyperbolicity} is considerably longer and more technical than our previous proofs that spaces have the Morse local-to-global property. To guide the reader, we give an outline of the proof in Section \ref{subsec:rel_hyp_outline} after collecting some required facts about relatively hyperbolic spaces in Section \ref{subsec:rel_hyp_background}. The proof of Theorem \ref{thm:local_to_global_and_relative_hyperbolicity} then spans Sections \ref{subsec:local_qg_with_bounded_projections} to \ref{subsec:proof_or_rel_hyp}. As part of the proof, we investigate features of local quasi-geodesics in relatively hyperbolic spaces that maybe of independent interest. These include showing that local quasi-geodesics that ``avoid" all peripheral subsets are actually global Morse quasi-geodesics (Section \ref{subsec:local_qg_with_bounded_projections}) and developing a notion of ``deep points" for a local quasi-geodesic that decomposes a local quasi-geodesic into pieces that alternate between avoiding and passing through peripheral subsets (Section \ref{subsec:deep_points}).

\subsection{Background on relatively hyperbolic spaces}\label{subsec:rel_hyp_background}
If $\mc{P}$ is a collection of subsets of a geodesic metric space $X$, let $\cone{X}$ be the space obtained from $X$ by adding a point $c_P$ for each $P \in \mc{P}$ and connecting $c_P$ to every element of $P$ by an edge of length $1$. The space $X$ is \emph{hyperbolic relative to $\mc{P}$} if $\cone{X}$ is $\delta$--hyperbolic for some $\delta \geq 0$ and $X$ satisfies a \emph{bounded subset penetration} property that controls how geodesics travel through the elements of $\mc{P}$ (see \cite{Sisto_metric_rel_hyp} for a complete description and several equivalent definition of a relatively hyperbolic space).

For the remainder of this section, $X$ will be a fixed geodesic metric space that is hyperbolic relative to a collection of peripheral subsets $\mc{P}$, $\delta$ will be the constant so that $\cone{X}$ is $\delta$--hyperbolic, and $\pi \colon X \to \cone{X}$ will denote the distance non-increasing inclusion of $X$ into $\cone{X}$.

We now recall several facts about the relatively hyperbolic space $X$. The first says that the peripherals must be isolated away from each other.

\begin{lem}[Isolated peripherals; {\cite[Definition 2.4, Theorem 1.1]{Sisto_metric_rel_hyp}}] \label{lem:intersection of peripherals}
There is increasing  function $F \colon [0,\infty) \to [0,\infty)$ so that for all $P, U \in \mc{P}$ we have:
\[\mathrm{diam}\left( \mc{N}_K(P) \bigcap \mc{N}_K(U) \right) \leq F(K).\]
\end{lem}

A key tool in our study of relatively hyperbolic spaces is the coarse closest point projection onto peripheral subsets, denoted by $\pi_P \colon X \to P$ for all $P \in \mc{P}$. The basic properties of $\pi_P$ are outlined in the following lemma.

\begin{lem}[Properties of projection onto peripherals; \cite{Sisto_Rel_hyp_projections}]\label{lem:projections_lemmas}
There exist $\mu \geq 0$  so that for each $P \in \mc{P}$ there is a $(\mu,\mu)$--coarsely Lipschitz map $\pi_P \colon X \to P$ such that the following hold for all $P,U \in \mc{P}$.
\begin{enumerate}
    \item For all $x \in P$, $\pi_P(x) = x$.
    \item For all $x \in X - P$, $d\bigl(x,\pi_P(x) \bigr) \leq d(x,P) +1$.
    \item If $P \neq U$, $\diam(\pi_P(U)) \leq \mu$.
    \item For all $x \in X$, if $\gamma$ is a geodesic in $X$ from $x$ to $\pi_P(x)$, then $\diam(\pi_P(\gamma)) \leq \mu$.
\end{enumerate}
\end{lem}

The next two lemmas will frequently be used in  tandem. Together they imply that if two points $x,y \in X$ project significantly far apart on a peripheral $P$, then every quasi-geodesic connecting $x$ and $y$ must travel close to $P$ for a distance comparable to the distance between $\pi_P(x)$ and $\pi_P(y)$.

\begin{lem}[Linear quasiconvexity of peripherals;  {\cite[Lemma 4.5]{DrutuSapir}}]\label{lem: quasi-geodesics remain in a nbhd}
For every $\lambda\geq 1$ and $\epsilon\geq0$ there exists $r = r(\lambda, \epsilon) \geq 1$ such that for every $R\geq 1$ if $\gamma$ is a $(\lambda,\epsilon)$--quasi-geodesic  joining points in $\mc{N}_{R}(P)$, then $\gamma \subseteq \mc{N}_{rR}(P)$. 
\end{lem}

\begin{lem}[Bounded quasi-geodesic image; {\cite[Lemma 1.15]{Sisto_Rel_hyp_projections}}]\label{lem: BGI for rel hyp}
There exists $Q >0$ and $R \colon [1,\infty) \times [0,\infty) \to [0,\infty)$ so that for all $x,y \in X$ and $P \in \mc{P}$, if $d(\pi_P (x), \pi_P(y)) \geq Q$, then any $(\lambda, \epsilon)$--quasi-geodesic between $x$ and $y$ intersects $\mc{N}_{R}(\pi_P(x))$ and $\mc{N}_{R}(\pi_P(y))$ where $R = R(\lambda,\epsilon)$.
\end{lem}

The projections onto peripheral subsets along with the map $\pi \colon X \to \cone{X}$ produce a distance formula where distances in $X$ can be approximated by summing distances in the projections.

\begin{thm}[The distance formula; {\cite[Theorem 3.1]{Sisto_Rel_hyp_projections}\cite[Theorem 6.10]{BHS_HHSII}}]\label{thm:distance_formula_for_rel_hyp}
There exists $T_0 \geq 0$ so that for all $T \geq T_0$, there is $A\geq 1$ so that for all $x,y \in X$ we have \[d_X(x,y) \stackrel{A,A}{\asymp} \ignore{d_{\widehat{X}}\bigl(\pi(x),\pi(y)\bigr)}{T} + \sum \limits_{P \in \mc{P} } \ignore{d_P\bigl(\pi_P(x),\pi_P(y)\bigr))}{T} \]
\noindent where $\ignore{N}{T} = N$ if $N \geq T$ and $0$ otherwise.
\end{thm}

The most important consequence of the distance formula for our proof of Theorem \ref{thm:local_to_global_and_relative_hyperbolicity} are that quasi-geodesics whose projection to each peripheral subset is uniformly bounded will quasi-isometrically embed into $\cone{X}$ and  be Morse quasi-geodesics of $X$.

\begin{cor}[Quasi-isometric embedding and bounded projections]\label{cor:qi-emebed iff bounded projections}
Let $\gamma$ be a $(\lambda,\epsilon)$--quasi-geodesic in $X$.
\begin{enumerate}
    \item For every $C >0$, there exists $k \geq 1$ and $ c \geq 0$ so that if $\diam\bigl( \pi_P (\gamma) \bigr) \leq C$ for all $P \in \mc{P}$, then  $\pi \circ \gamma$ is a $(k,c)$--quasi-geodesic in $\cone{X}$.
    \item For every $k\geq 1$ and $c \geq 0$, there exists $C>0$ so that if $\pi \circ \gamma$ is a $(k,c)$--quasi-geodesic in $\cone{X}$, then $\diam\bigl( \pi_P (\gamma) \bigr) \leq C$ for all $P \in \mc{P}$.
\end{enumerate}
\end{cor}

\begin{proof}
Item (1) is a direct consequence of the distance formula. For Item (2), assume $\pi \circ \gamma$ is a $(k,c)$--quasi-geodesic and let $Q >0$ and $R = R(\lambda,\epsilon)$ be the constants from Lemma \ref{lem: BGI for rel hyp}. If $\diam(\pi_P(\gamma)) > \lambda[Q + 2R+2+\epsilon +k(2+2R)+c+1)]$ for some $P \in\mc{P}$, then there is $t,s$ in the domain of $\gamma$ so that $\gamma(t),\gamma(s) \in \mc{N}_R(P)$ and $|t-s| > k(2+2R)+c$. However, this would be a contradiction since $|t-s| \leq k(2+2R)+c$ as $\pi \circ \gamma(t), \pi \circ \gamma(s)$ are both within $R+1$ of the cone point $c_P$ and $\pi \circ \gamma$ is a $(k,c)$--quasi-geodesic. Thus, we have $\diam(\pi_P(\gamma)) \leq  \lambda[Q + 2R+2+\epsilon +k(2+2R)+c+1)]$ for all $P \in \mc{P}$.
\end{proof}

\begin{cor}[Quasi-geodesics with bounded projections are Morse]\label{cor:bounded projections implies Morse}
For each $C \geq 0$, $\lambda \geq 1$, and $\epsilon \geq 0$, there exists a Morse gauge $M$ so that if $\gamma$ is a $(\lambda,\epsilon)$--quasi-geodesic in $X$ and $\diam(\pi_P(\gamma)) \leq C$ for all $P \in \mc{P}$, then $\gamma$ is $M$--Morse.
\end{cor}

\begin{proof}
  Let $\alpha \colon [a,b] \to X$ be a $(k,c)$--quasi-geodesic with $\alpha(a),\alpha(b) \in \gamma$. We claim $\diam\bigl( \pi_P (\alpha) \bigr)$ is bounded above by a constant $C' = C'(k,c,C)$.

Let $Q \geq 0$ and $R = R(k,c)$ be the constants from Lemma \ref{lem: BGI for rel hyp} and suppose $\diam\bigl(\pi_P(\alpha)\bigr) \geq Q$ for some $P \in \mc{P}$. Let $t$ and $s$ be the first and last points in $[a,b]$ so that $\alpha(t),\alpha(s) \in \mc{N}_{R}(P)$. By Lemma \ref{lem: BGI for rel hyp}, $\diam\bigl( \pi_P(\alpha\vert_{[a,t]})\bigr) \leq Q$ and $\diam\bigl( \pi_P(\gamma\vert_{[s,b]})\bigr) \leq Q$. Since there exists $r=r(k,c)\geq 1$ so that $\alpha\vert_{[t,s]} \subseteq \mc{N}_{rR}(P)$ (Lemma \ref{lem: quasi-geodesics remain in a nbhd}), we have \[\diam\bigl( \pi_P(\alpha\vert_{[t,s]})\bigr) \leq k^2 \cdot d\bigl( \pi_P(\alpha(t)), \pi_P(\alpha(s)) \bigr)+5k^2(rR+c+1).\] Thus,
\begin{align*}
  \diam\bigl( \pi_P(\alpha)\bigr) \leq&  k^2 \cdot d\bigl( \pi_P(\alpha(a)), \pi_P(\alpha(b)) \bigr) + 5k^2(rR+c+1) +2Q \\
  \leq& k^2 C + 5k^2(rR+c+1) + 2Q = C' .
\end{align*}

Since $\diam\bigl( \pi_P (\alpha) \bigr) \leq C'$ and  $\diam\bigl( \pi_P (\gamma) \bigr) \leq C$, $\pi \circ \alpha$ and $\pi \circ \gamma$ are respectively  $(k',c')$ and $(\lambda',\epsilon')$--quasi-geodesic in $\widehat{X}$  where $(k',c')$ and $(\lambda',\epsilon')$ depend only on $k$, $c$, $C$ and $\lambda$, $\epsilon$,  $C$ respectively (Corollary \ref{cor:qi-emebed iff bounded projections}). Since $\widehat{X}$ is $\delta$--hyperbolic, there exists a Morse gauge $N$, depending on $\lambda'$, $\epsilon'$, and  $\delta$, so that $\pi \circ \gamma$ is $N$--Morse and $\pi \circ \alpha$ is contained in the $N(k',c')$--neighborhood of $\pi \circ \gamma$. For $x \in \alpha$, let $y \in \gamma$ so that $d_{\widehat{X}}(x,y) \leq N(k',c')$.  Since $d_P(x,y) \leq C + C'$ for all $P \in \mc{P}$, the distance formula (Theorem \ref{thm:distance_formula_for_rel_hyp}) produces $C'' = C''(\lambda,\epsilon,C,k,c)$ so that $d_X(x,y) \leq C''$. Thus, $\gamma$ is $M$--Morse where $M$ depends on $\lambda$, $\epsilon$, $C$, and $\delta$.
\end{proof}

Our last preliminary result is an adaption of the work of Hruska \cite[Proposition 8.14]{Hruska10} and Sisto \cite[Proposition 5.7]{Sisto_metric_rel_hyp} to fit our needs. In the sequel, we say a map $\gamma \colon [a,b] \to X$ is an \emph{unparametrized $(\lambda,\epsilon)$--quasi-geodesic} if there exists a homeomorphism $f \colon [a',b'] \to [a,b]$ so that $\gamma \circ f$ is a $(\lambda,\epsilon)$--quasi-geodesic of $X$.

\begin{lem}\label{lem: projections are unparametrized q geod}
 For every $\lambda \geq 1$ and  $\epsilon\geq 0$, there are $k \geq 1$ and $c \geq 0$ so that if $\gamma \colon [a,b] \to X$ is  $(\lambda,\epsilon)$--quasi-geodesic in $X$, then $\pi \circ \gamma$ is an unparametrized  $(k,c)$--quasi-geodesic in $\cone{X}$.
\end{lem}

\begin{proof}
It is sufficient to verify that $\pi \circ \gamma$ is within finite Hausdorff distance of the image of a  geodesic in $\cone{X}$. Since $X$ is a geodesic space, we can assume $\gamma$ is a continuous quasi-geodesic without any loss of generality \cite[Lemma III.H.1.11]{Bridson_Haefliger}.

By \cite[Lemma 1.14]{Sisto_Rel_hyp_projections}, there is a continuous $(\lambda',\epsilon')$--quasi-geodesic $\eta \colon [a',b'] \to X$ so that $\pi \circ \eta$ is an unparametrized geodesic in $\cone{X}$.
By \cite[Proposition 5.7]{Sisto_metric_rel_hyp}, there is a constant $D = D(\lambda,\epsilon)$ and sequences $a=q_0\leq p_1<q_1<\dots<p_n<q_n \leq p_{n+1} = b$  and $a'=q'_0\leq p'_1<q'_1<\dots<p'_n<q'_n \leq p'_{n+1} = b'$  so that 
\begin{itemize}
	\item the Hausdorff distance between $\gamma \vert_{[q_i,p_{i+1}]}$ and $\eta\vert_{[q'_i,p'_{i+1}]}$ is bounded by $D$ for $i \in \{0,\dots,n\}$;
	\item $\gamma\vert_{[p_{i+1},q_{i+1}]}$, and $\eta\vert_{[p'_{i+1},q'_{i+1}]} $ are both contained in the $D$--neighborhood of the same peripheral $P_i \in \mc{P}$ for $i \in \{0,\dots,n-1\}$.
\end{itemize}
\noindent This implies the Hausdorff distance between $\pi \circ \gamma \vert_{[q_i,p_{i+1}]}$ and $\pi \circ \eta\vert_{[q'_i,p'_{i+1}]}$ is also at most $D$ for $i \in \{0,\dots,n\}$ and $\diam\bigl( \pi(\gamma \vert_{[p_{i+1},q_{i+1}]}) \bigr) \leq 2D+2$ for all $i \in\{0,\dots,n-1\}$. Thus, the Hausdorff distance between $\pi \circ \gamma$ and the unparametrized geodesic $\pi \circ \eta$ is bounded by a constant depending only on $\lambda$ and $\epsilon$ as desired.
\end{proof}

Before continuing, we set constants that we shall use for the remainder of the section. We fix a pair of quasi-geodesic constants  $\lambda \geq 1$ and $\epsilon \geq 0$ and then fix the following:
\begin{itemize}
    \item $\mu$ the constant from Lemma \ref{lem:projections_lemmas};
    \item $F \colon [0,\infty) \to [0,\infty)$ the function from Lemma \ref{lem:intersection of peripherals};
    \item  $r = r(\lambda,\epsilon)$ the constant from Lemma \ref{lem: quasi-geodesics remain in a nbhd};
    \item  $Q$ and $R=R(\lambda,\epsilon)$ the constants from Lemma \ref{lem: BGI for rel hyp}.
\end{itemize}
With these constants fixed, we define $\theta = \theta(\lambda,\epsilon)$ to be 
\[ \theta = 100\lambda\mu(\lambda+\epsilon+\mu+F(2rR)+r+ R+ rR +Q+1). \] 

Henceforth, whenever $\lambda$, $\epsilon$, $\mu$, $r$, $Q$, $R$, or $\theta$ are written, they implicitly refer to the  constants outlined above. Note, $r$, $R$, and $\theta$ depend on $\lambda$ and $\epsilon$, while $\mu$, $F$, and $Q$ do not. Further, $I$ will always denote a compact interval of $\mathbb{R}$ for the remainder of the section.

\subsection{A sketch of the proof of Theorem \ref{thm:local_to_global_and_relative_hyperbolicity}} \label{subsec:rel_hyp_outline}
Similar to the case of $\CAT(0)$ spaces, the key to proving Theorem \ref{thm:local_to_global_and_relative_hyperbolicity} is to establish that, when every peripheral subset has the $\Phi$--Morse local-to-global property, every local Morse quasi-geodesic of sufficient scale in $X$ is finite Hausdorff distance from any quasi-geodesic between its endpoints. However, unlike the $\CAT(0)$ case, the argument requires substantially more care and technicalities.  The main steps of the proof are outlined below with work on the proof beginning in the next section.

\begin{enumerate}[{Step} 1:]
    \item \label{item:local_qg_with_bounded projections} Use the local-to-global property of quasi-geodesics in the hyperbolic space $\cone{X}$ to prove a local version of Corollary \ref{cor:bounded projections implies Morse}, i.e.,  any local quasi-geodesic where the projection of every local segment is uniformly bounded is a Morse quasi-geodesic (Section \ref{subsec:local_qg_with_bounded_projections}).
    \item Show any local Morse quasi-geodesic $\gamma$ has a decomposition $\gamma = \decomposition$  and a collection of peripheral subsets $\{P_1, \dots,P_n\}$ so that 
    \begin{itemize}
        \item each $\alpha_i$ is sufficiently long and contained in the $rR$--neighborhood of the peripheral subset $P_i$;
        \item the projection of each local segment of each $\sigma_i$ to every peripheral subset is uniformly bounded (Section \ref{subsec:deep_points}).
    \end{itemize}
    \item Argue each $\sigma_i$ and $\alpha_i$ are Morse quasi-geodesics. For the $\sigma_i$, this follows from Step \ref{item:local_qg_with_bounded projections} since  projection of each local segment of each $\sigma_i$ to every peripheral subset is uniformly bounded. For the $\alpha_i$, we apply the Morse local-to-global property of each of the peripherals $P_i$ whose $rR$--neighborhood contains $\alpha_i$  (Corollary \ref{cor: alphas are Morse q.g}).
    \item \label{item:endpoints_projection} Prove the projection of the endpoints of $\gamma$ to each $P_i$ are coarsely equal to the projections of the endpoints of $\alpha_i$ to $P_i$. We show this by establishing  $P_1,\dots, P_n$ are ``linearly ordered" by $\gamma$, that is, if $1\leq i<j<k\leq n$, the projection of $P_i$ onto $P_k$ is coarsely equal to the projection of $P_j$ onto $P_k$ (Section \ref{subsec:ordering_of_peripherals}).
    \item \label{item:BGI} Since the $\alpha_i$ are sufficiently long, Step \ref{item:endpoints_projection} ensures that the projection of the endpoints of $\gamma$ to each of the $P_i$ are far enough apart that we can apply Lemma \ref{lem: BGI for rel hyp} to concluded that every quasi-geodesic between the endpoints of $\gamma$ must pass close to the endpoints of each $\alpha_i$ (Lemma \ref{lem: BCP for local things}).
    \item Since each of the $\sigma_i$ and $\alpha_i$ are Morse quasi-geodesics, Step \ref{item:BGI} ensures any quasi-geodesic between the endpoints of $\gamma$ will be within finite Hausdorff distance from $\gamma$. This directly implies $\gamma$ has the Morse property and Lemma \ref{lem:local_close_to_global} will imply that $\gamma$ is a global quasi-geodesic.
\end{enumerate}

\subsection{Local quasi-geodesics with bounded projections}\label{subsec:local_qg_with_bounded_projections}

We now study local quasi-geodesics where the projection of local subsegments to all peripheral subsets are uniformly bounded.  

\begin{defn} \label{defn:bounded_projections}
An $(L;\lambda,\epsilon)$--local quasi-geodesic $\gamma$ has \emph{$D$--bounded projections} if  every subsegment $\sigma \subseteq \gamma$ of parametrized length at most $L$ has  $\diam(\pi_P ({\sigma}))\leq D$ for all $P \in \mc{P}$.
\end{defn}

The goal is to prove a local version of Corollary \ref{cor:bounded projections implies Morse}. That is, for each $D$ there exists a local scale $L$ so that $(L;\lambda,\epsilon)$--local quasi-geodesics with $D$--bounded projections are global Morse quasi-geodesics.

\begin{cor}\label{cor: local qg+bounded projection=Morse}
For each  $D \geq 0$, there exist  $L_1= L_1(\lambda,\epsilon,D)$, $A_2 = A_2(\lambda, \epsilon, D)$, and Morse gauge $M=M(\lambda,\epsilon,D)$ so that if an $(L_1;\lambda,\epsilon)$--local quasi-geodesic $\gamma \colon  I \to X$ has $D$--bounded projection, then $\gamma$ is a global $(M;A_2, A_2)$--Morse quasi-geodesic of $X$. 
\end{cor}

To prove Corollary \ref{cor: local qg+bounded projection=Morse}, we first show that local quasi-geodesics with bounded projections are quasi-geodesics in~$\cone{X}$ .

\begin{prop}\label{prop:bounded_projections_imply_quasi-geodesic}
For each $D\geq 0$, there exist $\lambda' = \lambda'(\lambda,\epsilon)$, $\epsilon'=\epsilon'(\lambda,\epsilon)$, $L_1= L_1(\lambda,\epsilon,D)$ and $A_1 = A_1(\lambda, \epsilon, D)$ such that the following holds. If $\gamma \colon  I \to X$ is an $(L_1;\lambda,\epsilon)$--local quasi-geodesic with $D$--bounded projections, then 
\begin{enumerate}
    \item ${\pi} \circ \gamma \colon I \to \widehat{X}$ is an unparametrized $(\lambda',\epsilon')$--quasi-geodesic.
     \item ${\pi} \circ \gamma \colon I \to \widehat{X}$ is a parametrized $(A_1, A_1)$--quasi-geodesic.
\end{enumerate}
\end{prop}
\begin{proof}
  Let $k$ and $c$ be so that every $(\lambda,\epsilon)$--quasi-geodesic in $X$ is a unparametrized $(k,c)$--quasi-geodesic in $\cone{X}$ (Lemma \ref{lem: projections are unparametrized q geod}). Let  $\ell$, $\lambda'$, and $\epsilon'$ be such that every $(\ell;k^3,5k^2c)$--local quasi-geodesic in the hyperbolic space $\cone{X}$ is a global $(\lambda',\epsilon')$--quasi-geodesic in $\cone{X}$ (Theorem \ref{thm:local_to_global_hyperbolic}). Let $A \geq 1$ be the constant from the distance formula (Theorem \ref{thm:distance_formula_for_rel_hyp}) with $T = T_0 +D$. Assume $L_1 \geq 20\ell A^2\lambda k(1+\epsilon+c)$ and let $\gamma \colon  I \to X$ be an $(L_1;\lambda,\epsilon)$--local quasi-geodesic with $D$--bounded projections.  We assume $\diam(I) > L_1$ as the proposition follows from Lemma \ref{lem: projections are unparametrized q geod} and the distance formula  when $\diam(I) \leq L_1$. 

For the first item, let $a_0<a_1< \dots <a_n$ be elements of $I$ such that $[a_0,a_n] =I$ and $L_1/4 \leq |a_i - a_{i+1} | \leq L_1/2$ for all $0\leq i \leq n-1$. Since ${\pi} \circ \gamma\vert_{[a_i,a_{i+1}]}$ is an unparametrized $(k,c)$--quasi-geodesic for each $0\leq i \leq n-1$, there exists  real numbers $b_0 < b_1< \dots <b_n$ and a homeomorphism $f \colon [b_0,b_n] \to [a_0,a_n]$ such that $f(b_i) = a_i$ for $0 \leq i \leq n$ and ${\pi} \circ \gamma \circ f \vert_{[b_i,b_{i+1}]}$ is a $(k,c)$--quasi-geodesic in $\cone{X}$ for each $0 \leq i \leq n-1$. Let $\widetilde{\gamma} = {\pi} \circ \gamma \circ f$. The next two claims show $\widetilde{\gamma}$ is an $(\ell,k^3,5k^2c)$--local quasi-geodesic in $\cone{X}$. This implies $\widetilde{\gamma}$ is a $(\lambda',\epsilon')$--quasi-geodesic of $\cone{X}$ by the local-to-global property of quasi-geodesics in the hyperbolic space $\cone{X}$.
\begin{claim}\label{claim:reparmeterized_segments_are_long}
$2\ell \leq |b_i-b_{i+1}|$  for all $0\leq i \leq n-1$.
\end{claim}
\begin{subproof}[Proof of Claim \ref{claim:reparmeterized_segments_are_long}]
Fix $i \in \{0,1,\dots,n\}$. Since $\gamma$ has $D$--bounded projections, the distance formula with threshold $T = T_0 +D$ provides $d_{X}(\gamma(a_i),\gamma(a_{i+1})) \leq A\cdot d_{\cone{X}}\bigl(\cone{\gamma}(a_i),\cone{\gamma}(a_{i+1})\bigr) +A.$ We show $2\ell \leq |b_i-b_{i+1}|$ with the following calculation.
\begin{align*}
    \frac{L_1}{4} \leq |a_i - a_{i+1}| &\leq \lambda \cdot d_X\bigl(\gamma(a_i),\gamma(a_{i+1})\bigr) +\epsilon\\
    &\leq A\lambda \cdot d_{\cone{X}}\bigl(\cone{\gamma}(a_i),\cone{\gamma}(a_{i+1})\bigr) +A\lambda +\epsilon\\
    &= A\lambda \cdot d_{\cone{X}}\bigl(\widetilde{\gamma}(b_i),\widetilde{\gamma}(b_{i+1})\bigr) +A\lambda +\epsilon\\
    &\leq  A\lambda k \cdot |b_i - b_{i+1}| +A\lambda c +A\lambda +\epsilon. \qedhere
\end{align*}
\end{subproof}

\begin{claim}\label{claim:projection_is_local_qg}
$\widetilde{\gamma}$ is an $(\ell;k^3,5k^2c)$--local quasi-geodesic in $\cone{X}$.
\end{claim}

\begin{subproof}[Proof of Claim \ref{claim:projection_is_local_qg}]
Let $s,t \in [b_0,b_n]$ with $|s-t| \leq \ell$. By Claim \ref{claim:reparmeterized_segments_are_long}, we can assume there exist $i\in \{1,\dots,n-1\}$ so that $s \in [b_{i-1},b_i]$ and $t \in [b_i,b_{i+1}]$. It suffices to show that $\widetilde{\gamma}\vert_{[b_{i-1},b_{i+1}]}$ is a $(k^3,5k^2c)$--quasi-geodesic.  

Now, $\widetilde{\gamma}\vert_{[b_{i-1},b_{i+1}]}$ is a concatenation of the $(k,c)$--quasi-geodesics $\widetilde{\gamma}\vert_{[b_{i-1},b_{i}]}$ and $\widetilde{\gamma}\vert_{[b_{i},b_{i+1}]}$. 
Further, $\widetilde{\gamma}\vert_{[b_{i-1},b_{i+1}]}$ is a unparametrized $(k,c)$--quasi-geodesic  because $\widetilde{\gamma}\vert_{[b_{i-1},b_{i+1}]}$ is a reparametrization of  $\cone{\gamma}\vert_{[a_{i-1},a_{i+1}]}$ and $\cone{\gamma}\vert_{[a_{i-1},a_{i+1}]}$ is an unparametrized $(k,c)$--quasi-geodesic since $|a_{i-1} - a_{i+1}| \leq L_1$. Thus, Lemma \ref{lem:unparametrized_concatentation} below implies $\widetilde{\gamma}\vert_{[b_{i-1},b_{i+1}]}$ is a $(k^3,5k^2c)$--quasi-geodesic.
\end{subproof}

For the second item, let $t_1, t_2 \in I$ with $|t_1 - t_2| \leq L_1$ and let $\gamma(t_1)= x_1, \gamma(t_2) = x_2$. Using the distance formula as in Claim \ref{claim:reparmeterized_segments_are_long}, we get
 \[d_{\cone{X}}({\pi}(x_1),{\pi}(x_2)) \geq \frac{1}{A}d_{X}(x_1,x_2) -A \geq \frac{1}{A}\left(\frac{1}{\lambda}|t_1 -t_2|- \epsilon \right) -A.\]
Since $\pi \colon X \to \cone{X}$ is distance non-increasing, we also have 
 \[d_{\cone{X}}({\pi}(x_1),{\pi}(x_2))\leq  d_X(x_1,x_2) \leq \lambda | t_1 - t_2| + \epsilon. \]
 These calculations show that  $\widehat{\gamma}$ is an $(L_1; A', A')$--local quasi-geodesic in the hyperbolic space $\cone{X}$ for some $A'= A'(\lambda,\epsilon,D)$ . By increasing $L_1$, we can use the local-to-global property of quasi-geodesics in the hyperbolic space $\cone{X}$, to obtain that $\widehat{\gamma}$ is a $(A_1, A_1)$--quasi-geodesic where $A_1 = A_1(\lambda,\epsilon,D)$.
\end{proof}

\begin{lem}\label{lem:unparametrized_concatentation}
Let $Y$ be a metric space and  $\gamma_1 \colon [a,b] \to Y$ and $\gamma_2\colon [b,d] \to Y$ be $(k, c)$--quasi-geodesics with $\gamma_1(b) = \gamma_2(b)$. If the concatenation $\gamma_1 \ast \gamma_2$ is an unparametrized $(k,c)$--quasi-geodesic, then $\gamma_1\ast\gamma_2$ is a $(k^3, 5k^2 c)$--quasi-geodesic.
\end{lem}

\begin{proof}
Let $f \colon [a',d'] \to [a,d]$ be the homeomorphism of $\gamma_1 \ast \gamma_2$ so that $ (\gamma_1 \ast \gamma_2) \circ f$ is a $(k,c)$--quasi-geodesic. Let  $\gamma = \gamma_1 \ast \gamma_2$ and $s,t \in [a,d]$. Let $s', t',b' \in [a',d']$ so that $s = f(s')$, $t= f(t')$, and $b = f(b')$.  Without loss of generality, we can assume $s \in [a,b]$ and $t\in[b,d]$. The fact that $d\big(\gamma(s), \gamma(t) \bigr) \leq k |s-t| +2c$ follows from the triangle inequality. The other inequality is the following calculation.
\begin{align*}
    |s-t| &\leq k \cdot d\big(\gamma(s), \gamma(b) \bigr) + k \cdot d\big(\gamma(b), \gamma(t) \bigr) +2c \\
    &= k \cdot d\big(\gamma(f(s')), \gamma(f(b')) \bigr) + k \cdot d\big(\gamma(f(b')), \gamma(f(t')) \bigr) +2c\\
    &\leq k^2|s' -t'| +4k c\\
    &\leq k^3 \cdot d\big(\gamma(f(s')), \gamma(f(t')) \bigr) + 5k^2 c\\
    &= k^3 \cdot d\big(\gamma(s), \gamma(t) \bigr) + 5k^2 c. \qedhere
\end{align*}
\end{proof}

We now apply Corollaries \ref{cor:qi-emebed iff bounded projections} and \ref{cor:bounded projections implies Morse} to finish the proof  Corollary \ref{cor: local qg+bounded projection=Morse}.

\begin{proof}[Proof of Corollary \ref{cor: local qg+bounded projection=Morse}]
Let $L_1 = L_1(\lambda,\epsilon,D)$ be the local scale from Proposition \ref{prop:bounded_projections_imply_quasi-geodesic} and $\gamma$ be an $(L_1,\lambda,\epsilon)$--local quasi-geodesic with $D$--bounded projection. By Proposition \ref{prop:bounded_projections_imply_quasi-geodesic}, $\pi \circ \gamma$ is a parametrized $(A_1,A_1)$--quasi-geodesic where $A_1 = A_1(\lambda,\epsilon,D)$. The fact that $\gamma$ is an $(A_2, A_2)$--quasi-geodesic for $A_2 = A_2(\lambda,\epsilon,D)$ then follows from Lemma \ref{lem: Folklore}. Since $\pi \circ \gamma$ is a parametrized $(A_1,A_1)$--quasi-geodesic, Corollary \ref{cor:qi-emebed iff bounded projections} provides a constant $C = C(\lambda,\epsilon,D)$ so that $\diam(\pi_P(\gamma)) \leq C$ for each $P \in \mc{P}$. By Corollary \ref{cor:bounded projections implies Morse},  $\gamma$ is $M$--Morse for some $M$ depending ultimately on $\lambda$, $\epsilon$, and~$D$.
\end{proof}

\subsection{Deep points decomposition of local quasi-geodesics}\label{subsec:deep_points}

As a consequence of the isolation and quasiconvexity of the peripheral subsets (Lemmas \ref{lem:intersection of peripherals} and \ref{lem: quasi-geodesics remain in a nbhd}), any quasi-geodesic $\gamma$ in $X$ has a decomposition $\gamma = \decomposition$ so that each $\alpha_i$ is ``deep" in a peripheral, i.e., $\alpha_i$ runs close to a single peripheral for a long time, and each $\sigma_i$ avoids getting close to any peripheral for a significant length of time (see \cite{Hruska10} and \cite{Sisto_metric_rel_hyp} for details). The goal of this section is to give a similar decomposition for local quasi-geodesics in $X$. We begin by defining when an element $t$ of the domain of a local quasi-geodesic $\gamma$ is ``deep" in a peripheral $P \in \mc{P}$. This means that we can find points $s_1$, $s_2$ on either side of $t$ whose distance from $t$ is between $\theta$ and $L/4$ and with $\gamma(s_1)$,$\gamma(s_2)$ close to $P$.

\begin{defn}[Deep points]
Let $X$ be hyperbolic relative to a collection $\mc{P}$ of peripheral subsets and  $\gamma \colon I \to X$ be an $(L;\lambda, \epsilon)$--local quasi-geodesic where $L \geq 12\theta$.   A number $t \in I$ is \emph{$P$--deep} for some $P \in \mc{P}$, if there are $s_1 <t <s_2$ in $I$ such that $\theta \leq |s_i - t| \leq L/4$ and $d(\gamma(s_i), P)\leq rR$. We say $s_1$ (resp. $s_2$) is a \emph{left (resp. right) witness of $t$}. For each $P \in \mc{P}$, define $\deep(\gamma;P)$ to be the set of $P$--deep elements of $I$.
\end{defn}

The next three results establish that $\deep(\gamma;P)$ is a collection of disjoint intervals and isolated points of the domain of $\gamma$ and that points in the domain of $\gamma$ cannot be deep for two different peripherals.

\begin{lem}\label{lem: deep points are uniformly near}
Let $\gamma \colon I \to X$ be an $(L; \lambda,\epsilon)$--local quasi-geodesic, where $L \geq 12 \theta$. If  $t \in \deep(\gamma;P)$, then there exists $u_1<t<u_2$ so that $\gamma([u_1,u_2]) \subseteq N_{rR} (P)$ and $\theta/2 \leq |t-u_i| \leq L/4$.
\end{lem}

\begin{proof}
Let $s_1, s_2$ be the left and right witnesses of $t \in \deep(\gamma;P)$. Since $\gamma(s_i) \in \mc{N}_{rR}(P)$ and $\gamma \vert_{[s_1,s_2]}$ is a $(\lambda,\epsilon)$--quasi-geodesic,\[d\bigl( \pi_P(\gamma(s_1)), \pi_P(\gamma(s_2)) \bigl) \geq \frac{2\theta}{\lambda} - \epsilon - 2rR-2 > Q.\] By Lemmas \ref{lem: BGI for rel hyp} and \ref{lem: quasi-geodesics remain in a nbhd}, there exists $[u_1, u_2] \subseteq [s_1,s_2]$ so that $|s_i - u_i|\leq (2rR+1)\lambda + \epsilon$ and $\gamma([u_1,u_2])$ is contained in $\mc{N}_{r R}(P)$. Since $|s_i - u_i|< \theta/2$ and $\theta \leq |s_i -t| \leq L/4 $, $t \in [u_1, u_2]$ and $\theta/2 \leq |t-u_i| \leq L/4$.
\end{proof}

\begin{cor}\label{cor:unique peripherals for deep points}
Let $\gamma \colon I \to X$ be an $(L; \lambda,\epsilon)$--local quasi-geodesic, where $L \geq 12 \theta$. For all distinct $P,U \in \mc{P}$, $\deep(\gamma;P) \cap \deep(\gamma;U) = \emptyset$.
\end{cor}

\begin{proof}
Suppose there exists $t \in  \deep(\gamma;P) \cap \deep(\gamma;U)$ where $P \neq U$ and let $F \colon [0,\infty) \to [0,\infty)$ be the function from Lemma \ref{lem:intersection of peripherals} that bounds the coarse intersection of $P$ and $U$. Let $[p_1,p_2]$ and $[u_1,u_2]$ be the intervals containing $t$ provided by Lemma \ref{lem: deep points are uniformly near} for $P$ and $U$ respectively. Without loss of generality, assume $p_1 \leq u_1$. Thus, $\gamma(u_1),\gamma(t) \in \mc{N}_{rR}(P) \cap \mc{N}_{rR}(U)$ and \[d(\gamma(t),\gamma(u_1)) \geq \frac{\theta}{2\lambda} - \epsilon \geq F(rR)\] by choice of $\theta$. However, this contradicts Lemma \ref{lem:intersection of peripherals}, so we must have~$\deep(\gamma;P) \cap \deep(\gamma;U)=\hspace{-3mm}~\emptyset$.
\end{proof}

\begin{cor}\label{cor: deep points are a disjoint union}
Let $\gamma \colon I \to X$ be an $(L;\lambda,\epsilon)$--local quasi-geodesic with $L \geq 12 \theta$. If $[t_1,t_2] \subseteq I$  so that $|t_1 - t_2| \leq L/4$ and $t_1,t_2 \in \deep(\gamma;P)$ for some $P \in \mc{P}$, then we have $[t_1,t_2] \subseteq \deep(\gamma;P)$.
In particular, for each $P \in \mc{P}$, $\deep(\gamma;P)$ is a (possibly empty) union of disjoint intervals and isolated points of $I$ where the distance between any two connected components of $\deep(\gamma;P)$ is greater than $L/4$. 
\end{cor}

\begin{proof}
Let $t \in [t_1,t_2]$.  We give the proof of the existence of a left witness for $t$. The proof for a right witness is analogous. 
Let $s_1, s_2 \in I$ be a left and right witnesses for $t_1$ respectively. If $|t - t_1| \geq \theta$ or $|s_1 - t| \leq L/4$, then either $t_1$ or $s_1$ is a left witness for $t$. Suppose $|t - t_1| \leq \theta$ and $|s_1 - t| \geq L/4$. Thus, $|t_1 - s_1| \geq L/4 - \theta \geq 2\theta$. 
By Lemma \ref{lem: deep points are uniformly near}, $\gamma(t_1)$ and $\gamma(s_1)$ are both within $rR$ of $P$, so $d\bigl( \pi_P(\gamma(t_1)),\pi_P(\gamma(s_1)) \bigr) > Q$ and Lemmas \ref{lem: BGI for rel hyp} and \ref{lem: quasi-geodesics remain in a nbhd} produces $[u_1,u_2] \subseteq [s_1,t_1]$ so that \[d\bigl( \pi_P(\gamma(s_1)) , \gamma(u_1) \bigr) \leq R \text{ and } d\bigl( \pi_P(\gamma(t_1)) , \gamma(u_2) \bigr) \leq R\]  and $\gamma\bigl( [u_1,u_2] \bigr)$ is contained in the $rR$--neighborhood of $P$.
Since \[d\bigl( \pi_P(\gamma(s_1)), \gamma(s_1) \bigr)\leq rR+1 \text{ and } d\bigl( \pi_P(\gamma(t_1)), \gamma(t_1) \bigr)  \leq rR+1\] and $\gamma \vert_{[s_1,t_1]}$ is a $(\lambda,\epsilon)$--quasi-geodesic, $|s_1 - u_1| \leq \theta$ and $|t_1 - u_2| \leq \theta$.
In particular, $|u_1 - t| \geq \theta$ and $|u_2 - t| \leq 2\theta$.  Thus, there exists $\tau \in [u_1,u_2]$ so that $\theta \leq |\tau - t| \leq L/4$ and $\gamma(\tau) \in \mc{N}_{rR}(P)$. 
\end{proof}

Corollaries \ref{cor:unique peripherals for deep points} and \ref{cor: deep points are a disjoint union} establish that the domain of a local quasi-geodesic can be decomposed into non-overlapping deep points for peripheral subsets and points that are not deep for any peripheral.  However, we want to ensure that the deep parts of our decomposition run near the peripheral for a sufficiently long time. To achieve this, we define the following \emph{relevant decomposition} of a local quasi-geodesic.

\begin{defn}[Relevant subsegments, relevant decomposition] \label{defn: relevant decomposition}
Let $B \geq 0$ and $\gamma \colon I \to X$ be an $(L;\lambda,\epsilon)$--local quasi-geodesic with $L \geq 12\theta$. A subsegment $\alpha \subseteq \gamma$ is \emph{$B$--relevant} if the parametrized length of $\alpha$ is at least $B$ and there exists $P \in \mc{P}$ so that the domain of $\alpha$ is the closure of a connected component of $\deep(\gamma;P)$. If $\rel_B(\gamma) = \{\alpha_i\}_{i=1}^n$ is the set of $B$--relevant subsegments of $\gamma$, then  Corollaries \ref{cor:unique peripherals for deep points} and \ref{cor: deep points are a disjoint union},  decompose $\gamma$ into a concatenation:
\[\gamma= \sigma_0 \ast \alpha_1 \ast \sigma_1 \ast \cdots \ast \alpha_n \ast \sigma_n,\] where each $\sigma_i$ does not contain a $B$--relevant subsegment of $\gamma$. We call this decomposition the \emph{$B$--relevant decomposition of $\gamma$}. For each $\alpha_i \in \rel_B(\gamma)$, there exists a single $P_i \in \mc{P}$ so that the interior of the domain of $\alpha_i$ is a subset of $\deep(\gamma;P_i)$. We call $P_i$ the \emph{relevant peripheral for $\alpha_i$} and the set $\{P_i\}_{i=1}^n$ the \emph{relevant peripherals for $\gamma$}.
\end{defn}

The next two lemmas give some basic properties of the relevant decomposition of a local quasi-geodesic. Lemma \ref{lem:alphas are close and sigmas have bounded projections} says the $\alpha_i$ of the decomposition run close to the relevant peripheral $P_i$ while the $\sigma_i$ have bounded projection (in terms of $B$) and hence cannot run close to any peripheral subset. Lemma \ref{lem:adjacent_relevant_peripherals_are_distinct} says relevant peripherals for ``adjacent" $\alpha_i$ must be distinct.

\begin{lem}\label{lem:alphas are close and sigmas have bounded projections}
For every  $B >2 \theta$ and $L \geq 12\theta$, if $\decomposition$ is the $B$--relevant decomposition of  an $(L; \lambda, \epsilon)$--local quasi-geodesic $\gamma$, then
\begin{enumerate}
    \item \label{item:relevant} $\alpha_i \subseteq \mc{N}_{r R} (P_i)$ where $P_i$ is the relevant peripheral for $\alpha_i$;
    \item \label{item:non-relevant} each $\sigma_i$ has  $\lambda(4 B + \epsilon)$--bounded projections.
\end{enumerate}
\end{lem}

\begin{proof}
Item (\ref{item:relevant}) follows immediately from Lemma \ref{lem: deep points are uniformly near}. For Item (\ref{item:non-relevant}), let $t_1,t_2$  be in the domain of $\sigma_i$ with $|t_1 - t_2| \leq L/4$. Suppose $d\bigl(\pi_P(\sigma_i(t_1)),\pi_P(\sigma_i(t_2))\bigr) \geq \lambda(4 B + \epsilon)$ for some $P\in\mc{P}$. By Lemma \ref{lem: BGI for rel hyp}, we can find points $[s_1, s_2] \subseteq [t_1,t_2]$ with  $|s_1 -s_2| >3B$ so that $\sigma_i(t_1)$ and $\sigma_i(t_2)$ are within $R$ of $\pi_{P}(\sigma_i(t_1))$ and $\pi_{P}(\sigma_i(t_1))$ respectively. Since $3\theta < 3B \leq |s_1 - s_2| \leq L/4 $,  Lemma \ref{lem: quasi-geodesics remain in a nbhd} implies $[s_1 + B, s_2-B] \subseteq \deep(\gamma;P)$, contradicting the definition of $\sigma_i$. 
\end{proof}

\begin{lem}\label{lem:adjacent_relevant_peripherals_are_distinct}
For every  $B >2\theta$, there exists $L_2 = L_2(\lambda,\epsilon,B)$ so that if $\decomposition$ is the $B$--relevant decomposition of  an $(L_2; \lambda, \epsilon)$--local quasi-geodesic $\gamma$ and $\{P_1,\dots, P_n\}$ are the relevant peripherals for $\gamma$, then $P_i \neq P_{i+1}$ for each $1 \leq 1 \leq n-1$.
\end{lem}

\begin{proof}
Let $L_1 = L_1(\lambda,\epsilon,B)$ and $A_1 = A_1(\lambda,\epsilon, B)$ be the constants from Proposition \ref{prop:bounded_projections_imply_quasi-geodesic} so that if $\eta$ is an $(L_1,\lambda, \epsilon)$--local quasi-geodesic in $X$ with the {$\lambda(4 B + \epsilon)$}--bounded projections, then $\pi \circ \eta$ is an $(A_1,A_1)$--quasi-geodesic in $\cone{X}$. 
Let $L_2 = L_1+ 12\theta A^2_1 $ and $\gamma$ be an $(L_2,\lambda,\epsilon)$--local quasi-geodesic with $B$--relevant decomposition $\decomposition$ and relevant peripherals $\{P_1,\dots, P_n\}$. 
Suppose $P_i =P_{i+1}$ for some $1 \leq i \leq n-1$. 
Let $I$ be the domain for $\gamma$ and $[a,b] \subseteq I$ be the domain of $\sigma_i $.
If $|a-b| \leq  L_2/4$, then Corollary \ref{cor: deep points are a disjoint union} would imply $[a,b] \subseteq \deep(\gamma;P_i)$ contradicting the definition of $\sigma_i$. If $|a-b| > L_2/4$, then $\pi \circ \sigma_i(a)$ and $\pi\circ \sigma_i(b)$ are more than $2+2rR$ far apart in $\cone{X}$ as $\pi\circ \sigma_i$ is a $(A_1,A_1)$--quasi-geodesic and $L_2 /(4A_1) - A_1 >\theta > 2 + 2rR$. 
However, this contradicts that $\gamma(a), \gamma(b) \in \mc{N}_{rR}(P_i)$. Therefore, we must have that $P_i \neq P_{i+1}$.
\end{proof}

A simple, but central application of Lemma \ref{lem:alphas are close and sigmas have bounded projections} is establishing that the $\sigma_i$ and $\alpha_i$ are Morse quasi-geodesics. Since the $\sigma_i$ have bounded projections, Corollary \ref{cor: local qg+bounded projection=Morse} implies, for sufficient local scale, they are Morse quasi-geodesics, regardless of what the peripheral subsets are. When the peripheral subsets have the Morse local-to-global property, then there is a local scale so that the $\alpha_i$ will also be Morse quasi-geodesics by virtue of the fact that they run close to the peripheral $P_i$.

\begin{cor}\label{cor: sigma_i are Morse q.g}
For all $B > 2\theta$, there exists constants $L_1 = L_1(\lambda,\epsilon,B)$, $A_2 = A_2(\lambda,\epsilon,B)$, and Morse  gauge $M = M(\lambda,\epsilon,B)$ so that if $\decomposition$ is the $B$--relevant decomposition of an $(L_2;\lambda,\epsilon)$--quasi-geodesic, then each $\sigma_i$ is an $(M;A_1,A_1)$--Morse quasi-geodesic. 
\end{cor}

\begin{proof}
Apply Corollary \ref{cor: local qg+bounded projection=Morse} with $D = \lambda(4B+\epsilon)$.
\end{proof}

\begin{cor}\label{cor: alphas are Morse q.g}
Suppose every $P \in \mc{P}$ has the $\Phi$--Morse local-to-global property. For every Morse gauge $M$, there are constants $L_3 \geq 12 \theta$, $k$, $c$ and Morse gauge $N$ (all depending only on $\lambda$, $\epsilon$, and $M$)  so that for every $B > 2\theta$, if $\decomposition$ is the $B$--relevant decomposition of  an $(L_3; M; \lambda, \epsilon)$--local quasi-geodesic  $\gamma$, then each  $\alpha_i$ is an $(N;k,c)$--Morse quasi-geodesics. 
\end{cor}
\begin{proof}

Recall $\alpha_i \subseteq \mc{N}_{rR}(P_i)$ by  Proposition \ref{lem:alphas are close and sigmas have bounded projections}. Since $P_i$ has the $\Phi$--Morse local-to-global property, $\mc{N}_{rR}(P_i)$, equipped with the induced metric from $X$, has the $\Psi$--Morse local-to-global property where $\Psi$ ultimately depends only on $\Phi$, $\lambda$, and $\epsilon$. Thus, there exist constants $L_3$, $\lambda_0$, $\epsilon_0$ and Morse gauge $M_0$ so that $\Psi(M,\lambda,\epsilon) = (L_3,M_0,\lambda_0,\epsilon_0)$ and if $\gamma$ is an $(L_3;M;\lambda, \epsilon)$--local Morse quasi-geodesic, then each $\alpha_i$ is an $(M_0;\lambda_0,\epsilon_0)$--Morse quasi-geodesic in $\mc{N}_{rR}(P_i)$. The distance formula (Theorem \ref{thm:distance_formula_for_rel_hyp}) and Lemma \ref{lem: quasi-geodesics remain in a nbhd} then imply each $\alpha_i$ is an $(N;k,c)$--Morse quasi-geodesic in $X$ where $N$, $k$, and $c$ ultimately depend only on $\lambda$, $\epsilon$, and $M$.
\end{proof}

\begin{rem}[Uniformity of $\Phi$]
Corollary \ref{cor: alphas are Morse q.g} is the key place in the proof of Theorem \ref{thm:local_to_global_and_relative_hyperbolicity} where the Morse local-to-global property of the peripheral subsets is used. The requirement that the Morse local-to-global property is controlled by the same function $\Phi$ for each element of $\mc{P}$ is precisely to ensure that each $\alpha_i$ is a Morse quasi-geodesic whose parameters do not depend on which peripherals are relevant.
\end{rem}

\subsection{Ordered projections of relevant peripherals}\label{subsec:ordering_of_peripherals}

The goal of this section is to prove that if $ \decomposition$ is the $B$--relevant decomposition of a local Morse quasi-geodesic $\gamma$, then the distance between the projection of the endpoints of $\gamma$ to any relevant peripheral $P_i$ is bounded below the by a linear function of $B$.  We first show each $\sigma_i$ has uniformly bounded projection (independent of $B$) onto the relevant peripherals immediately before and after $\sigma_i$ (Proposition \ref{prop: adjacent projection of sigmas are bounded}). Next, we establish that the relevant peripherals are ``linearly ordered" along the local quasi-geodesic, that is, if $1\leq i<j<k \leq n$, the projection of $P_i$ onto $P_k$ is uniformly close to the projection of $P_j$ onto $P_k$ (Proposition \ref{prop: Behrstock inequality for local qg}). These two facts mean that the projection of $P_1$ and $P_n$ to each other relevant peripheral $P_i$ coarsely equal the endpoints of $\alpha_i$. Since each $\alpha_i$ is a quasi-geodesic with parametrized length at least $B$, the distance between $\pi_{P_i}(P_1)$ and $\pi_{P_i}(P_n)$ will be bounded below by a linear function of $B$ for each other relevant peripheral (Corollary \ref{cor: the distance between projections of domains depends on B}). Finally, we show that the projection of the endpoints of $\gamma$ onto any $P_i$ must agree with the projection  of $P_1$ and $P_n$ onto $P_n$ (Corollary \ref{cor:endpoints_project_to_peripherals}).

We begin with proving the projection of $\sigma_i$ onto $P_i$ and $P_{i+1}$ is bounded independent of $B$.

\begin{prop}\label{prop: adjacent projection of sigmas are bounded}
There exists a constant  $K_0= K_0(\lambda, \epsilon)$  so that the following holds. For every $B > 2\theta$  there exists $L_4 =L_4(\lambda,\epsilon, B)$ so that if $\sigma_0 \ast \alpha_1 \ast \sigma_1 \ast \cdots \alpha_n \ast \sigma_n$ is the $B$--relevant decomposition of an $(L_4;\lambda,\epsilon)$--local quasi-geodesic $\gamma$, then $\mathrm{diam} (\pi_{P_{i}}(\sigma_{i-1})) \leq K_0$ and $\mathrm{diam} (\pi_{P_{i}}(\sigma_{i})) \leq K_0$ for each $1\leq i \leq n$.
\end{prop}

We first show that the conclusion of Proposition \ref{prop: adjacent projection of sigmas are bounded} holds for the portion of the  $\sigma_i$ that is close to its endpoints.

\begin{lem}\label{lem: the sigma coarsely project near at the endpoints}
There exists $C = C (\lambda, \epsilon)$ such that the following holds. Let $L \geq 12\theta$, $B >2\theta$, and $\sigma_0\ast \alpha_1 \ast \sigma_1 \ast \dots \alpha_n \ast \sigma_n$ be the $B$--relevant decomposition of an $(L;\lambda, \epsilon)$--local quasi-geodesic $\gamma$. If $[a_i,b_i]$ is the domain of $\sigma_i$ and $s_i, t_i \in [a_i,b_i]$ so that $|a_i - s_i| \leq L/4$ and $|b_i-t_i| \leq L/4$, then  \[ \diam\bigl(\pi_{P_i} (\gamma\vert_{[a_i,s_i]}) \bigr) \leq C \text{ and }  \diam\bigl(\pi_{P_{i+1}}(\gamma\vert_{[t_i,b_i]}) \bigr)\leq C\]
whenever $P_i$ or $P_{i+1}$ exists.
\end{lem}

\begin{proof}
We shall only prove the case of $\gamma \vert_{[a_i,s_i]}$ as the other case is analogous.
Let $C = 2\lambda(2\theta + 2R+\epsilon+1)$.
There exists $[v_1,v_2] \subseteq [a_i, s_i]$ so that $d\bigl(\pi_{P_i}(\gamma(v_1)), \pi_{P_i}(\gamma(v_2)\bigr)\geq C/2$.
By Lemma \ref{lem: BGI for rel hyp} 
there is $[u_1,u_2] \subseteq [v_1, v_2]$ such that \[d\bigl( \gamma(u_1), \pi_{P_i}(\gamma(v_1)) \bigr) \leq R \text{ and } d\bigl( \gamma(u_2), \pi_{P_i}(\gamma(v_2)) \bigr) \leq R.\]

Since $\gamma\vert_{[a_i, s_i]}$ is a $(\lambda,\epsilon)$--quasi-geodesic, the choice of $C$ gives us $2\theta +1 \leq |u_1 - u_2| \leq L/4$. Thus, there exists $\tau \in (u_1,u_2) \cap \deep(\gamma;P_{i})$. By Corollary  \ref{cor: deep points are a disjoint union}, $[a_i,\tau] \subseteq \deep(\gamma;P_{i})$, but this is a contradiction to the fact that $\alpha_{i}$ is $B$--relevant for $P_{i}$. Thus, we must have $\mathrm{diam}\bigl(\pi_P(\gamma\vert_{{[a_i, s_i]}}) \bigr)\leq C.$
\end{proof}

To prove the general case of Proposition \ref{prop: adjacent projection of sigmas are bounded}, we employ the following lemma of Sisto.

\begin{lem}[{\cite[Lemma 4.14]{Sisto_metric_rel_hyp}}]\label{lem: BGI in cone-off}
There exists $C_1 >0$  so that for every $\lambda'\geq 1$ and $\epsilon' \geq 0$ there exists $C_2 = C_2(\lambda',\epsilon')$ such that if $\widehat{\gamma}$ is a $(\lambda',\epsilon')$--quasi-geodesic of $\widehat{X}$ connecting $x,y$ and $\widehat{\gamma}$ does not intersect the $C_2$--neighborhood of $P \in \mc{P}$ in $\cone{X}$, then $d\bigl(\pi_P(x), \pi_P(y)\bigr) \leq C_1.$
\end{lem}

In light of Lemmas \ref{lem: the sigma coarsely project near at the endpoints} and \ref{lem: BGI in cone-off}, Proposition \ref{prop: adjacent projection of sigmas are bounded} will follow if we can show that the portion of $\sigma_i$ far from the endpoints satisfies the hypotheses of Lemma \ref{lem: BGI in cone-off}.

\begin{proof}[Proof of Proposition \ref{prop: adjacent projection of sigmas are bounded}]
We only show that the projection of $\sigma_{i-1}$ onto $P_{i}$ is uniformly bounded as the other case is analogous. Let $c_P$ be the cone point of $P_{i}$.
Let $L_1 = L_1(\lambda,\epsilon,B)$, $\lambda' = \lambda'(\lambda,\epsilon)$, $\epsilon'=\epsilon'(\lambda,\epsilon)$, and $A_1= A_1(\lambda, \epsilon, B )$ be the constants from Proposition \ref{prop:bounded_projections_imply_quasi-geodesic} such that if $\eta$ is an $(L_1;\lambda,\epsilon)$--local quasi-geodesic  of $X$ with $\lambda(4B+\epsilon)$--bounded projections, then $\pi \circ \eta$ is a unparametrized $(\lambda', \epsilon')$--quasi-geodesic and a parametrized $(A_1,A_1)$--quasi-geodesic in $\cone{X}$. Let $C_1, C_2$ be the constants of Lemma \ref{lem: BGI in cone-off} applied to $\lambda', \epsilon'$. Assume $L_4 \geq 8 (A_1 (C_2 + rR +1) + A_1 + L_1 + 12\theta)$.

Let $[a,b]$ be the domain of $\sigma_{i-1}$. If $|a-b| \leq L_4/4$, then the result follows from Lemma \ref{lem: the sigma coarsely project near at the endpoints}. So, suppose there exists $c \in (a,b)$ so that $|c-b| = L_4/4$. Since Lemma \ref{lem: the sigma coarsely project near at the endpoints} bounded the diameter of $\pi_{P_i}\bigl( \sigma_{i-1}([c,b])\bigr)$, our goal is to bound the diameter of $\pi_{P_{i}}\bigl(\sigma_{i-1}([a,c]) \bigr)$.
 By Lemma \ref{lem:alphas are close and sigmas have bounded projections}, $\sigma_{i-1}(b)$ lies in the $rR$--neighborhood of $P_{i}$ in $X$, so  ${\pi} \circ \sigma_{i-1}(b)$ has distance at most $rR+1$ from $c_P$ in $\widehat{X}$. 
Since $\pi \circ \sigma_{i-1}$ is a parametrized $(A_1, A_1)$--quasi-geodesic, the choice of $L_4$ yields 
\[ ({\pi}\circ \sigma_{i-1})^{-1} \bigl( \mc{N}_{C_2}(c_P) \bigr) \subseteq (c,b].\] 
In particular, $\pi \circ \sigma_{i-1} \vert_{[a,c]}$ is an unparametrized $(\lambda', \epsilon')$--quasi-geodesic of $\widehat{X}$ that does not intersect the $(C_2+1)$--neighborhood of $c_P$. Lemma \ref{lem: BGI in cone-off} now yields the result. 
\end{proof}

 Corollary \ref{cor: alphas are Morse q.g} established that the relevant subsegments of the decomposition are global Morse quasi-geodesics when the peripherals have the Morse local-to-global property. Since the diameter of the projection of $\sigma_{i-1}$ and $\sigma_i$ to the relevant peripheral $P_i$ are bounded independent of $B$, we can use the quasi-geodesic $\alpha_i$ to produce a linear lower bound on the distance between $\pi_{P_i}(\sigma_{i-1})$ and $\pi_{P_i}(\sigma_i)$ in terms of $B$.
 
\begin{cor}\label{cor: relevant implies large projection}
Suppose each $P\in\mc{P}$ has the  $\Phi$--Morse local-to-global property. Let $B > 2 \theta$ and $M$ be a Morse gauge. There exist constants $L_5 = L_5(\lambda,\epsilon, B,M)$ and $A_3 = A_3 (\lambda,\epsilon,M)$  so that if $\sigma_0 \ast \alpha_1 \ast \dots  \ast\alpha_n \ast \sigma_n$ is the $B$--relevant decomposition of an $(L_5;M;\lambda,\epsilon)$--local Morse quasi-geodesic and $P_i$ is the relevant peripheral for $\alpha_i$, then
\[d\bigl( \pi_{P_i}(\sigma_{i-1}), \pi_{P_i}(\sigma_{i}) \bigr) \geq \frac{1}{A_3}B - A_3.\]
\end{cor}

\begin{proof}
Let $L_5 = \max\{L_3,L_4\}+1$ where $L_3 = L_3(\lambda,\epsilon, B,M)$ and $L_4 = L_4(\lambda,\epsilon, B,M)$ are the constants from Corollary \ref{cor: alphas are Morse q.g} and Proposition \ref{prop:bounded_projections_imply_quasi-geodesic} respectively. Let $k = k(\lambda,\epsilon, M)$ and $c = c(\lambda,\epsilon, M)$ be the constants from Corollary \ref{cor: alphas are Morse q.g} so that each $\alpha_i$ is a $(k,c)$--quasi-geodesic. If $x_i$ and $y_i$ are the left and right endpoints of $\alpha_i$ respectively, then $x_i,y_i \in \mc{N}_{rR}(P_i)$ by Lemma \ref{lem:alphas are close and sigmas have bounded projections}. Let $K_0 = K_0(\lambda,\epsilon)$ be the constant from Proposition \ref{prop: adjacent projection of sigmas are bounded} so that the diameter of $\pi_{P_i}(\sigma_{i-1})$ and $\pi_{P_i}(\sigma_i)$ is at most $K_0$. Recall,   $d(x_i,y_i) \geq B/k - c$, since $\alpha_i$ has parametrized length at least $B$ (Definition \ref{defn: relevant decomposition}). Thus, we establish the corollary with the following calculation:
\begin{align*}
\frac{1}{k}B - c \leq& d(x_i,y_i)\\
\leq& d\left(\pi_{P_i}(x_i),\pi_{P_i}(y_i)\right) + 2rR+2\\
\leq& d\left( \pi_{P_i}(\sigma_{i-1}),\pi_{P_i}(\sigma_i) \right) +2rR +2 +2K_0.\qedhere
\end{align*}
\end{proof}

We now pause to define a local scale $\Lambda$ larger than any of the local scales required by any of the results in this section up until this point. This will allow all of the proceeding results to apply to any $(\Lambda;M;\lambda,\epsilon)$--local Morse quasi-geodesic in $X$. This scale depends on the relevancy constant $B$ from the decomposition of the local quasi-geodesic, the Morse gauge $M$, our fixed quasi-geodesic constants $\lambda$ and $\epsilon$, and the function $\Phi$ governing the Morse local-to-global property of the peripherals.

\begin{defn}[Largest local scale given $B$, $M$, and $\Phi$]
Given $B > 2\theta$, a Morse gauge $M$, and a Morse local-to-global function $\Phi$, let $L_1 = (\lambda,\epsilon,D)$ be constant from Proposition \ref{prop:bounded_projections_imply_quasi-geodesic} for $D = \lambda(4B + \epsilon)$ and  $L_2$, $L_3$, $L_4$, and $L_5$ be the constants depending on $\lambda$, $\epsilon$, $B$, and $M$ from Lemma \ref{lem:adjacent_relevant_peripherals_are_distinct}, Corollary \ref{cor: alphas are Morse q.g}, Proposition \ref{prop: adjacent projection of sigmas are bounded}, and Corollary \ref{cor: relevant implies large projection}. Define $\Lambda_\Phi(\lambda,\epsilon,B,M) = \max\{L_1,L_2,L_3,L_4,L_5,12 \theta\}+1$.
\end{defn}

We now show the second main tool of this section, the linear ordering of the relevant peripheral subsets along the local Morse quasi-geodesic.

\begin{prop}\label{prop: Behrstock inequality for local qg} Suppose every $P \in \mc{P}$ has the $\Phi$--Morse local-to-global property. For each Morse gauge $M$, there exist constants $K_1 = K_1(\lambda,\epsilon)$ and $B_1 = B_1(\lambda,\epsilon,M)$ such that the following holds for for any $B \geq B_1$. Let $\Lambda = \Lambda_\Phi (\lambda,\epsilon,B,M)$ and $\sigma_0 \ast \alpha_1 \ast \sigma_1 \ast \cdots \ast \alpha_n \ast \sigma_n$ be the $B$--relevant decomposition of a $(\Lambda;M;\lambda, \epsilon)$--local Morse quasi-geodesic $\gamma$. If  $\{P_1,\dots,P_n\}$ are the relevant peripherals for $\gamma$, then for all $1\leq i<j<k \leq n$ \[ \pi_{P_k}(P_i) \subseteq \mc{N}_{K_1}\left(\pi_{P_k}(P_j)\right) \text{ and } \pi_{P_i}(P_k) \subseteq \mc{N}_{K_1}\left(\pi_{P_i}(P_j)\right).\]
\end{prop}

\begin{proof} We only show the first containment as the second follows analogously.

Assume $B > 2\theta$. By Lemma \ref{lem:adjacent_relevant_peripherals_are_distinct}, we know $P_{i-1} \neq P_i$ and $P_{i+1} \neq P_i$. 
By Proposition \ref{prop: adjacent projection of sigmas are bounded}, there exists $K_0 = K_0(\lambda, \epsilon)$ so that \[\diam(\pi_{P_i}(\sigma_{i-1}))\leq K_0 \text{ and } \diam(\pi_{P_i}(\sigma_{i})) \leq K_0.\] 
By Lemma \ref{lem:alphas are close and sigmas have bounded projections}, we have \[\sigma_{i-1} \cap \mc{N}_{rR}(P_{i}) \neq \emptyset \text{ and } \sigma_{i} \cap \mc{N}_{rR}(P_{i}) \neq \emptyset.\] Using the fact that $\pi_{P_i}$ is $(\mu, \mu)$--coarsely Lipschitz and Corollary \ref{cor: relevant implies large projection} we obtain:
 \begin{align*}
     d\bigl( \pi_{P_{i}}(P_{i-1}), \pi_{P_i}(P_{i+1}) \bigr)  &\geq d \bigl( \pi_{P_i}(\sigma_{i-1}),\pi_{P_i}(\sigma_{i+1})\bigr) - 2\mu (rR + 1) - 2K_0 \\
     &\geq \frac{1}{A_3}B - A_3 -2\mu (rR + 1)-2K_0 \tag{$\ast$} \label{eq:adjacent_sigmas_have_big_projection}
 \end{align*}
where $A_3 = A_3 (\lambda,\epsilon,M)$ is as in Corollary \ref{cor: relevant implies large projection}.

Let $K_1 = {\mu} R + 4{\mu}$, and suppose that $B_1 > 2\theta$ is chosen large enough so that for $B \geq B_1$, (\ref{eq:adjacent_sigmas_have_big_projection}) yields
\[ d\bigl( \pi_{P_{i}}(P_{i-1}), \pi_{P_i}(P_{i+1}) \bigr) \geq Q + 2K_1.\]

We  now proceed by induction on $n$, the number of terms in the $B$--relevant decomposition. 

Assume $n=3$, so \[\gamma  = \sigma_0 \ast \alpha_1 \ast \sigma_1 \ast \alpha_2 \ast \sigma_2 \ast \alpha_3 \ast \sigma_3.\] Let $x \in P_{1}$ and $y = \pi_{P_3}(x)$. Since $B \geq B_1$, (\ref{eq:adjacent_sigmas_have_big_projection}) implies $ d\bigl( \pi_{P_2}(x), \pi_{P_2}(y) \bigr) \geq Q +2K_1$. Thus, if $\beta$ is a geodesic from $x$ to $y$, then $\beta \cap \mc{N}_{R}(P_2)$ (Lemma \ref{lem: BGI for rel hyp}). By the properties of $\pi_{P_3}$ (Lemma \ref{lem:projections_lemmas}), \[\diam\bigl(\pi_{P_3}(\beta) \bigr) \leq \mu \text{ and }d\bigl( \pi_{P_3}(\beta), \pi_{P_3}(P_2) \bigr) \leq {\mu} R + {\mu}.\] This implies $\pi_{P_3}(P_1) \subseteq \mc{N}_{{\mu} R+4{\mu}}(\pi_{P_3}(P_2)) = \mc{N}_{K_1}(\pi_{P_3}(P_2))$.

Assume the proposition is true of all $n'< n$ and consider \[\gamma = \sigma_0 \ast \alpha_1\ast \sigma_1 \ast \dots \ast \alpha_n \ast \sigma_n.\] 
Let $1\leq i<j<k \leq n$. If $k<n$ or $1<i$, then the induction hypothesis applies to the local quasi-geodesic $\sigma_{i-1} \ast \alpha_{i}\ast \sigma_{i} \ast \dots \alpha_k \ast \sigma_k$ and we are finished. Thus, we can assume $k=n$ and $i=1$. 
Now, the induction hypothesis applies to $\sigma_{1} \ast \alpha_2\ast \sigma_2 \ast \dots \alpha_n \ast \sigma_n$ and $\sigma_{0} \ast \alpha_1\ast \sigma_1 \ast \dots \alpha_{n-1} \ast \sigma_{n-1}$. In particular, $\pi_{P_j}(P_n) \subseteq \mc{N}_{K_1}(\pi_{P_j}(P_{j+1}))$ and $\pi_{P_j}(P_1) \subseteq \mc{N}_{K_1}(\pi_{P_j}(P_{j-1}))$ for $2\leq j \leq n-1$.
By the choice of $B_1$ and (\ref{eq:adjacent_sigmas_have_big_projection}), we have \[d\bigl( \pi_{P_j}(P_1), \pi_{P_j}(P_n) \bigr) \geq Q.\] As in the base case, this implies every geodesic connecting a point $x \in P_1$ to $\pi_{P_n}(x)$ passes through $\mc{N}_R(P_j)$, yielding $\pi_{P_n}(P_1) \subseteq \mc{N}_{{\mu} R+4{\mu}}(\pi_{P_n}(P_j))=\mc{N}_{K_1}(\pi_{P_n}(P_j))$ by Lemma \ref{lem:projections_lemmas}.
\end{proof}

The ordering of the peripherals has two immediate consequences when coupled with Proposition \ref{prop: adjacent projection of sigmas are bounded}. Each relevant peripheral is in fact distinct and the distance between the projection of two relevant peripherals onto a peripheral between them is bounded below by a linear function of $B$.

\begin{cor}\label{cor: the distance between projections of domains depends on B}
Suppose each $P\in\mc{P}$ has the $\Phi$--Morse local-to-global property and let $M$ be a Morse gauge. There exist  $B_2 = B_2(\lambda,\epsilon,M)$ and $A_4 = A_4(\lambda, \epsilon, M)$ so that for all $B \geq B_2$, if $\Lambda = \Lambda_\Phi (\lambda,\epsilon,B,M)$  and $\{P_1,\dots,P_n\}$ are the $B$--relevant peripherals of a $(\Lambda;M;\lambda,\epsilon)$--local Morse quasi-geodesic, then
\begin{itemize}
    \item $P_i = P_j$ if and only if $i = j$;
    \item $d\bigl( \pi_{P_j}(P_i), \pi_{P_j}(P_k) \bigr) \geq \frac{1}{A_4} B - A_4$ whenever $1\leq i <j<k \leq n$.
\end{itemize}
\end{cor}
\begin{proof}
By Corollary \ref{cor: relevant implies large projection}, there exists $B_2 = B_2(\lambda,\epsilon,M)$ so that $B_2$ is larger than the constant $B_1 = B_1(\lambda,\epsilon,M)$ from Proposition \ref{prop: Behrstock inequality for local qg} and $\diam(P_i) \geq \mu+2K_1+10$ where $K_1 = K_1(\lambda,\epsilon)$ is the constant from Proposition \ref{prop: Behrstock inequality for local qg}. 
In Lemma \ref{lem:adjacent_relevant_peripherals_are_distinct}, we showed $P_{i-1} \neq P_i$ and $P_{i+1} \neq P_i$, so  $\pi_{P_i}(P_{i-1})$ and $\pi_{P_i}(P_{i+1})$ both have diameter at most $\mu$ (Lemma \ref{lem:projections_lemmas}). 
Proposition \ref{prop: Behrstock inequality for local qg} implies $\pi_{P_i}(P_j)$  is contained in the $K_1$--neighborhood of $\pi_{P_i}(P_{i-1})$ or $\pi_{P_i}(P_{i+1})$. Thus, $\diam\bigl( \pi_{P_i}(P_j) \bigr) \leq \mu +2 K_1$. 
If $P_i = P_j$ for $i \neq j$, then $\pi_{P_i}(P_j) = P_i$. But, this would be a contradiction as $\diam(P_i) \geq \mu + 2K_1 +10$.

The second item follows by combining Corollary \ref{cor: relevant implies large projection} with Proposition \ref{prop: Behrstock inequality for local qg} and the fact that $\sigma_{j-1} \cap \mc{N}_{rR}(P_{j-1}) \neq \emptyset$ and $\sigma_j \cap \mc{N}_{rR}(P_{j+1})\neq \emptyset $.
\end{proof}

We now use the results of this section to show that the projection of the endpoints of a local quasi-geodesic to a relevant peripheral $P_i$ are coarsely equal to the projection of the first and last relevant peripheral on $P_i$. By Corollary \ref{cor: the distance between projections of domains depends on B}, this implies the distance between the projection of the endpoints is bounded below by a linear function of the relevancy constant $B$.

\begin{cor}\label{cor:endpoints_project_to_peripherals}
Suppose each $P\in\mc{P}$ has the $\Phi$--Morse local-to-global property and let $M$ be a Morse gauge. There exists constants $B_3=B_3(\lambda,\epsilon,M)$ and $K_2 = K_2(\lambda,\epsilon)$ so that for all $B \geq B_3$, if $\Lambda =\Lambda_\Phi (\lambda,\epsilon,B,M)$ and $\sigma_0 \ast \alpha_1 \ast \dots \ast \alpha_n \ast \sigma_n$ is the $B$--relevant decomposition of a $(\Lambda;M;\lambda,\epsilon)$--local Morse quasi-geodesic, then
\[\pi_{P_i}(\sigma_0) \subseteq \mc{N}_{K_2}\bigl(\pi_{P_i}(P_1)\bigr) \text{ and } \pi_{P_i}(\sigma_n) \subseteq \mc{N}_{K_2}\bigl(\pi_{P_i}(P_n)\bigr) \text{ for all $i \in \{1,\dots,n\}$}.\]
\end{cor}
\begin{proof}
Let $K_1(\lambda,\epsilon)$ and $B_1(\lambda,\epsilon,M)$ be the constant from Proposition \ref{prop: Behrstock inequality for local qg}. We first show that each $\sigma_i$ has a projection onto each $P_j$ bounded independently of $B$.

\begin{claim}\label{claim: sigmas have bounded projection on all relevant domains}
There exists $B'= B'(\lambda,\epsilon,M)$ and $K' = K'(\lambda,\epsilon)$ so that if $B \geq B'$, then $\diam\bigl(\pi_{P_j}(\sigma_i)\bigr) \leq K'$ for all $i \in \{0,\dots,n\}$ and $j \in \{1,\dots,n\}$.
\end{claim}
\begin{subproof}
Without loss of generality, assume $i<j$. By Proposition \ref{prop: adjacent projection of sigmas are bounded}, we can assume $j \neq i+1$.  Let $x \in \sigma_i$ and $y = \pi_{P_{i+1}}(x)$ and assume $B \geq B_1$. We first show that  $d\bigl( \pi_{P_j}(x),\pi_{P_j}(y) \bigr) \leq Q$ for large enough $B$. 

If $d\bigl( \pi_{P_j}(x),\pi_{P_j}(y) \bigr) \geq Q$, then every  geodesic in $X$ from $x$ to $y$ passes through the $R$--neighborhood of $P_{j}$ (Lemma \ref{lem: BGI for rel hyp}).
 By Lemma \ref{lem:projections_lemmas}, if $\eta$ is a geodesic in $X$ from $x$ to $\pi_{P_{i+1}}(x) = y$, then $\diam(\pi_{P_{i+1}}(\eta) \bigr) \leq \mu$. Thus, if $d\bigl( \pi_{P_j}(x),\pi_{P_j}(y) \bigr) \geq Q$, then  $d\bigl(y, \pi_{P_{i+1}}(P_j) \bigr)\leq 2\mu +\mu R$.
 Now, $y =\pi_{P_{i+1}}(x) \in \pi_{P_{i+1}}(\sigma_i)$ and  $\pi_{P_{i+1}}(\sigma_i)$ is contained in the $(K_0+\mu rR + \mu)$--neighborhood of $\pi_{P_{i+1}}(P_i)$ by  Proposition \ref{prop: adjacent projection of sigmas are bounded}. Therefore, if $d\bigl( \pi_{P_j}(x),\pi_{P_j}(y) \bigr) \geq Q$, then $d\bigl( \pi_{P_{i+1}}(P_i),\pi_{P_{i+1}}(P_j) \bigr) \leq K_0 + 5\mu r R$.
 However, Corollary \ref{cor: the distance between projections of domains depends on B} provides  a constant $B_2=B_2(\lambda,\epsilon,M)$, so that for all $B \geq B_2$, the distance between $\pi_{P_{i+1}}(P_i)$ and $\pi_{P_{i+1}}(P_j)$ is bounded below by a linear function (also determined by $\lambda,\epsilon,M$) of $B$. Hence, for sufficiently large $B$, we would have $d\bigl( \pi_{P_{i+1}}(P_i),\pi_{P_{i+1}}(P_j) \bigr) > K_0 + 5\mu r R$. This contradiction implies $d\bigl( \pi_{P_j}(x),\pi_{P_j}(y) \bigr) \leq Q$ for all $B \geq B'$ where $B' = B'(\lambda,\epsilon,M)$

Now for any $x,z \in \sigma_i$, we can use the above plus the triangle inequality and the fact that $\diam\bigl(\pi_{P_{j}}(P_{i+1}))\bigr) \leq \mu$ to obtain $d\bigl( \pi_j(x),\pi_j(z) \bigr) \leq 2 Q + \mu$.
\end{subproof}

 Now, let $B_3 = B_1 + B'$  and $K_2 = K' + \mu rR +\mu$ where $B'$ and $K'$ are as in Claim \ref{claim: sigmas have bounded projection on all relevant domains}. Since, $\pi_{P_i}$ is $(\mu,\mu)$--coarsely Lipschitz and $\sigma_0$, $\sigma_n$ intersect the $rR$--neighborhoods of $P_1$ and $P_n$ respectively (Lemma \ref{lem:alphas are close and sigmas have bounded projections}),  Claim \ref{claim: sigmas have bounded projection on all relevant domains} implies \[\pi_{P_i}(\sigma_0) \subseteq \mc{N}_{K_2}\bigl(\pi_{P_i}(P_1)\bigr) \text{ and } \pi_{P_i}(\sigma_n) \subseteq \mc{N}_{K_2}\bigl(\pi_{P_i}(P_n)\bigr).\qedhere\]
\end{proof}

\subsection{Proof of the Morse local-to-global property}\label{subsec:proof_or_rel_hyp}

We now give the proof of Theorem \ref{thm:local_to_global_and_relative_hyperbolicity}. The heart of the proof is Lemma \ref{lem: BCP for local things} below, which uses the tools developed throughout this section to show that, for sufficiently large local scales, a local Morse quasi-geodesic will be uniformly close to any quasi-geodesic between its endpoints.

\begin{lem}\label{lem: BCP for local things}
Suppose each $P \in \mc{P}$ has the $\Phi$--Morse local-to-global property and let $M$ be a Morse gauge. There exists $B_4 = B_4(\lambda,\epsilon,M)$ so that for each $B \geq B_4$, $k \geq 1$, and $c \geq 0$, there exists $C = C(\lambda,\epsilon,k,c,B,M)$ such that if $\Lambda = \Lambda_\Phi (\lambda,\epsilon,B,M)$ and $\gamma$ is a $(\Lambda; M; \lambda, \epsilon)$--local Morse quasi-geodesic, then any $(k, c)$--quasi-geodesic with the same endpoints as $\gamma$ has Hausdorff distance  at most $C$ from $\gamma$.
\end{lem}

\begin{proof}
 Let $B >2\theta$ be larger than the constants $B_1$, $B_2$, and $B_3$ from Proposition \ref{prop: Behrstock inequality for local qg}, Corollary \ref{cor: the distance between projections of domains depends on B}, and Corollary \ref{cor:endpoints_project_to_peripherals}. Let $\Lambda = \Lambda_\Phi (\lambda,\epsilon,B,M)$ and $\decomposition$ be the $B$--relevant decomposition of a $(\Lambda;M;\lambda,\epsilon)$--local Morse quasi-geodesic $\gamma$. Let $x$ be the left endpoint of $\gamma$ and $y$ the right endpoint. Let $\eta$ be a $(k,c)$--quasi-geodesic connecting $x$ and $y$.
 
 For each $i$, let $p_i$ be the left endpoints of $\alpha_i$ and $q_i$ be the right endpoint. We first show that, after increasing $B$, $\eta$ passes close to each of the $p_i$ and $q_i$.
 \begin{claim} \label{claim:close to geodesic}
 There exist constants $B_4 = B_4(\lambda,\epsilon,M)$ and $C_1=C_1(\lambda,\epsilon,k,c)$ so that if  $B \geq B_4$,  then  there exists $x_i,y_i \in \eta$ with $d(x_i,p_i) \leq C_1$ and $d(y_i,q_i) \leq C_1$ for each $i$.
 \end{claim}

 \begin{subproof}
 Let $K_1 = K_1(\lambda,\epsilon)$ and $K_2 = K_2(\lambda,\epsilon)$ be the constants from Proposition \ref{prop: Behrstock inequality for local qg} and Corollary \ref{cor:endpoints_project_to_peripherals} so that
 \[\pi_{P_i}(P_1) \subseteq \mc{N}_{K_1}\bigl(\pi_{P_i}(P_{i-1})\bigr); \quad \pi_{P_i}(P_n) \subseteq \mc{N}_{K_1}\bigl(\pi_{P_i}(P_{i+1})\bigr)\]
  and 
\[\pi_{P_i}(x) \in \mc{N}_{K_2}\bigl(\pi_{P_i}(P_1)\bigr) ; \quad \pi_{P_i}(y) \in \mc{N}_{K_2}\bigl(\pi_{P_i}(P_n)\bigr) \tag{$\ast$} \label{eq:projection of endpoints}\]
for each $1\leq i \leq n$.
By Corollary \ref{cor: the distance between projections of domains depends on B}, there exists $B_4 = B_4(\lambda,\epsilon,M)$ such that if $B \geq B_4$, then $d(\pi_{P_i}(x), d(\pi_{P_i}(y)) \geq Q$ for each $i\in\{1,\dots,n\}$. By Lemma \ref{lem: BGI for rel hyp}, there exist $x_i,y_i \in \eta$ and a constant $R_1 =R_1(k,c)$ so that \[d\bigl(x_i,\pi_{P_i}(x) \bigr) \leq R_1 \text{  and } d\bigl(y_i,\pi_{P_i}(y) \bigr) \leq R_1 \tag{$\ast \ast$} \label{eq:BGI}.\] 

We now argue that $p_i$ is close to $x_i$. The case of $q_i$ and $y_i$ is analogous.
First assume $i \neq 1$.
Since $\sigma_{i-1} \cap \mc{N}_{rR}(P_{i-1}) \neq \emptyset$, we have $d \big( \pi_{P_i}(\sigma_{i-1}), \pi_{P_i}(P_{i-1}) \bigr) \leq \mu r R +\mu$.  By Proposition \ref{prop: adjacent projection of sigmas are bounded}, $\diam\bigl(\pi_{P_i}(\sigma_{i-1}) \bigr) \leq K_0$ where $K_0 = K_0(\lambda,\epsilon)$. Thus, $d\bigl( p_i , \pi_{P_i}(P_{i-1}) \bigr) \leq rR+1+K_0+\mu r R+\mu$ as $p_i \in \mc{N}_{rR}(P_i)$. Since $\diam\bigl(\pi_{P_i}(P_{i-1}) \bigr) \leq \mu$ (Lemma \ref{lem:projections_lemmas}), (\ref{eq:projection of endpoints}) and (\ref{eq:BGI}) imply that 
\begin{align*}
    d( p_i , x_i)  \leq& d\bigl( p_i, \pi_{P_i}(P_{i-1}) \bigr) + d\bigl( \pi_{P_i}(P_{i-1}), \pi_{P_i}(P_1) \bigr) + d\bigl( \pi_{P_i}(P_1), \pi_{P_i}(x) \bigr) + d\bigl( \pi_{P_i}(x), x_i \bigr) + 2\mu\\
    \leq& (rR+1+K_0+\mu r R+\mu) + K_1 + K_2 + R_1+2\mu.
\end{align*}
If $i =1$, then  $d\bigl(p_1, \pi_{P_1}(\sigma_0) \bigr) \leq rR +1$ implies 
  \[ d(p_1,x_1) \leq d\bigl( p_1, \pi_{P_1}(x) \bigr) + d\bigl( \pi_{P_1}(x), x_1 \bigr) \leq K_0 + rR+1 +R_1. \qedhere\]
  \end{subproof}
To complete the proof of Lemma \ref{lem: BCP for local things}, let $B \geq B_4$ where $B_4 = B_4(\lambda,\epsilon)$ is as in Claim \ref{claim:close to geodesic}.
By Corollaries \ref{cor: sigma_i are Morse q.g} and \ref{cor: alphas are Morse q.g}, there exists a Morse gauge $M' = M'(\lambda,\epsilon,B,M)$ so that each $\sigma_i$ and $\alpha_i$ is $M'$--Morse.  Thus, $\eta$ has Hausdorff distance at most $M'(k,c+2C_1)$ from $\gamma$ where $C_1 = C_1(\lambda,\epsilon,k,c)$ is as in Claim \ref{claim:close to geodesic}.
\end{proof}

\begin{proof}[Proof of Theorem \ref{thm:local_to_global_and_relative_hyperbolicity}]
Let $\lambda \geq 1$, $\epsilon \geq 0$, and $M$ be a Morse gauge. Recall, we want to show there exists $L>0$, $\lambda' \geq 1$, $\epsilon'\geq 0$, and Morse gauge $N$ so that every $(L;M;\lambda,\epsilon)$--local Morse quasi-geodesic in $X$ is a global $(N;\lambda',\epsilon')$--Morse quasi-geodesic.

Let $B = B_4+1$ where $B_4 = B_4(\lambda,\epsilon,M)$ is the constant from Lemma \ref{lem: BCP for local things}, then let $\Lambda = \Lambda_\Phi (\lambda,\epsilon,B,M)$ and $C$ be the constant from Lemma \ref{lem: BCP for local things} with $k=1$, $c=0$, and $B=B_4 +1$. Let $L= \lambda(3 C +\epsilon +2) +\Lambda$ and  $\gamma$ be an $(L;M;\lambda,\epsilon)$--Morse local quasi-geodesic. Lemma \ref{lem: BCP for local things}  shows that $\gamma$ satisfies the hypothesis of Lemma \ref{lem:local_close_to_global}. Thus, $\gamma$ is a global $(\lambda',\epsilon')$--quasi-geodesic where $\lambda'$ and $\epsilon'$ depends only on $\lambda$, $\epsilon$, and $M$. Since all subsegments $\gamma$ are also $(L;M;\lambda,\epsilon)$--local Morse quasi-geodesics,  Lemma \ref{lem: BCP for local things} also implies that $\gamma$ is $N$--Morse for some $N = N(\lambda,\epsilon,M)$.
\end{proof}

\bibliographystyle{alpha}
\bibliography{Tran}
\end{document}